\documentclass[twoside]{article} 

\usepackage{jmlr2e} 
\usepackage{algorithm}
\usepackage{algorithmic}
\usepackage{hyperref}
\usepackage{natbib}

\usepackage{enumitem}

\usepackage{amsmath,amssymb}
\usepackage{pgfplots}
\usepackage{color}
\usepackage{flushend}
\usepackage{multirow}
\usepackage{rotating}
\usepackage{appendix}
\usepackage{tikz,pgfplots}
\usepackage{subcaption}
\usepackage{bbm}

\input{mysymbol.sty}

\newcommand\myeq{\mathrel{\stackrel{\makebox[0pt]{\mbox{\normalfont\tiny a.s.}}}{=}}}

\newtheorem{assumption}{\hspace{0pt}\bf Assumption}

\begin{document}

\jmlrheading{1}{2000}{1-48}{4/00}{10/00}{}


\ShortHeadings{Stochastic Conditional Gradient Methods}{Mokhtari, Hassani, and Karbasi}
\firstpageno{1}

\title{Stochastic Conditional Gradient Methods: \\ From Convex Minimization to Submodular Maximization}

\author{\name Aryan Mokhtari  \email aryanm@mit.edu \\
       \addr Laboratory for Information and Decision Systems\\
       Massachusetts Institute of Technology  \\
       Cambridge, MA 02139
       \AND
       \name Hamed Hassani \email hassani@seas.upenn.edu  \\
       \addr Department of Electrical and Systems Engineering\\
       University of Pennsylvania\\
       Philadelphia, PA 19104
       \AND
       \name Amin Karbasi \email amin.karbasi@yale.edu \\
       \addr Department of Electrical Engineering and Computer Science\\
       Yale University\\
       New Haven, CT 06520
       }
\editor{}

\maketitle

\thispagestyle{empty}

\begin{abstract}
This paper considers stochastic optimization problems for a large class of objective functions, including convex and continuous submodular. Stochastic proximal gradient methods have been widely used to solve such problems; however, their applicability remains limited when the problem dimension is large and the projection onto a convex set is computationally costly. Instead, stochastic conditional gradient algorithms are proposed as an alternative solution which rely on (i) Approximating gradients via a simple averaging technique requiring a single stochastic gradient evaluation per iteration; (ii) Solving a linear program to compute the descent/ascent direction. The gradient averaging technique reduces the noise of gradient approximations as time progresses, and replacing projection step in proximal methods by a linear program lowers the computational complexity of each iteration. We show that under convexity and smoothness assumptions, our proposed stochastic conditional gradient method converges to the optimal objective function value at a sublinear rate of $\mathcal{O}(1/t^{1/3})$. Further, for a monotone and continuous DR-submodular function and subject to a \textit{general} convex body constraint, we prove that our proposed method achieves a $((1-1/e)\text{OPT} -\eps)$ guarantee (in expectation) with   $\mathcal{O}{(1/\eps^3)}$ stochastic gradient computations. This guarantee matches the known hardness results and closes the gap between deterministic and stochastic continuous submodular maximization. Additionally, we achieve $((1/e)\text{OPT} -\eps)$ guarantee after operating on $\mathcal{O}{(1/\eps^3)}$ stochastic gradients for the case that the objective function is continuous DR-submodular but non-monotone and the constraint set is a down-closed convex body. By using stochastic continuous optimization as an interface, we also provide the first $(1-1/e)$ tight approximation guarantee for maximizing  a \textit{monotone but stochastic} submodular \textit{set} function subject to a general matroid constraint and $(1/e)$ approximation guarantee for the non-monotone case. Numerical experiments for both convex and submodular settings are provided, and they illustrate fast convergence time for our proposed stochastic conditional gradient method relative to alternatives.
\end{abstract}



\section{Introduction}

Stochastic optimization arises in many procedures including wireless communications \citep{ribeiro2010ergodic}, learning theory \citep{vapnik2013nature}, machine learning \citep{bottou2010large}, adaptive filters \citep{haykin2008adaptive}, portfolio selection \citep{shapiro2009lectures} to name a few. In this class of problems the goal is to optimize an objective function defined as an expectation over a set of random functions subject to a general convex constraint.  In particular, consider an optimization variable $\bbx\in \ccalX \subset\reals^n$ and a random variable $\bbZ\in \ccalZ$ that together determine the choice of a stochastic function $\tilde{F}:\ccalX\times \ccalZ\to \reals$. The goal is to solve the program 
\begin{equation}\label{eq:problem}
\max_{\bbx\in\ccalC} F(\bbx) := \max_{\bbx\in\ccalC}  \mathbb{E}_{\bbz{\sim} P}\left[{\tilde{F}(\bbx,\bbz)}\right],
\end{equation}
where $\ccalC$ is a convex compact set, $\bbz$ is the realization of the random variable $\bbZ$ drawn from a distribution $P$, and $F(\bbx)$ is the expected value of the random functions $\tilde{F}(\bbx,\bbz)$ with respect to the random variable $\bbZ$. In this paper, our focus is on the cases where the function $F$ is either concave or continuous submodular. The first case which considers the problem of maximizing a concave function is equivalent to \textit{stochastic convex minimization}, and the second case in which the goal is to maximize a continuous submodular function is called  \textit{stochastic continuous submodular maximization.} Note that the main challenge here is to solve Problem \eqref{eq:problem} without computing the objective function $F(\bbx) $ or its gradient $\nabla F(\bbx) $ explicitly, since we assume that  either the probability distribution $P$ is unknown or the cost of computing the expectation is prohibitive.

In this regime, stochastic algorithms are the method of choice since they operate on computationally cheap stochastic estimates of the objective function gradients. Stochastic variants of the proximal gradient method are perhaps the most popular algorithms to tackle this category of problems both in convex minimization and submodular maximization. However, implementation of proximal methods requires projection onto a convex set at each iteration to enforce feasibility, which could be computationally expensive or intractable in many settings. To avoid the cost of projection, recourse to conditional gradient methods arises as a natural alternative. Unlike proximal algorithms, conditional gradient methods do not suffer from computationally costly projection steps and only require solving linear programs which can be implemented at a significantly lower complexity in many practical applications. 

In deterministic convex minimization, in which the exact gradient information is given, conditional gradient method, a.k.a., Frank-Wolfe algorithm \citep{frank1956algorithm,DBLP:conf/icml/Jaggi13}, succeeds in lowering the computational complexity of proximal algorithms due to its projection free property. However, in the stochastic regime, where only an estimate of the gradient is available, stochastic conditional gradient methods either diverge or underperform the proximal gradient methods. In particular, in stochastic convex minimization it is known and proven that stochastic conditional gradient methods may not converge to the optimal solution without an increasing batch size, whereas stochastic proximal gradient methods converge to the optimal solution at the sublinear rate of $\mathcal{O}(1/\sqrt{t})$. Hence, the possibility of designing a convergent stochastic conditional gradient method with a small batch size remains unanswered. 

In deterministic continuous submodular maximization where objective function is submodular, (neither convex nor concave), a variant of the condition gradient method called, continuous \citep{calinescu2011maximizing,bian16guaranteed} obtains a $(1-1/e)$ approximation guarantee for monotone functions, in contrast the best-known result for proximal gradient methods is a $(1/2)$ approximation guarantee \citep{hassani2017gradient}. However, in stochastic submodular maximization setting, stochastic variants of the continuous greedy algorithm with a small batch size fail to achieve a constant factor approximation \citep{hassani2017gradient}, whereas stochastic proximal gradient method recovers the $(1/2)$ approximation obtained by proximal gradient method in deterministic settings \citep{hassani2017gradient}. It is unknown if one can design a stochastic conditional gradient method that obtains a constant approximation guarantee, ideally $(1-1/e)$, for Problem~\eqref{eq:problem}.

In this paper, we introduce a stochastic conditional gradient method for solving the  generic stochastic optimization Problem \eqref{eq:problem}. The proposed method lowers the noise of gradient approximations through a simple gradient averaging technique which only requires a single stochastic gradient computation per iteration, i.e., the batch size can be as small as $1$. The proposed stochastic conditional gradient method improves the best-known convergence guarantees for stochastic conditional gradient methods  in both convex minimization and submodular maximization settings. A summary of our theoretical contributions follows.


\begin{enumerate}[label=(\roman*)]

\item In \textit{stochastic convex minimization}, we propose a stochastic variant of the Frank-Wolfe algorithm which converges to the optimal objective function value at the sublinear rate of $\mathcal{O}(1/t^{1/3})$. In other words, the proposed SFW algorithm achieves an $\eps$-suboptimal objective function value after $\mathcal{O}(1/\eps^{3})$ stochastic gradient evaluations.

\item In \textit{stochastic continuous submodular maximization}, we propose a stochastic conditional gradient method, which can be interpreted as a stochastic variant of the continuous greedy algorithm, that obtains the first tight $(1-1/e)$ approximation guarantee, when the function is monotone. For the non-monotone case, the proposed method obtains a $(1/e)$ approximation guarantee. Moreover, for the more general case of \textit{$\gamma$-weakly DR-submodular monotone maximization} the proposed stochastic conditional gradient method achieves a $(1-e^{-\gamma})$-approximation guarantee.

\item In \textit{stochastic discrete submodular maximization}, the proposed stochastic conditional gradient method achieves $(1-1/e)$ and $(1/e)$ approximation guarantees for monotone and non-monotone settings, respectively, by maximizing the multilinear extension of the stochastic discrete objective function. Further, if the objective function is monotone and has \textit{curvature $c$} the proposed stochastic conditional gradient method achieves a $(1-e^{-c})/c$ approximation guarantee.

%

\end{enumerate}

We begin the paper by studying the related work on stochastic methods in convex minimization and submodular maximization (Section~\ref{sec:related_work}). Then, we proceed by stating stochastic convex minimization problem (Section~\ref{sec:stochastic_convex_min}). We introduce a stochastic conditional gradient method which can be interpreted as a stochastic variant of Frank-Wolfe (FW) algorithm for solving stochastic convex minimization problems (Section~\ref{sec:SFO}). The proposed \algFW method (SFW) differs from the vanilla FW method in replacing gradients by their stochastic approximations evaluated based on averaging over previously observed stochastic gradients. We further analyze the convergence properties of the proposed SFW method (Section~\ref{sec:SFO_analysis}). In particular, we show that the averaging technique in SFW lowers the noise of gradient approximation (Lemma \ref{lemma:bound_on_gradient_error}) and consequently the sequence of objective function values generated by SFW converges to the optimal objective function value at a sublinear rate of $\mathcal{O}(1/t^{1/3})$ in expectation (Theorem \ref{thm:rate_convex}). We complete this result by proving that the sequence of objective function values almost surely converges to the optimal objective function value (Theorem~\ref{thm:almost_sure_cnvrg}).

We then focus on the application of the proposed stochastic conditional gradient method in continuous submodular maximization (Section~\ref{sec:Continuous_Submodular_Maximization}). After defining the notions of submodularity (Section~\ref{defs}), we introduce the \alg (SCG) algorithm for solving continuous submodular maximization problems (Section~\ref{sec:scg}). The proposed SCG algorithm achieves the optimal $(1-1/e)$-approximation when the expected objective function is DR-submodular, monotone, and smooth (Theorem \ref{thm:optimal_bound_greedy}). To be more precise, the expected objective value of the iterates generated by SCG in expectation is not smaller $(1-1/e)OPT-\eps$ after $\mathcal{O}(1/\eps^{3})$ stochastic gradient evaluations. Moreover, for the case that the expected function is not DR-submodular but $\gamma$ weakly DR-submodular, the SCG algorithm obtains an $(1-e^{-\gamma})$-approximation guarantee (Theorem~\ref{thm:optimal_bound_greedy_weak}). We further extend our results to non-monotone setting by introducing the \algNM (NMSCG) method. We show that under the assumptions that the expected function is only DR-submodular and smooth NMSCG reaches a $(1/e)$-approximation guarantee (Theorem~\ref{thm:non_monotone_bound}).

The continuous multilinear extension of discrete submodular functions implies that the results for stochastic continuous DR-submodular maximization can be extended to stochastic discrete submodular maximization.  We formalize this connection by introducing the stochastic discrete submodular maximization problem and defining its continuous multilinear extension (Section~\ref{sec:Discrete_Submodular_Maximization}). By leveraging this connection, one can relax the discrete problem to a stochastic continuous submodular maximization, use SCG to solve the relaxed continuous problem within a $(1-1/e-\epsilon)$ approximation to the optimum value (Theorem~\ref{thm:multi_linear_extenstion_thm}), and use a proper rounding scheme (such as the contention resolution method \citep{chekuri2014submodular}) to obtain a feasible set whose value is a $(1-1/e-\eps)$ approximation to the optimum set in expectation. In summary, we show that SCG achieves an $((1-1/e)OPT-\epsilon)$ approximation for a generic discrete monotone stochastic submodular maximization problem after $\mathcal{O}(n^{3/2}/\eps^{3})$ iterations where $n$ is the size of the ground set (Corollary~\ref{main_corollary}). We further prove a $1/e$ approximation guarantee for NMSCG when we maximize a \textit{non-monotone} stochastic submodular set function (Theorem~\ref{thm:multi_linear_extenstion_thm_non_monotone_case}). Moreover, if the expected set function has curvature $c\in[0,1]$, SCG reaches an $(1/c)(1-e^{-c})$ approximation guarantee (Theorem~\ref{thm:optimal_bound_curvature_greedy}). 
We finally close the paper by concluding remarks (Section~\ref{sec:conclusion}).

\medskip\noindent{\bf Notation.\quad} Lowercase boldface $\bbv$ denotes a vector and uppercase boldface $\bbA$ denotes a matrix. We use $\|\bbv\|$ to denote the Euclidean norm of vector $\bbv$. The $i$-th element of the vector $\bbv$ is written as $v_i$ and the element on the i-$th$ row and $j$-th column of the matrix $\bbA$ is denoted by $\bbA_{i,j}$.

\section{Related Work}\label{sec:related_work}

In this section we overview the literature on conditional gradient methods in convex minimization as well as submodular maximization and compare our novel theoretical guarantees with the existing results.

\medskip\noindent{\bf Convex minimization.\ }The problem of minimizing a stochastic convex function subject to a convex constraint has been tackled by many researchers and many approaches have been reported in the literature. Projected stochastic gradient (SGD) and stochastic variants of Frank-Wolfe algorithm are among the most popular approaches. SGD updates the iterates by descending through the negative direction of stochastic gradient with a proper stepsize  and projecting the resulted point onto the feasible set \citep{robbins1951stochastic,nemirovski1978cezari,nemirovskii1983problem}. Although stochastic gradient computation is inexpensive, the cost of projection step can be high \citep{fujishige2011submodular} or intractable \citep{collins2008exponentiated}. In such cases projection-free conditional gradient methods, a.k.a., Frank-Wolfe algorithm, are more practical \citep{frank1956algorithm,DBLP:conf/icml/Jaggi13}. The online Frank-Wolfe algorithm proposed by \cite{DBLP:conf/icml/HazanK12} requires $\mathcal{O}(1/\eps^4)$ stochastic gradient evaluations to reach $\eps$ suboptimal objective function value, i.e., $\E{f(\bbx)-OPT}\leq \eps$ under the assumption that the objective function is convex and has bounded gradients. The stochastic Frank-Wolfe studied by \cite{DBLP:conf/icml/HazanL16} obtains an improved complexity of $\mathcal{O}(1/\eps^3)$ under the assumptions that the expected objective function is smooth (has Lipschitz gradients) and Lipschitz continuous (the gradients are bounded). More importantly, the stochastic Frank-Wolfe algorithm in \citep{DBLP:conf/icml/HazanL16} requires an increasing batch size $b$ as time progresses, i.e., $b=\mathcal{O}(t)$. In this paper, we propose a stochastic variant of conditional gradient method which achieves the complexity of $\mathcal{O}(1/\eps^3)$ under milder assumptions (only requires smoothness of the expected function) and operates with a fixed batch size, e.g., $b=1$. 

\medskip\noindent{\bf Submodular maximization.\ } Maximizing a deterministic submodular set function has been extensively studied. The celebrated result of \cite{nemhauser1978analysis} shows that a greedy algorithm achieves a $(1-1/e)$ approximation guarantee for a monotone function subject to a cardinality constraint. It is also known that this result is tight  under reasonable complexity-theoretic assumptions \citep{feige1998threshold}. Recently, variants of the greedy algorithm have been proposed to extend the above result to non-monotone and more general constraints \citep{feige2011maximizing, buchbinder2015tight, buchbinder2014submodular, DBLP:conf/icml/MirzasoleimanBK16,feldman2017greed}. While discrete greedy algorithms are fast, they usually do not provide the tightest guarantees for many classes of feasibility constraints. This is why continuous relaxations of submodular functions, e.g., the multilinear extension, have gained a lot of interest \citep{vondrak2008optimal, calinescu2011maximizing, chekuri2014submodular, feldman2011unified, gharan2011submodular, sviridenko2017optimal}. In particular, it is known that the continuous greedy algorithm achieves a $(1-1/e)$ approximation guarantee for monotone submodular functions under a general matroid constraint \citep{calinescu2011maximizing}. An improved   $((1-e^{-c})/c)$-approximation guarantee can be obtained if the objective function has curvature $c$ \citep{vondrak2010submodularity}. 
The problem of maximizing submodular functions also has been studied for the non-monotone case, and constant approximation guarantees have been established 
\citep{feldman2011unified,buchbinder2015tight,DBLP:conf/focs/EneN16,buchbinder2016constrained}.

\begin{table}[t]
\begin{center}
\begin{tabular}{|c|c|c|c|c| c|}
\hline
\textbf{Ref.}& \textbf{setting}  & \textbf{assumptions}& \textbf{batch} & \textbf{rate} &\textbf{complexity}  \\
\hline
\cite{DBLP:conf/icml/Jaggi13} & det.& smooth & --- &$\mathcal{O}(1/{t})$ & ---  \\ 
\hline
\cite{DBLP:conf/icml/HazanK12} & stoch.    &smooth, bounded grad. & $\mathcal{O} (t)$ &$\mathcal{O}(1/{t^{1/2}}) $ &$\mathcal{O}(1/\eps^{4}) $ \\
\hline
\cite{DBLP:conf/icml/HazanL16} & stoch.    &smooth, bounded grad. & $\mathcal{O} (t^2)$ &$\mathcal{O}(1/{t}) $ &$\mathcal{O}(1/\eps^{3}) $ \\
\hline\hline
This paper & stoch.  & smooth, bounded var.& $\mathcal{O} (1)$ &$O(1/{t^{1/3}}) $&$\mathcal{O}(1/\eps^{3}) $ \\
\hline
\end{tabular}
\caption{Convergence guarantees of conditional gradient (FW) methods for convex minimization}
\end{center}
\end{table}

\begin{table}[t]
\begin{center}
\begin{tabular}{| c| c| c| c| c| c| }
\hline
\textbf{Ref.}& \textbf{setting} & \textbf{function}  & \textbf{const.} & \textbf{utility} & \textbf{complexity}  \\
\hline 
 \cite{chekuri2015multiplicative} & det.& mon.smooth sub. & poly. & $(1-1/e) \rm{OPT}-\epsilon $ & $O(1/{\epsilon^{2}})$\\
\hline
\cite{bian16guaranteed} & det. &  mon. DR-sub. &cvx-down & $(1-1/e) \rm{OPT}-\epsilon $ & $O(1/{\epsilon}) $\\
\hline
\cite{DBLP:conf/nips/BianL0B17} & det. &  non-mon. DR-sub. &cvx-down & $(1/{e}) \rm{OPT}-\epsilon $ & $O(1/{\epsilon})$\\
\hline
\cite{hassani2017gradient}  & det. &mon. DR-sub. & convex &$(1/2)\rm{OPT}-\epsilon$& $O(1/{\epsilon})$ \\
\hline
\cite{hassani2017gradient} & stoch. &mon. DR-sub. & convex &$(1/2)\rm{OPT}-\epsilon$& $O(1/{\epsilon^{2}})$ \\
\hline
\cite{hassani2017gradient}  & stoch. & mon. weak DR-sub. &convex & $\frac{\gamma^2}{1+\gamma^2}\rm{OPT}-\epsilon$ & $O(1/{\epsilon^{2}})$\\
\hline
\hline
This paper & stoch. &mon. DR-sub. & convex &$(1-1/e)\rm{OPT}-\epsilon$& $O(1/{\epsilon^{3}})$ \\
\hline
This paper & stoch. & weak DR-sub. &convex & $(1-e^{-\gamma})\rm{OPT}-\epsilon$ & $O(1/{\epsilon^{3}})$\\
\hline
This paper & stoch. &non-mon. DR-sub. & convex &$(1/e)\rm{OPT}-\epsilon$& $O(1/{\epsilon^{3}})$ \\
\hline
\end{tabular}
\caption{Convergence guarantees for continuous DR-submodular function maximization}
\end{center}
\end{table}

\begin{table}[t]
\begin{center}
\begin{tabular}{|c| c| c| c| c| c|  }
\hline
\textbf{Ref.}& \textbf{setting}& \textbf{function}&\textbf{constraint}&\textbf{approx} & \textbf{method} \\
\hline
\cite{nemhauser81mip} & det. & mon. sub. &cardinality & $1-1/e$  &disc. greedy  \\\hline
\cite{nemhauser81mip} & det.  & mon. sub. & matroid & $1/2$ & disc. greedy  \\\hline
\cite{calinescu2011maximizing} & det. & mon. sub.   & matroid & $1-1/e$ & con. greedy\\\hline
\cite{hassani2017gradient} &{stoch. }& {mon. sub. }& {matroid} & ${1/2}$& {SGA} \\\hline\hline
This paper & {stoch. }& {mon. sub. }& {matroid} & $1-1/e$& {SCG } \\\hline
This paper & {stoch. }& {mon. sub. }& {matroid} & $(1-1/e^{c})/c$& {SCG } \\\hline
This paper & {stoch. }& {sub. }& {matroid} & $1/e$& {NMSCG } \\\hline
\end{tabular}
\caption{Convergence guarantees for submodular set function maximization}
\end{center}
\end{table}

Continuous submodularity naturally arises   in many learning applications such as robust budget allocation \citep{staib2017robust,soma2014optimal},  online resource allocation \citep{eghbali2016designing},  learning assignments \citep{golovin2014online}, as well as Adwords for e-commerce and advertising \citep{devanur2012online, mehta2007adwords}.  Maximizing a \textit{deteministic} continuous submodular function dates back to the work of \cite{wolsey1982analysis}. More recently, \cite{chekuri2015multiplicative} proposed a multiplicative weight update algorithm that achieves a $(1-1/e-\epsilon)$ approximation guarantee after $\tilde{O}(n/\epsilon^2)$ oracle calls to gradients of a monotone  smooth submodular function $F$ (i.e., twice differentiable DR-submodular) subject to a polytope constraint. A similar approximation factor can be obtained after $\mathcal{O}(n/\epsilon)$ oracle calls to gradients of $F$ for monotone DR-submodular functions subject to a down-closed convex body using the continuous greedy method \citep{bian16guaranteed}. However, such results require exact computation of the gradients $\nabla F$ which is not feasible in Problem~\eqref{eq:mainproblem}. An alternative approach is then to modify the current algorithms by replacing gradients $\nabla F(\bbx_t)$ by their stochastic estimates $\nabla \tilde{F}(\bbx_t,\bbz_t)$; however, this modification may lead to arbitrarily poor solutions as demonstrated in~\citep{hassani2017gradient}. Another alternative is to estimate the gradient by averaging over a (large) mini-batch of samples at each iteration. While this approach can potentially reduce the noise variance,  it increases the computational complexity of each iteration and is not favorable. The work by \cite{hassani2017gradient} is perhaps the first attempt to solve Problem~\eqref{eq:mainproblem} only by executing stochastic estimates of gradients (without using a large batch). They showed that the stochastic gradient ascent method achieves a $(1/2-\epsilon)$ approximation guarantee after $O(1/\epsilon^2)$ iterations. Although this work opens the door for maximizing stochastic continuous submodular functions  using computationally cheap stochastic gradients, it fails to achieve the optimal $(1-1/e)$ approximation. To close the gap, we propose in this paper \alg which outputs a solution with function value at least  $((1-1/e)\text{OPT}- \epsilon)$ after $O(1/\epsilon^3)$  iterations. Notably, our result only requires the expected function $F$ to be monotone and DR-submodular and the stochastic functions $\tilde{F}$ need not be monotone nor DR-submodular. Moreover, in contrast to the result in \citep{bian16guaranteed}, which holds for down-closed convex constraints, our result holds for any convex constraints. For non-monotone DR-submodular functions, we also propose the non-monotone stochastic continuous greedy (NMSCG) method that achieves a solution with function value at least $((1/e)\text{OPT}-\eps)$ after at most $O(1/\epsilon^3)$ iterations. Crucially, the feasible set in this case should be down-closed or otherwise no constant approximation guarantee is possible \citep{chekuri2014submodular}.

Our result also has important implications for the problem of maximizing a stochastic discrete submodular function subject to a matroid constraint. Since the proposed SCG method works in stochastic settings, we can relax the discrete objective function $f$ to a continuous function $F$ through the multi-linear extension (note that expectation is a linear operator). Then  we can maximize $F$ within a $(1-1/e-\epsilon)$ approximation to the optimum value by using only ${\mathcal{O}}(1/\eps^3)$ oracle calls to the stochastic gradients of $F$ when the functions are monotone. Finally, a proper rounding scheme (such as the contention resolution method \citep{chekuri2014submodular}) results in a feasible set whose value is a $(1-1/e)$ approximation to the optimum set in expectation\footnote{For the ease of presentation, and in the discrete setting, we only present our results for the matroid constraint. However, our stochastic continuous algorithms can provide constant factor approximations for any constrained submodular maximization setting where an efficient and loss-less rounding scheme exists. It includes, for example, knapsack constraints among many others. }. Using the same procedure we can also prove a $(1/e)$ approximation guarantee in expectation for the non-monotone case. Additionally, when the set function $f$ is monotone and has a curvature $c<1$ -- check \eqref{eq:curvature_def} for the definition of the curvature -- we show that the approximation factor can be improved from $(1-1/e)$ to $((1-1/e^c)/c)$.

\section{Stochastic Convex Minimization}\label{sec:stochastic_convex_min}

Many problems in machine learning can be reduced to the minimization of a convex objective function defined as an expectation over a set of random functions. In this section, our focus is on developing a stochastic variant of the conditional gradient method, a.k.a., Frank-Wolfe, which can be applied to solve general stochastic convex minimization problems. To make the notation consistent with other works on stochastic convex minimization, instead of solving Problem \eqref{eq:problem} for the case that the objective function $F$ is concave, we assume that $F$ is convex and intend to minimize it subject to a convex set $\ccalC$. Therefore, the goal is to solve the program 
\begin{equation}\label{eq:convex_problem}
\min_{\bbx\in\ccalC} F(\bbx) := \min_{\bbx\in\ccalC}  \mathbb{E}_{\bbz{\sim} P}\left[{\tilde{F}(\bbx,\bbz)}\right].
\end{equation}
A canonical subset of problems having this form is support vector machines, least mean squares, and logistic regression.

In this section we assume that only the expected (average) function $F$ is convex and smooth, and the stochastic functions $\tilde{F}$ may not be convex nor smooth. Since the expected function $F$ is convex, descent methods can be used to solve the program in \eqref{eq:convex_problem}. However, computing the gradients (or Hessians) of the function $F$ percisely requires access to the distribution $P$, which may not be feasible in many applications. To be more specific, we are interested in settings where the distribution $P$ is either unknown or evaluation of the expected value is computationally prohibitive. In this regime, stochastic gradient descent methods, which operate on stochastic approximations of the gradients, are the mostly used alternatives. In the following section, we aim to develop a stochastic variant of the Frank-Wolfe method which converges to an optimal solution of \eqref{eq:convex_problem}, while it only requires access to a single stochastic gradient $\nabla \tilde{F}(\bbx,\bbz)$ at each iteration. 


\subsection{Stochastic Frank-Wolfe Method}\label{sec:SFO}
In this section, we introduce a stochastic variant of the Frank-Wolfe method to solve Problem  \eqref{eq:convex_problem}. Assume that at each iteration we have access to the stochastic gradient $\nabla \tilde{F}(\bbx,\bbz)$ which is an unbiased estimate of the gradient $\nabla F(\bbx)$. It is known that a naive stochastic implementation of Frank-Wolfe (replacing gradient $\nabla F(\bbx)$ by $\nabla \tilde{F}(\bbx,\bbz)$) might diverge, due to non-vanishing variance of gradient approximations. To resolve this issue, we introduce a stochastic version of the Frank-Wolfe algorithm which reduces the noise of gradient approximations via a common averaging technique in stochastic optimization \citep{ruszczynski1980feasible,ruszczynski2008merit,yang2016parallel,mokhtari2017large}. 

By letting $t \in  \mathbf{N}$ be a discrete time index and $\rho_t$ a given stepsize which approaches zero as $t$ grows, the proposed biased gradient estimation $\bbd_t$ is defined by the recursion  
\begin{equation}\label{eq:grad_averaging}
\bbd_t = (1-\rho_t) \bbd_{t-1} + \rho_t \nabla \tilde{F}(\bbx_t,\bbz_t),
\end{equation}
where the initial vector is given as $\bbd_0=\bb0$. We will show that the averaging technique in \eqref{eq:grad_averaging} reduces the noise of gradient approximation as time increases. More formally, the expected noise of the gradient estimation $\E{\|\bbd_t-\nabla F(\bbx_t)\|^2}$ approaches zero asymptotically. This property implies that the biased gradient estimate $\bbd_t$ is a better candidate for approximating the gradient $\nabla F(\bbx_t)$ comparing to the unbiased gradient estimate $\nabla \tilde{F}(\bbx_t,\bbz_t)$ that suffers from a high variance approximation. We therefore define the descent direction $\bbv_t$ of our proposed  \algFW (SFW) method as the solution of the linear program
\begin{equation}\label{eq:descent_direction}
\bbv_t =\argmin_{\bbv\in \ccalC} \{ \bbd_t^T\bbv \}.
\end{equation}
As in the traditional FW method, the updated variable $\bbx_{t+1}$ is a convex combination of $\bbv_t$ and the iterate $\bbx_t$
\begin{equation}\label{eq:variable_update}
\bbx_{t+1} = (1-\gamma_{t+1}) \bbx_{t} + \gamma_{t+1} \bbv_t,
\end{equation}
where $\gamma_t$ is a proper positive stepsize. Note that each iteration of the proposed SFW method only requires a single stochastic gradient computation, unlike the methods in \citep{DBLP:conf/icml/HazanL16,reddi2016stochastic} which an require increasing number of stochastic gradient evaluations as the number of iterations $t$ grows. 

%
\begin{algorithm}[t]
\caption{ \algFW (SFW)}\label{algo_SFW} 
\begin{algorithmic}[1] 
\small{\REQUIRE Stepsizes $\rho_t>0$ and $\gamma_t>0$. Initialize $\bbd_0=\bb0$ and choose $\bbx_0\in\ccalC$
\FOR {$t=1,2,\ldots$}
   \STATE Update the gradient estimate $\bbd_t = (1-\rho_t) \bbd_{t-1} + \rho_t \nabla \tilde{F}(\bbx_t,\bbz_t)$;
   \STATE Compute $\bbv_t =\argmin_{\bbv\in \ccalC} \{ \bbd_t^T\bbv \}$;
   \STATE Compute the updated variable $\bbx_{t+1} = (1-\gamma_{t+1}) \bbx_{t} + \gamma_{t+1} \bbv_t$;
\ENDFOR}
\end{algorithmic}\end{algorithm}

The proposed SFW is summarized in Algorithm~\ref{algo_SFW}. The core steps are Steps 2-5 which follow the updates in \eqref{eq:grad_averaging}-\eqref{eq:variable_update}. The initial variable $\bbx_0$ can be any feasible vector in the convex set $C$ and the initial gradient estimate is set to be the null vector $\bbd_0=\bb0$. The sequence of positive parameters $\rho_t$ and $\gamma_t$ should be diminishing at proper rates as we describe in the following convergence analysis section.

\subsection{Convergence Analysis}\label{sec:SFO_analysis}

In this section we study the convergence rate of the proposed SFW method for solving the constraint convex program in \eqref{eq:convex_problem}. To do so, we first assume that the following conditions hold.

\begin{assumption}\label{ass:bounded_set_convex}
The convex set $\ccalC$ is bounded with diameter $D$, i.e., for all $\bbx,\bby \in \ccalC$ we can write
\begin{equation}
\|\bbx-\bby\|\leq D.
\end{equation}
\end{assumption}
\begin{assumption}\label{ass:smoothness_convex}
The expected function $F$ is convex. Moreover, its gradients $\nabla F$ are $L$-Lipschitz continuous over the set $\ccalC$, i.e., for all $\bbx,\bby \in \ccalC$
\begin{equation}
\| \nabla F(\bbx) -  \nabla F(\bby) \| \leq L \| \bbx - \bby \|.
\end{equation}
\end{assumption}
\begin{assumption}\label{ass:bounded_variance_convex}
The variance of the unbiased stochastic gradients $\nabla \tilde{F}(\bbx,\bbz)$ is bounded above by $\sigma^2$, i.e., for all random variables $\bbz$ and vectors $\bbx\in\ccalC$ we can write 
\begin{equation}
\E{ \| \nabla \tilde{F}(\bbx,\bbz) - \nabla F(\bbx)  \|^2 } \leq \sigma^2.
\end{equation}
\end{assumption}

Assumption~\ref{ass:bounded_set_convex} is standard in constrained convex optimization and is implied by the fact that the set $\ccalC$ is convex and compact. The condition in Assumption~\ref{ass:smoothness_convex} ensures that the objective function $F$ is smooth on the set $\ccalC$. Note that here we only assume that the (average) function $F$ has Lipschitz continuous gradients, and the stochastic gradients $\nabla \tilde{F}(\bbx,\bbz)$ may not be Lipschitz continuous. Finally, the required condition in Assumption~\ref{ass:bounded_variance_convex} is customary in stochastic optimization and  guarantees that the variance of stochastic gradients $\nabla \tilde{F}(\bbx,\bbz)$ is bounded by a finite constant $\sigma^2<\infty$.

To study the convergence rate of SFW, we first derive an upper bound on the error of gradient approximation $\|\nabla F(\bbx_t) - \bbd_t\|^2$ in the following lemma.

\begin{lemma} \label{lemma:bound_on_gradient_error}
Consider the proposed \algFW (SFW) method outlined in Algorithm~\ref{algo_SFW}. If the conditions in Assumptions \ref{ass:bounded_set_convex}-\ref{ass:bounded_variance_convex} hold, the sequence of squared gradient errors $\|\nabla F(\bbx_t) - \bbd_t\|^2$ satisfies
\begin{align}\label{eq:bound_on_gradient_error}
\E{\|\nabla F(\bbx_t) - \bbd_t\|^2 \mid \ccalF_{t}} \leq \left(1-\frac{\rho_t}{2}\right)\|\nabla F(\bbx_{t-1}) - \bbd_{t-1}\|^2+ \rho_t^2\sigma^2 
+\frac{2L^2D^2\gamma_{t}^2 }{\rho_t},
\end{align}
where $\ccalF_t$ is a sigma-algebra measuring all sources of randomness up to step $t$.
\end{lemma}

\begin{proof}
Check Appendix \ref{app:proof_of_lemma:bound_on_gradient_error}.
\end{proof}

The result in Lemma~\ref{lemma:bound_on_gradient_error} shows that squared error of gradient approximation $\|\nabla F(\bbx_t) - \bbd_t\|^2$ decreases in expecation at each iteration by the factor $(1-\rho_t/2)$ if the remaining terms on the right hand side of \eqref{eq:bound_on_gradient_error} are negligible relative to the term $(1-\rho_t/2)\|\nabla F(\bbx_{t-1}) - \bbd_{t-1}\|^2$. This condition can be satisfied, if the parameters $\rho_t$ and $\gamma_t$ are properly chosen. This observation verifies our intuition that the noise of the stochastic gradient approximation diminishes as the number of iterations increases. 

In the following lemma, we derive an upper bound on the suboptimality $F(\bbx_{t+1})-F(\bbx^*)$ which depends on the norm of gradient error $\|\nabla F(\bbx_t) - \bbd_t\|$.

\begin{lemma} \label{lemma:bound_on_suboptimality}
Consider the proposed \algFW (SFW) method outlined in Algorithm~\ref{algo_SFW}. If the conditions in Assumptions \ref{ass:bounded_set_convex}-\ref{ass:bounded_variance_convex} are satisfied, the suboptimality $F(\bbx_{t+1})-F(\bbx^*)$ satisfies
\begin{align}\label{eq:bound_on_suboptimality}
{F(\bbx_{t+1}) -F(\bbx^*) } 
\leq (1-\gamma_{t+1}){(F(\bbx_{t}) -F(\bbx^*))}+\gamma_{t+1} D\|\nabla F(\bbx_{t})-\bbd_t\| +\frac{LD^2\gamma_{t+1}^2}{2}.
\end{align}
\end{lemma}

\begin{proof}
Check Appendix \ref{app:proof_of_lemma:bound_on_suboptimality}.
\end{proof}

Based on Lemma \ref{lemma:bound_on_suboptimality}, the suboptimality $F(\bbx_{t+1}) -F(\bbx^*) $ approaches zero if the error of gradient approximation $\|\nabla F(\bbx_{t})-\bbd_t\|$ converges to zero sufficiently fast and the last term $({LD^2\gamma_{t+1}^2})/{2}$ is summable, i.e., $\sum_{t=1}^\infty \gamma_{t}^2<\infty$. In the special case of zero approximation error, which is equivalent to the case that we have access to the expected gradient $\nabla F(\bbx_t)$, by setting $\gamma_t=\mathcal{O}(1/t)$ it can be shown that the suboptimality $F(\bbx_{t}) -F(\bbx^*) $ converges to zero at the sublinear rate of $\mathcal{O}(1/t)$. Therefore, the result in Lemma \ref{lemma:bound_on_suboptimality} is consistent with the analysis of Frank-Wolfe method \citep{DBLP:conf/icml/Jaggi13}.

In the following theorem, by using the results in Lemmas \ref{lemma:bound_on_gradient_error} and \ref{lemma:bound_on_suboptimality}, we establish an upper bound on the expected suboptimality $\E{F(\bbx_{t}) -F(\bbx^*)}$.

\begin{theorem} \label{thm:rate_convex}
Consider the proposed \algFW (SFW) method outlined in Algorithm~\ref{algo_SFW}. Suppose the conditions in Assumptions \ref{ass:bounded_set_convex}-\ref{ass:bounded_variance_convex} are satisfied. If we set $\gamma_t=2/(t+8)$ and $\rho_t=4/(t+8)^{2/3}$, then the expected suboptimality $\E{F(\bbx_{t+1})-F(\bbx^*)}$ is bounded above by
\begin{align}\label{eq:rate}
\E{F(\bbx_{t}) -F(\bbx^*) } \leq \frac{Q}{(t+9)^{1/3}},
\end{align}
where the constant $Q$ is given by
\begin{equation}
Q:=\max\left\{9^{1/3}(F(\bbx_{0}) -F(\bbx^*)),\frac{LD^2}{2}+2D\max\left\{2\|\nabla F(\bbx_0) - \bbd_0\|,\left(16\sigma^2+2L^2D^2 \right)^{1/2}\right\} \right\}.
\end{equation}
\end{theorem}

\begin{proof}
Check Appendix \ref{app:proof_of_thm:rate_convex}.
\end{proof}

The result in Theorem~\ref{thm:rate_convex} indicates that the expected suboptimality $\E{F(\bbx_{t}) -F(\bbx^*) }$ of the iterates generated by the SFW method converges to zero at least at a sublinear rate of $\mathcal{O}(1/t^{1/3})$. In other words, it shows that to achieve the expected suboptimality $\E{F(\bbx_{t}) -F(\bbx^*) }\leq \eps$, the number of required stochastic gradients (sample gradients) to reach this accuracy is $\mathcal{O}(1/\eps^3)$.

 To complete the convergence analysis of SFW we also prove that the sequence of the objective function values $F(\bbx_{t})$ converges to the optimal value $F(\bbx^*) $ almost surely. This result is formalized in the following theorem.

\begin{theorem} \label{thm:almost_sure_cnvrg}
Consider the proposed \algFW (SFW) method outlined in Algorithm~\ref{algo_SFW}. Suppose that the conditions in Assumptions \ref{ass:bounded_set_convex}-\ref{ass:bounded_variance_convex} are satisfied. If we choose $\rho_t$ and $\gamma_t$ such that (i)~$\sum_{t=0}^\infty \rho_t = \infty $, (ii) $\sum_{t=0}^\infty \rho_t^2<\infty $, (iii) $\sum_{t=0}^\infty \gamma_t = \infty $,  and (iv) $\sum_{t=0}^\infty (\gamma_t^2/\rho_t)<\infty $, then the suboptimality ${F(\bbx_{t})-F(\bbx^*)}$ converges to zero almost surely, i.e.,
\begin{align}\label{eq:almost_sure_cnvrg}
\lim_{t \to \infty}  {F(\bbx_{t}) -F(\bbx^*) }\ \myeq \ 0.
\end{align}
\end{theorem}

\begin{proof}
Check Appendix \ref{app:proof_of_thm:almost_sure_cnvrg}.
\end{proof}

Theorem~\ref{thm:almost_sure_cnvrg} provides almost sure convergence of the sequence of objective function value $F(\bbx_t)$ to the optimal value $F(\bbx^*)$. In other words it shows that the sequence of the objective function values $F(\bbx_t)$ converges to $F(\bbx^*)$ with probability 1. Indeed, a valid set of choices for $\gamma_t$ and $\rho_t$ to satisfy the required conditions in Theorem~\ref{thm:almost_sure_cnvrg} are $\gamma_t=\mathcal{O}(1/t)$ and $\rho_t=\mathcal{O}(1/t^{2/3})$.

\section{Stochastic Continuous Submodular Maximization}\label{sec:Continuous_Submodular_Maximization}

In the previous section, we focused on the convex setting, but what if the objective function $F$ is not convex? In this section, we consider a  broad class of non-convex optimization problems that possess special combinatorial structures. More specifically, we focus on constrained maximization of stochastic continuous submodular functions that  demonstrate diminishing returns, i.e., continuous DR-submodular functions,
\begin{equation}\label{eq:mainproblem}
\max_{\bbx \in \ccalC}\ F(\bbx) \doteq \max_{\bbx \in \ccalC}\ \mathbb{E}_{\bbz\sim P}[{\tilde{F} (\bbx,\bbz)}].
\end{equation}
As before, the functions $\tilde{F}: \ccalX \times \mathcal{Z} \to \reals_{+} $ are stochastic where $\bbx \in \ccalX $ is the optimization variable, $\bbz\in  \mathcal{Z}$ is a realization of the random variable $\bbZ$ drawn from a distribution $P$, and $\ccalX\in \reals^n_+$ is a compact set. Our goal is to maximize the expected value of the random functions $\tilde{F}(\bbx,\bbz)$ over the convex body $\ccalC \subseteq \ccalX$. Note that we \emph{only assume} that $F(\bbx)$ is DR-submodular, and \emph{not} necessarily the stochastic functions $\tilde{F}(\bbx,\bbz)$. We also consider situations where  the distribution $P$ is either unknown (e.g., when the objective is given as an implicit stochastic model) or the domain of the random variable $\bbZ$ is very large (e.g., when the objective is defined in terms of an empirical risk) which makes the cost of computing the expectation very high. In these regimes, stochastic optimization methods, which operate on computationally cheap estimates of gradients, arise as  natural solutions. In fact, very recently, it was shown in \citep{hassani2017gradient} that stochastic gradient methods achieve a $(1/2)$ approximation guarantee to Problem~\eqref{eq:mainproblem}. In Section 3 of \citep{hassani2017gradient}, the authors also showed that if we simply substitute gradients by stochastic gradients in the update of conditional gradient methods (a.k.a., Frank-Wolfe), such as continuous greedy \citep{vondrak2008optimal} or its close variant \citep{bian16guaranteed}, the resulted method can perform arbitrarily poorly in stochastic continuous submodular maximization settings. Our goal in this section is to design a stable stochastic variant of conditional gradient method to solve Problem \eqref{eq:mainproblem} up to a constant factor. 

\subsection{Preliminaries} \label{defs}
We begin by recalling the definition of a submodular set function: A function $f:2^V\rightarrow \reals_+$, defined on the ground set $V$,  is called submodular if for all subsets $A,B\subseteq V$, we have $$f(A)+f(B)\geq f(A\cap B) + f(A\cup B).$$
The notion of submodularity goes beyond the discrete domain \citep{wolsey1982analysis, vondrak2007submodularity, bach2015submodular}.
Consider a continuous function $F: \ccalX \to \reals_{+}$ where the set $\ccalX$ is of  the form $\ccalX=\prod_{i=1}^n\ccalX_i$ and each $\ccalX_i$ is a compact subset of $\reals_+$. We call the continuous function $F$ submodular if for all $\bbx,\bby\in \ccalX$ we have
\begin{align}\label{eq:submodular_def}
F(\bbx) + F(\bby) \geq F(\bbx \vee	 \bby) + F(\bbx \wedge \bby) ,
\end{align}
where $\bbx \vee \bby := \max (\bbx ,\bby )$ (component-wise) and $\bbx \wedge \bby := \min (\bbx ,\bby )$ (component-wise). Further, a submodular function $F$ is monotone (on the set $\ccalX$) if 
\begin{align}\label{eq:monotone_def}
\bbx \leq \bby  \quad \Longrightarrow \quad F(\bbx) \leq  F(\bby),
\end{align}
for all $\bbx,\bby\in \ccalX$. Note that $\bbx \leq \bby$  in \eqref{eq:monotone_def} means that $x_i\leq y_i$ for all $i=1,\dots,n$. Furthermore, a differentiable submodular function $F$ is called \textit{DR-submodular} (i.e., shows diminishing returns) if the gradients are antitone, namely, for all $\bbx,\bby\in \ccalX$ we have 
\begin{align}\label{eq:antitone_def}
\bbx \leq \bby  \quad \Longrightarrow \quad \nabla F(\bbx) \geq \nabla  F(\bby).
\end{align}
When the function $F$ is twice differentiable, submodularity implies that all cross-second-derivatives are non-positive \citep{bach2015submodular}, i.e., 
\begin{equation}
\forall\ i\neq j,\ \ \forall\ \bbx\in \ccalX, ~~ \frac{\partial^2 F(\bbx)}{\partial x_i \partial x_j} \leq 0, 
\end{equation}
 and DR-submodularity implies that all second-derivatives  are  non-positive \citep{bian16guaranteed}, i.e., 
\begin{equation}
\forall\ i,j,\ \ \forall\ \bbx\in \ccalX, ~~ \frac{\partial^2 F(\bbx)}{\partial x_i \partial x_j} \leq 0.
\end{equation}
\subsection{Stochastic Continuous Greedy}\label{sec:scg}
We proceed to introduce, \alg (SCG), which  is a stochastic variant of the continuous greedy method \citep{vondrak2008optimal} to solve Problem~\eqref{eq:mainproblem}.
We only assume that the expected objective function $F$ is monotone and DR-submodular and the stochastic functions $\tilde{F}(\bbx,\bbz)$ may not be monotone nor submodular. Since the objective function $F$ is monotone and DR-submodular, continuous greedy algorithm \citep{calinescu2011maximizing,bian16guaranteed}  can be used in principle to solve Problem~\eqref{eq:mainproblem}. Note that each update of continuous greedy requires computing the gradient of  $F$, i.e., $\nabla F(\bbx):=\mathbb{E}[{\nabla \tilde{F} (\bbx,\bbz)}]$. However, if we only have access to the (computationally cheap) stochastic gradients ${\nabla \tilde{F} (\bbx,\bbz)}$, then the continuous greedy method will not be directly usable \citep{hassani2017gradient}.  
This limitation is due to the non-vanishing variance of gradient approximations. To resolve this issue, we use the gradient averaging technique in Section~\ref{sec:SFO}. As in SFW, we define the estimated gradient $\bbd_t$ by the recursion  
\begin{equation} \label{eq:von1}
\bbd_t = (1-\rho_t) \bbd_{t-1} + \rho_t \nabla \tilde{F}(\bbx_t,\bbz_t),
\end{equation}
where $\rho_t$ is a positive stepsize and the initial vector is defined as $\bbd_0=\bb0$. We therefore define the  ascent direction $\bbv_t$ of our proposed SCG method as follows
%
%
\begin{equation}\label{eq:von2}
\bbv_t =\argmax_{\bbv\in \ccalC} \{ \bbd_t^T\bbv \},
\end{equation}
which is a linear objective maximization over the convex set $\ccalC$. Indeed, if instead of the gradient estimate $\bbd_t$ we use the exact gradient $\nabla F(\bbx_t)$ for the updates in~\eqref{eq:von2}, the continuous greedy update will be recovered. Here, as in continuous greedy, the initial decision vector is the null vector, $\bbx_0=\bb0$. Further, the stepsize for updating the iterates is equal to $1/T$, and the variable $\bbx_t$ is updated as
\begin{equation}\label{eq:von3}
\bbx_{t+1} =  \bbx_{t} + \frac{1}{T} \bbv_t.
\end{equation}
The stepsize $1/T$ and the initialization $\bbx_0=\bb0$ ensure that after $T$ iterations the variable $\bbx_T$ ends up in the convex set $\ccalC$. We would like to highlight that the convex body $\ccalC$ may not be down-closed or contain $\bb0$. Nonetheless, the final iterate $\bbx_T$ returned by SCG will be a feasible point in $\ccalC$. 
The steps of the proposed SCG method are outlined in Algorithm~\ref{algo_SCGGA}. Note that the major difference between SFW in Algorithm~\ref{algo_SFW}  and SCG in Algorithm~\ref{algo_SCGGA} is in Step 4 where the variable $\bbx_{t+1}$ is computed. In SFW, $\bbx_{t+1}$ is a convex combination of $\bbx_t$ and $\bbv_t$, while in SCG $\bbx_{t+1}$ is computed by moving from $\bbx_t$ towards the direction $\bbv_t$ with the stepsize $1/T$.

%
\begin{algorithm}[tb]
\caption{\alg (SCG)}\label{algo_SCGGA} 
\begin{algorithmic}[1] 
{\REQUIRE Stepsizes $\rho_t>0$. Initialize $\bbd_0=\bbx_0=\bb0$
\FOR {$t=1,2,\ldots, T$}
   \STATE Compute $\bbd_t = (1-\rho_t) \bbd_{t-1} + \rho_t \nabla \tilde{F}(\bbx_t,\bbz_t)$;
   \STATE Compute $\bbv_t =\argmax_{\bbv\in \ccalC} \{ \bbd_t^T\bbv \}$;
   \STATE Update the variable $\bbx_{t+1} =\bbx_{t} + \frac{1}{T} \bbv_t$;
\ENDFOR}
\end{algorithmic}\end{algorithm}

We proceed to study the convergence properties of our proposed SCG method for solving Problem~\eqref{eq:mainproblem}. To do so, we first assume that the following conditions hold.

\begin{assumption}\label{ass:bounded_set}
{The Euclidean norm of the elements in the constraint set  \ccalC are uniformly bounded, i.e., for all $\bbx \in \ccalC$ we can write}
\begin{equation}
\|\bbx\|\leq D.
\end{equation}
\end{assumption}
\begin{assumption}\label{ass:smoothness}
The function $F$ is DR-submodular and monotone. Further, its gradients are $L$-Lipschitz continuous over the set $\ccalX$, i.e., for all $\bbx,\bby \in \ccalX$
\begin{equation}
\| \nabla F(\bbx) -  \nabla F(\bby) \| \leq L \| \bbx - \bby \|.
\end{equation}
\end{assumption}
\begin{assumption}\label{ass:bounded_variance}
The variance of the unbiased stochastic gradients $\nabla \tilde{F}(\bbx,\bbz)$ is bounded above by $\sigma^2$, i.e., for any vector $\bbx\in\ccalX$ we can write 
\begin{equation}
\E{\|  \nabla \tilde{F}(\bbx,\bbz) - \nabla F(\bbx)  \|^2} \leq \sigma^2,
\end{equation}
where the expectation is with respect to the randomness of $\bbz \sim P$.
\end{assumption}


Due to the initialization step of  SCG (i.e., starting from $\bb0$) we need a bound on the furthest feasible solution from $\bb0$ that we can end up with; and such a bound is guaranteed by Assumption~\ref{ass:bounded_set}.
%
The condition in Assumption~\ref{ass:smoothness} ensures that the objective function $F$ is smooth. Note  again that $\nabla \tilde{F}(\bbx,\bbz)$ may or may not be Lipschitz continuous. 
%
Finally, the required condition in Assumption~\ref{ass:bounded_variance} guarantees that the variance of stochastic gradients $\nabla\tilde{F}(\bbx,\bbz)$ is bounded by a finite constant $\sigma^2<\infty$. Note that Assumptions \ref{ass:smoothness}-\ref{ass:bounded_variance} are stronger than Assumptions~\ref{ass:smoothness_convex}-\ref{ass:bounded_variance_convex} since they ensure smoothness and bounded gradients for all points $\bbx\in \ccalX$ and not only for the feasible points $\bbx\in \ccalC$.

To study the convergence of SCG, we first derive an upper bound for the expected error of gradient approximation (i.e., $\mathbb{E}[\|\nabla F(\bbx_t) - \bbd_t\|^2]$) in the following lemma.

\begin{lemma}\label{lemma:bound_on_grad_approx_sublinear}
Consider \alg (SCG)  outlined in Algorithm~\ref{algo_SCGGA}.  If  Assumptions \ref{ass:bounded_set}-\ref{ass:bounded_variance} are satisfied and $\rho_t=\frac{4}{(t+8)^{2/3}}$, then for $t=0,\dots,T$ we have
\begin{align}\label{eq:grad_error_bound_2}
\E{ {\|\nabla F(\bbx_{t}) - \bbd_{t}\|^2} }&\leq \frac{Q}{(t+9)^{2/3}},
\end{align}
where $Q:=\max \{ 4\|\nabla F(\bbx_{0}) - \bbd_{0}\|^2  , 16\sigma^2+2L^2 D^2 \}.$
\end{lemma}

\begin{proof}
Check Appendix~\ref{proof:lemma:bound_on_grad_approx_sublinear}.
\end{proof}

Let us now use the result of Lemma~\ref{lemma:bound_on_grad_approx_sublinear} to show that the sequence of iterates generated by SCG reaches  a $(1-1/e)$ approximation guarantee for Problem~\eqref{eq:mainproblem}.

\begin{theorem}\label{thm:optimal_bound_greedy}
Consider \alg (SCG)  outlined in Algorithm~\ref{algo_SCGGA}.  If  Assumptions \ref{ass:bounded_set}-\ref{ass:bounded_variance} are satisfied and $\rho_t=\frac{4}{(t+8)^{2/3}}$, then the expected objective function value for the iterates generated by SCG satisfies the inequality
\begin{align}\label{eq:claim_for_sto_greedy}
\E{F(\bbx_T)} \geq (1-1/e) \text{OPT}- \frac{15DQ^{1/2}}{T^{1/3}}-  \frac{LD^2}{2T},
\end{align}
where $\text{OPT}\:=\max_{\bbx \in \ccalC}\ F(\bbx)$ and $Q:=\max \{ 4\|\nabla F(\bbx_{0}) - \bbd_{0}\|^2  , 16\sigma^2+2L^2 D^2 \}$.
\end{theorem}

\begin{proof}
Check Appendix~\ref{proof:thm:optimal_bound_greedy}.
\end{proof}

The result in Theorem \ref{thm:optimal_bound_greedy} shows that the sequence of iterates generated by SCG, which only has access to a noisy unbiased estimate of the gradient at each iteration, is able to achieve the optimal approximation bound $(1-1/e)$, while the error term vanishes at a sublinear rate of $\mathcal{O}(T^{-1/3})$.

\subsection{Weak Submodularity}

In this section, we extend our results to a more general case where the expected objective function $F$ is weakly-submodular. A continuous function $F$ is $\gamma$-weakly DR-submodular if 
\begin{equation}
\gamma = \inf_{\bbx,\bby\in \ccalX,  \bbx\leq \bby} \ \inf_{i\in [n]}\ \frac{[\nabla F(\bbx)]_i}{[\nabla F(\bby)]_i},
\end{equation}
where $[\bba]_i$ denotes the $i$-th element of vector $\bba$. See \citep{eghbali2016designing} for related definitions. In the following theorem, we prove that the proposed SCG method achieves a $ (1-e^{-\gamma})$ approximation guarantee when the expected function $F$ is monotone and weakly DR-submodular with parameter~$\gamma$.

\begin{theorem}\label{thm:optimal_bound_greedy_weak}
Consider \alg (SCG)  outlined in Algorithm~\ref{algo_SCGGA}.  If  Assumptions \ref{ass:bounded_set}-\ref{ass:bounded_variance} are satisfied and the function $F$ is $\gamma$-weakly DR-submodular, then for  $\rho_t=\frac{4}{(t+8)^{2/3}}$ the expected objective function value of the iterates generated by SCG satisfies the inequality
\begin{align}\label{eq:claim_for_sto_greedy_weak}
\E{F(\bbx_T)} \geq (1-e^{-\gamma}) \text{OPT}- \frac{15DQ^{1/2}}{T^{1/3}}-  \frac{LD^2}{2T},
\end{align}
where $\text{OPT}\:=\max_{\bbx \in \ccalC}\ F(\bbx)$ and $Q:=\max \{ 4\|\nabla F(\bbx_{0}) - \bbd_{0}\|^2 , 16\sigma^2+2L^2 D^2 \}.$
\end{theorem}

\begin{proof}
Check Appendix~\ref{proof:thm:optimal_bound_greedy_weak}.
\end{proof}

\subsection{Non-monotone Continuous Submodular Maximization}
In this section, we aim to extend the results for the proposed \alg algorithm to maximize \textit{non-monotone} stochastic DR-submodular. The problem formulation of interest is similar to Problem \eqref{eq:mainproblem} except the facts that the objective function ${F}: \ccalX  \to \reals_{+} $ may not be monotone and the set $\ccalX$ is a bounded box. To be more precise, we aim to solve the program 
\begin{equation}\label{eq:mainproblem_nonmonotone}
\max_{\bbx \in \ccalC}\ F(\bbx) \doteq \max_{\bbx \in \ccalC}\ \mathbb{E}_{\bbz\sim P}[{\tilde{F} (\bbx,\bbz)}],
\end{equation}
where $F:\ccalX\to\reals$ is continuous DR-submodular, $\ccalX=\prod_{i=1}^n \ccalX_i$, each $\ccalX_i=[\underline{u}_i,\bar{u}_i]$ is a bounded interval, and the convex set $\ccalC$ is a subset of $\ccalX=[\underline{\bbu},\bar{\bbu}]$, where $\underline{\bbu}=[\underline{u}_1;\dots;\underline{u}_n]$ and $\bar{\bbu}=[\bar{u}_1;\dots;\bar{u}_n]$. In this section, we propose the first stochastic conditional gradient method for solving the stochastic non-monotone maximization problem in \eqref{eq:mainproblem_nonmonotone}. In this section, we further assume that the convex set $\ccalC$ is down-closed and $\bb0\in\ccalC$.

%
\begin{algorithm}[tb]
\caption{\algNM (NMSCG)}\label{algo_NMSCG} 
\begin{algorithmic}[1] 
{\REQUIRE Stepsizes $\rho_t>0$. Initialize $\bbd_0=\bbx_0=\bb0$
\FOR {$t=1,2,\ldots, T$}
   \STATE Compute $\bbd_t = (1-\rho_t) \bbd_{t-1} + \rho_t \nabla \tilde{F}(\bbx_t,\bbz_t)$;
   \STATE Compute $\bbv_t =\argmax_{\bbv\in \ccalC, \bbv\leq \bar{\bbu}-\bbx_t}  \{ \bbd_t^T\bbv \}$;
   \STATE Update the variable $\bbx_{t+1} =\bbx_{t} + \frac{1}{T} \bbv_t$;
\ENDFOR}
\end{algorithmic}\end{algorithm}

We introduce a variant of the \alg method that achieves a $(1/e)$-approximation guarantee for Problem \eqref{eq:mainproblem_nonmonotone}. The proposed \algNM (NMSCG) method is inspired by the unified measured
continuous greedy algorithm in \citep{feldman2011unified} and the Frank-Wolfe method in \citep{DBLP:conf/nips/BianL0B17} for non-monotone \textit{deterministic} continuous submodular maximization. The steps of NMSCG are summarized in Algorithm \ref{algo_NMSCG}. The stochastic gradient update (Step 2) and the update of the variable (Step 4) are identical to the ones for SCG in Section \ref{sec:scg}. The main difference between NMSCG and SCG is in the computation of the ascent direction $\bbv_t$. In particular, in NMSCG the ascent direction vector $\bbv_t$ is obtained by solving the linear program 
\begin{align}
\bbv_t =\argmax_{\bbv\in \ccalC, \bbv\leq \bar{\bbu}-\bbx_t}  \{ \bbd_t^T\bbv \},
\end{align}
which differs from \eqref{eq:von2} by having the extra constraint $\bbv\leq \bar{\bbu}-\bbx_t $. This extra condition is added to ensure that the solution does not grow aggressively, since in non-monotone case dramatic growth of the solution may lead to poor performance. In NMSCG, the initial variable is $\bbx_0=\bb0$, which is a legitimate initialization as we assume that the convex set $\ccalC$ is down-closed. In the following theorem, we establish a $1/e$- guarantee for NMSCG. 

\begin{theorem}\label{thm:non_monotone_bound}
Consider  \algNM (NMSCG)  outlined in Algorithm~\ref{algo_NMSCG} with the averaging parameter $\rho_t=4/(t+8)^{2/3}$.  If Assumptions~\ref{ass:bounded_set} and \ref{ass:bounded_variance} hold and the gradients $\nabla F$ are $L$-Lipschitz continuous and the convex set $\ccalC$ is down-closed, then the iterate $\bbx_T$ generated by NMSCG satisfies the inequality
\begin{align}\label{claim:non_monotone_bound}
\E{F(\bbx_{T})} \geq  e^{-1}F(\bbx^*)
-\frac{15DQ^{1/2}}{T^{1/3}}   -  \frac{LD^2}{2T},
\end{align}
where $Q:=\max \{ 4 \|\nabla F(\bbx_{0}) - \bbd_{0}\|^2  , 16\sigma^2+2L^2 D^2 \}.$
\end{theorem}

\begin{proof}
See Section \ref{proof:thm:non_monotone_bound}.
\end{proof}

The result in Theorem \ref{thm:non_monotone_bound} states that the sequence of iterates generated by NMSCG achieves a $((1/e)OPT-\eps)$ approximation guarantee after $\mathcal{O}(1/\eps^3)$ stochastic gradient computations. 

\section{Stochastic Discrete Submodular Maximization}\label{sec:Discrete_Submodular_Maximization}

Even though submodularity has been mainly studied in  discrete domains \citep{fujishige2005submodular}, many efficient methods for optimizing  such submodular set functions rely on continuous relaxations either through a multi-linear extension \citep{vondrak2008optimal} (for maximization) or Lovas extension \citep{lovasz1983submodular} (for minimization). In fact, Problem~\eqref{eq:mainproblem} has a discrete counterpart, recently considered in \citep{hassani2017gradient, karimi2017stochastic}:
\begin{equation}\label{eq:stochsub}
\max_{S\in \ccalI}f(S) \doteq \max_{S\in \ccalI} \mbE_{\bbz\sim P} [\tilde{f}(S, \bbz)],
\end{equation}
where the function $f:2^V  \rightarrow\reals_+$ is submodular, the functions $\tilde{f}:2^V \times \mathcal{Z} \rightarrow\reals_+$ are stochastic, $S$ is the optimization set variable defined over a ground set $V$, $\bbz\in \mathcal{Z}$ is the realization of a random variable $\bbZ$ drawn from the distribution $P$, and $\ccalI$ is a general matroid constraint.   Since $P$ is unknown, problem~\eqref{eq:stochsub} cannot be directly solved using the current state-of-the-art techniques. Instead, \cite{hassani2017gradient} showed that by lifting the problem to the continuous domain (via multi-linear relaxation) and using stochastic gradient methods on a continuous relaxation to reach a solution that is within a factor $(1/2)$ of the optimum. Contemporarily, \citep{karimi2017stochastic} used a concave relaxation technique to provide a $(1-1/e)$ approximation for the class of submodular coverage functions. Our work closes the gap for maximizing the stochastic submodular set maximization, namely, Problem~\eqref{eq:stochsub}, by providing the  first tight $(1-1/e)$ approximation guarantee for general monotone submodular set functions subject to a matroid constraint. For the non-monotone case, we obtain a $(1/e)$ approximation guarantee.
We, further, show that a $((1-e^{-c})/c)$-approximation guarantee can be achieved when the function $f$ has curvature $c$.

According to the results in Section~\ref{sec:Continuous_Submodular_Maximization}, SCG achieves in expectation a  $(1-1/e)$-optimal solution for Problem~\eqref{eq:mainproblem} when the function $F$ is monotone and DR-submodular, and NMSCG obtains $(1/e)$-optimal solution for the non-monotone case. The focus of this section is on extending these results into the discrete domain and showing that  SCG and NMSCG can be used to maximize a stochastic submodular \emph{set} function $f$, namely Problem~\eqref{eq:stochsub}, through the multilinear extension of the function $f$. To be more precise, in lieu of solving the program in~\eqref{eq:stochsub}
%
%
%
%
one can solve the continuous optimization problem
\begin{align}\label{eq:multilinear_program}
\max_{\bbx \in \ccalC} \ F(\bbx) \ = \ \max_{\bbx \in \ccalC} \ \sum_{S\subset V}f(S) \prod_{i\in S} x_i \prod_{j\notin S} (1-x_j),
\end{align}
where $F$ is the multilinear extension of the function $f$ 
and the convex set $\ccalC= \text{conv}\{1_{I} : I\in \ccalI \}$ is the matroid polytope \citep{calinescu2011maximizing} which is down-closed (note that in \eqref{eq:multilinear_program}, $x_i$ denotes the $i$-th element of the vector $\bbx$). The fractional solution of the program  \eqref{eq:multilinear_program} can then be rounded into a feasible discrete solution without any loss (in expectation) in objective  value by methods such as  randomized PIPAGE ROUNDING \citep{calinescu2011maximizing}. Note that randomized PIPAGE ROUNDING requires $O(n)$ computational complexity \citep{karimi2017stochastic} for the uniform matroid and $O(n^2)$ complexity for general matroids.

Indeed, the conventional continuous greedy algorithm is able to solve the program in \eqref{eq:multilinear_program}; however, each iteration of the method is computationally costly due to gradient $\nabla F(\bbx)$ evaluations. Instead, \cite{feldman2011unified} and \cite{calinescu2011maximizing} suggested approximating the gradient using a sufficient number of samples from $f$.  This mechanism still requires access to the set function $f$ multiple times at each iteration, and hence is not efficient for solving Problem~\eqref{eq:stochsub}. The idea is then to use a stochastic (unbiased) estimate for the gradient  $\nabla F$.  In the following remark, we provide a method to compute an unbiased estimate of the gradient using $n$ samples from $\tilde{f}(S_i, \bbz)$, where $\bbz \sim P$ and $S_i$'s, $i=1, \cdots, n$, are carefully chosen sets.


\begin{remark}\label{sample_remark}
(Constructing an Unbiased Estimator of the Gradient in Multilinear Extensions) Recall that $f(S)=\mbE_{\bbz\sim P} [\tilde{f}(S, \bbz)]$. In terms of the multilinear extensions, we obtain $F(\bbx) = \mbE_{\bbz\sim P} [\tilde{F}(\bbx, \bbz)]$, where $F$ and $ \tilde{F}$ denote the multilinear extension of $f$ and $\tilde{f}$, respectively. So $\nabla \tilde{F}(\bbx, \bbz)$ is an unbiased estimator of $\nabla F(\bbx)$ when $\bbz\sim P$. Note that finding the gradient of $\tilde{F}$ may not be easy as it contains exponentially many terms. Instead, we can provide computationally cheap unbiased estimators for $\nabla \tilde{F}(\bbx,\bbz)$.
%
%
%
%
It can easily be shown that 
\begin{equation}
\frac{\partial \tilde{F}}{\partial x_i} = \tilde{F}(\bbx,\bbz; x_i \leftarrow 1) - \tilde{F}(\bbx,\bbz; x_i \leftarrow 0).
\end{equation}
where for example by $(\bbx; x_i \leftarrow 1)$ we mean a vector which has value $1$ on its $i$-th coordinate and is equal to $\bbx$ elsewhere. To create an unbiased estimator for $\frac{\partial \tilde{F}}{\partial x_i} $ at a point $\bbx$ with realization $\bbz$ we can simply sample a set $S$ by including each element in it independently with probability $x_i$ and use $\tilde{f}(S \cup \{i\},\bbz) - \tilde{f}(S \setminus \{i\},\bbz)$ as an unbiased estimator for the $i$-th partial derivative of $\tilde{F}$. We can sample one single set $S$ and use the above trick for all the coordinates.  This involves $n$ function computations for $\tilde{f}$. 
\end{remark}

Indeed, the stochastic gradient ascent method proposed by \cite{hassani2017gradient} can be used to solve the multilinear extension problem in \eqref{eq:multilinear_program} using unbiased estimates of the gradient at each iteration. However, the stochastic gradient ascent method fails to achieve the optimal $(1-1/e)$ approximation. Further, the work of  \cite{karimi2017stochastic} achieves a $(1-1/e)$ approximation solution only when each $\tilde{f}(\cdot, \bbz)$ is a coverage function. Here, we show that SCG achieves the first $(1-1/e)$ tight approximation guarantee for the discrete stochastic submodular Problem~\eqref{eq:stochsub}. More precisely, we show  that SCG finds a solution for \eqref{eq:multilinear_program}, with an expected function value that is at least $(1-1/e)\text{OPT} -\epsilon$, in $\mathcal{O}(1/\epsilon^3)$ iterations.   
To do so, we first show in the following lemma  that the difference between any coordinates of gradients of two consecutive iterates generated by SCG, i.e., $\nabla_j F(\bbx_{t+1})-\nabla_j F(\bbx_{t})$ for $j\in\{1,\dots,n\}$, is bounded by $\|\bbx_{t+1}-\bbx_{t}\|$ multiplied by a factor which is independent of the problem dimension $n$.

\begin{lemma}\label{lemma:lip_constant}
Consider \alg (SCG)  outlined in Algorithm~\ref{algo_SCGGA} with iterates $\bbx_t$, and recall the definition of the multilinear extension function $F$ in~\eqref{eq:multilinear_program}. If we define $r$ as the rank of the matroid $\ccalI$ and $m_f \triangleq \max_{i \in \{1, \cdots, n\}} f(i)$, then the following
\begin{align}\label{claim:lip_constant}
\left|\nabla_{j} F(\bbx_{t+1}) -\nabla_{j} F(\bbx_t) \right| \leq m_f\sqrt{r} \|\bbx_{t+1}-\bbx_{t}\|,
\end{align}
holds for $j=1,\dots,n$.
\end{lemma}

\begin{proof}
See Appendix \ref{proof:lemma:lip_constant}.
\end{proof}

The result in Lemma \ref{lemma:lip_constant} states that in an \textit{ascent direction of SCG}, the gradient is $m_f\sqrt{r}$-Lipschitz continuous. 
Here, $m_f$ is the maximum marginal value of the function $f$ and $r$ is the rank of the matroid. 

Let us now explain how the variance of the stochastic gradients of $F$ relates to the variance of the marginal values of $f$. Recall that the stochastic function $\tilde{F}$ is a multilinear extension of the stochastic set function $\tilde{f}$, and it can be shown that
\begin{equation}
\nabla_j \tilde{F}(\bbx,\bbz) = \tilde{F}(\bbx,\bbz; x_j=1) - \tilde{F}(\bbx,\bbz; x_j=0).
\end{equation}
 Hence, from 
submodularity we have $\nabla_j \tilde{F}(\bbx,\bbz) \leq \tilde{f}(\{j\},\bbz)$. Using this simple fact we can deduce that  
\begin{equation}\label{var_bound_multi}
\E{\|  \nabla \tilde{F}(\bbx,\bbz) - \nabla F(\bbx)  \|^2} \leq n\max_{j  \in [n]} \mathbb{E}[ \tilde{f}(\{j\}, \mathbf{z})^2 ].
\end{equation}

Using the result of Lemma~\ref{lemma:lip_constant}, the expression in \eqref{var_bound_multi}, and a coordinate-wise analysis, the bounds in Theorem~\ref{thm:optimal_bound_greedy} can be improved and specified for the case of multilinear extension maximization problem as we show in the following theorem.

\begin{theorem}\label{thm:multi_linear_extenstion_thm}
Consider \alg (SCG)  outlined in Algorithm~\ref{algo_SCGGA}. Recall the definition of the multilinear extension function $F$ in~\eqref{eq:multilinear_program} and the definitions of $r$ and $m_f$ in Lemma \ref{lemma:lip_constant}. Further, set the averaging parameter as $\rho_t=4/(t+8)^{2/3}$.  If Assumption~\ref{ass:bounded_set} holds and the function $f$ is monotone and submodular, then the iterate $\bbx_T$ generated by SCG satisfies the inequality
\begin{align}\label{eq:claim_for_sto_greedy_multi_linear}
\E{F(\bbx_T)} \geq (1-1/e) OPT- \frac{15DK}{T^{1/3}}
{-\frac{m_f{r}D^2}{2T},}
\end{align}
where $K:=\max \{ 2\|\nabla F(\bbx_{0}) - \bbd_{0}\| , 2\sqrt{n}\sqrt{\max_{j  \in [n]} \mathbb{E}[ \tilde{f}(\{j\}, \mathbf{z})^2 ]}+\sqrt{3r}m_f D \}$ and $OPT$ it the optimal value of Problem~\eqref{eq:multilinear_program}.
 \end{theorem}

\begin{proof}
The proof is similar to the proof of Theorem~\ref{thm:optimal_bound_greedy}. For more details, check Appendix~\ref{proof:thm:multi_linear_extenstion_thm}.\end{proof}

The result of Theorem \ref{thm:multi_linear_extenstion_thm} indicates that the sequence of iterates generated by SCG achieves a $(1-1/e) \text{OPT} - \eps$ approximation guarantee. Note that the constants on the right hand side of \eqref{eq:claim_for_sto_greedy_multi_linear} are independent of $n$, except $K$ that is at most proportional to $\sqrt{n}$. As a result, we have the following guarantee for SCG in the case of multilinear functions.

\begin{corollary}\label{main_corollary}
Consider \alg (SCG)  outlined in Algorithm~\ref{algo_SCGGA}. Suppose the conditions in Theorem \ref{thm:multi_linear_extenstion_thm} are satisfied. Then, the sequence of iterates generated by SCG achieves a $(1-1/e)OPT - \epsilon$ solution after $\mathcal{O}({n^{3/2}}/{\eps^3})$ iterations. As a consequence, maximizing a stochastic Submodular set function with SCG requires $\mathcal{O}({n^{5/2}}/{\eps^3})$ evaluations of the function $\tilde{f}$ in order to reach a $(1-1/e)OPT -\epsilon$ solution. 
\end{corollary}

\begin{proof}
According to the result in Theorem \ref{thm:multi_linear_extenstion_thm}, SCG reaches a $(1-1/e)OPT-\mathcal{O}(n^{1/2}/T^{1/3})$ solution after $T$ iterations. Therefore, to achieve a $((1-1/e)OPT-\epsilon)$ approximation, $\mathcal{O}(n^{3/2}/\eps^{3})$ iterations are required. Since each iteration requires access to an unbiased estimator of the gradient $\nabla F(\bbx)$ and it can be computed by $n$ samples from $\tilde{f}(S_i, \bbz)$ (Remark~\ref{sample_remark}), then the total number of calls to the function $\tilde{f}$ to reach a $(1-1/e)OPT-\epsilon$ solution is of order $\mathcal{O}({n}^{5/2}/{\eps^3})$ for the SCG method.
\end{proof}

The result in Corollary \ref{main_corollary} shows that after at most $\mathcal{O}({n^{5/2}}/{\eps^3})$ function evaluations of the stochastic set function $\tilde{f}$ the iterates generated by SCG achieves a continuous solution $\bbx_t$ with an objective function value that satisfies $\E{F(\bbx_t)}\geq (1-1/e)OPT-\eps$ where $OPT$ is the optimal objective function value of Problem~\eqref{eq:multilinear_program}. Further, by using a lossless rounding scheme we can obtain a discrete set $\ccalS^\dag$ such that  $\E{f(S^\dag)}\geq (1-1/e)\max_{S\in \ccalI}f(S)-\eps$.

Indeed, by following similar steps we can extend the result for NMSCG to the discrete submodular maximization problem when the objective function is non-monotone and stochastic. We formally prove this claim in the following theorem. 

\begin{theorem}\label{thm:multi_linear_extenstion_thm_non_monotone_case}
Consider \algNM (NMSCG)  outlined in Algorithm~\ref{algo_NMSCG}. Recall the definition of the multilinear extension function $F$ in~\eqref{eq:multilinear_program} and the definitions of $r$ and $m_f$ in Lemma \ref{lemma:lip_constant}. Further, set the averaging parameter as $\rho_t=4/(t+8)^{2/3}$.  If Assumption~\ref{ass:bounded_set} holds and the function $f$ is non-monotone and submodular, then the iterate $\bbx_T$ generated by SCG satisfies the inequality
\begin{align}\label{eq:claim_for_sto_greedy_multi_linear_non_monotone_case}
\E{F(\bbx_T)} \geq (1/e) OPT- {\frac{15DK}{T^{1/3}}
 -\frac{m_f{r}D^2}{2T},}
\end{align}
where $K:=\max \{ 2 \|\nabla F(\bbx_{0}) - \bbd_{0}\|  , 2\sqrt{n}\sqrt{\max_{j  \in [n]} \mathbb{E}[ \tilde{f}(\{j\}, \mathbf{z})^2 ]}+\sqrt{3r}m_f D \}$ and $OPT$ it the optimal value of Problem~\eqref{eq:multilinear_program}.
\end{theorem}
 
\begin{proof}
The proof is similar to the proof of Theorem~\ref{thm:non_monotone_bound}. For more details, check Appendix~\ref{proof:thm:multi_linear_extenstion_thm_non_monotone_case}.
\end{proof}

\begin{corollary}\label{main_corollary_nonmonotone}
Consider \algNM (NMSCG)  outlined in Algorithm~\ref{algo_NMSCG}. Suppose the conditions in Theorem \ref{thm:multi_linear_extenstion_thm_non_monotone_case} are satisfied. Then, the sequence of iterates generated by NMSCG achieves a $(1/e)OPT - \epsilon$ solution after $\mathcal{O}({n^{3/2}}/{\eps^3})$ iterations. As a consequence, maximizing a stochastic Submodular set function with NMSCG requires $\mathcal{O}({n^{5/2}}/{\eps^3})$ evaluations of the function $\tilde{f}$ in order to reach a $(1/e)OPT -\epsilon$ solution. 
\end{corollary}


\subsection{Convergence Bounds Based on Curvature}
For the continuous greedy method it has been shown that if the submodular function $f$ has a curvature $c\in[0,1]$ the algorithm reaches a $(1/c)(1-e^{-c})$ approximation guarantee. In this section, we show that the same improvement can be established for SCG in the stochastic setting. To do so, we first formally define the curvature $c$ of a monotone submodular function $f$ as
\begin{equation}\label{eq:curvature_def}
c:= 1- \min_{S, j\notin S} \frac{f(S\cup\{j\}) - f(S)}{f(\{j\})}.
\end{equation}
Indeed, smaller curvature value $c$ leads to an easier submodular maximization problem, and, in this case, we should be able to achieve a tighter approximate solution. In the following theorem, we match this expectation and show that if $c<1$ the bound in \eqref{eq:claim_for_sto_greedy} can be improved. 

\begin{theorem}\label{thm:optimal_bound_curvature_greedy}
Consider the proposed \alg (SCG) defined in \eqref{eq:von1}-\eqref{eq:von3}. Further recall the definition of the function $f$ curvature $c$ in \eqref{eq:curvature_def}. If Assumption~\ref{ass:bounded_set} is satisfied and the function $f$ is monotone and submodular, then the expected objective function value for the iterate $\bbx_T$ generated by SCG satisfies the inequality
\begin{align}\label{eq:claim_for_sto_greedy_curvature}
\E{F(\bbx_T)} \geq \frac{1}{c}(1-e^{-c}) OPT -  \frac{15DK}{T^{1/3}}  - \frac{m_frD^2 }{2T},
\end{align}
where $K:=\max \{ 2 \|\nabla F(\bbx_{0}) - \bbd_{0}\|  , 2\sqrt{n}\sqrt{\max_{j  \in [n]} \mathbb{E}[ \tilde{f}(\{j\}, \mathbf{z})^2 ]}+\sqrt{3r}m_f D \}$.
\end{theorem}

\begin{proof}
Check Appendix~\ref{proof:thm:optimal_bound_curvature_greedy}.
\end{proof}

\begin{corollary}\label{main_corollary}
Consider \alg (SCG)  outlined in Algorithm~\ref{algo_SCGGA}. Suppose the conditions in Theorem \ref{thm:optimal_bound_curvature_greedy} are satisfied. Then, the sequence of iterates generated by SCG achieves a $((1-e^{-c})/c)OPT - \epsilon$ solution after $\mathcal{O}({n^{3/2}}/{\eps^3})$ iterations. As a consequence, maximizing a stochastic submodular set function with SCG requires $\mathcal{O}({n^{5/2}}/{\eps^3})$ function evaluations.
\end{corollary}

\section{Numerical Experiments}\label{sec:experiments}

In this section, we compare the performances of the proposed stochastic conditional gradient method with state-of-the-art algorithms in both convex and submodular settings.

\subsection{Convex Setting}

We first compare the proposed SFW algorithm and mini-batch FW for a stochastic quadratic program of the form \eqref{eq:convex_problem}. Then, we compare their performances in solving a matrix completion problem. In this section, by mini-batch FW we refer to a variant of FW that simply replaces gradients by a mini-batch of stochastic gradients.
\vspace{2mm}
 
\textbf{Quadratic Programming.} Consider a positive definite matrix $\bbA\in \mathbb{S}_{++}^n$, a vector $\bbb\in \reals^n$, a random variable $\bbz\in \reals^n$, and the random diagonal matrix $diag(\bbz)\in \reals^{n\times n}$ defined by $\bbz$. The function $F$ is defined as 
\begin{equation}\label{test_example_1}
F(\bbx)= \E{\tilde{F}(\bbx,\bbz)} =  \E{  \frac{1}{2}\bbx^T (\bbA+\text{diag}(\bbz)) \bbx + (\bbb+\bbz)^T\bbx}.
\end{equation}
We assume that each element of $\bbz$ is sampled from a normal distribution $\ccalN(0,\sigma)$. Therefore, the objective function can be simplified to $F(\bbx)=   \frac{1}{2}\bbx^T \bbA \bbx + \bbb^T\bbx$. Further, we assume that the set $\ccalC$ is defined as $\ccalC=\{ \bbx\in \reals^n \mid l\leq x_i \leq u\}$. Here, we assume that the distribution is unknown to the algorithm and at each iteration we only have access to the stochastic gradients $\nabla \tilde{F}(\bbx,\bbz)=  (\bbA+\text{diag}(\bbz)) \bbx + \bbb+\bbz$. 

In our experiments, we set the dimension of the problem to $n=5$ and the lower bound and upper bounds for the set $\ccalC$ to $l=10$ and $u=100$. We construct $\bbA$ and $\bbb$ in such a way that the optimal solution of the unconstrained set, namely $-\bbA^{-1}\bbb$, does not belong to the set $\ccalC$. 

Figure \ref{fig_cvx} demonstrates the suboptimality gap $F(\bbx_T)-F(\bbx^*)$ for the iterates generated by the proposed SFW method (with batch size $b=1$) as well as the naive stochastic implementation of FW with batch sizes $b=\{1,10,50\} $ for the cases that $T=\{100,200,400,800,1600,3200,6400,12800\}$. We further illustrates the performance of the (deterministic) FW as a benchmark. Indeed, to perform the update of FW we use the exact gradient $\bbA\bbx+\bbb$ at each iteration. Note that the optimal solution $\bbx^*$ and the optimal objective function value $F(\bbx^*)$ are pre-computed by solving the quadratic program $\min_{\bbx\in \ccalC} \frac{1}{2}\bbx^T \bbA \bbx + \bbb^T\bbx$. The left and right plots correspond to the cases that $\sigma=100$ and $\sigma=300$, respectively. In the left plot, which corresponds to the case that $\sigma=100$, we observe that our proposed SFW method performs similar to the FW algorithm, while it uses only a single noisy stochastic gradient per iteration. The vanilla mini-batch FW method with a single stochastic gradient evaluation ($b=1$) performs poorly. By increasing the size of the batch, the performance of the mini-batch FW improves, but it still underperforms SFW. In the right plot, which corresponds to the case with  a larger variance, naturally the gap between the deterministic FW method and the stochastic algorithms becomes more significant. In this case, we observe that mini-batch FW even with large batch size of $b=100$ is significantly worse than the proposed SFW method that only uses a single stochastic gradient per iteration. It is  worth mentioning that increasing the batch size in mini-batch FW accelerates convergence and improves convergence accuracy, however, the suboptimality saturates at some point. In contrast, SFW  converges to the optimal objective function at a sublinear rate in both small and large variance cases, matching our theory. 


\begin{figure*}[t!]
  \centering
        \begin{subfigure}{0.48\textwidth}
    \begin{center}
      \centerline{\includegraphics[width=1.05\columnwidth]{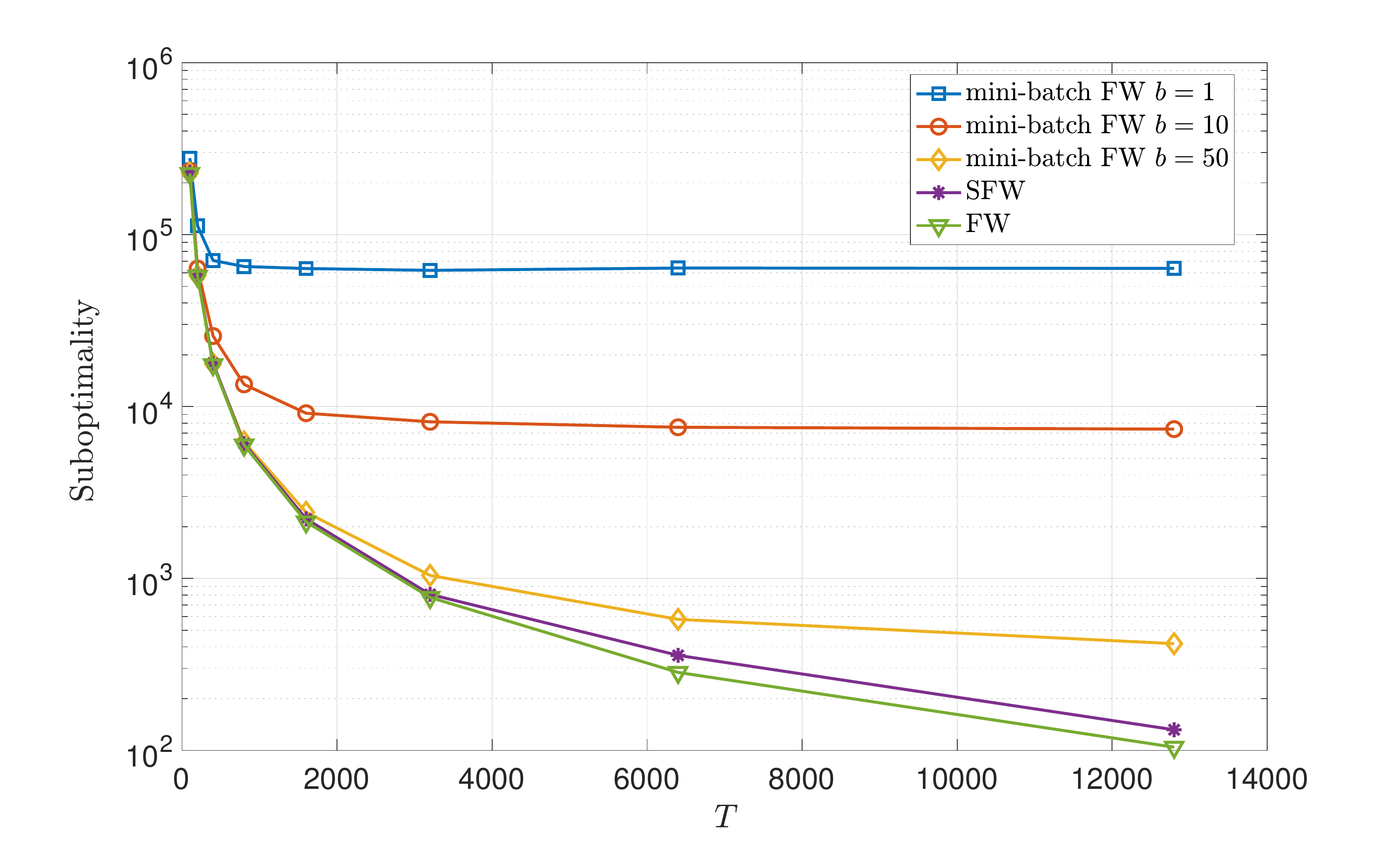}}
            \caption{$\sigma =100$ }
      \label{small_var}
    \end{center}
  \end{subfigure}
    \begin{subfigure}{0.48\textwidth}
    \begin{center}
      \centerline{\includegraphics[width=1.05\columnwidth]{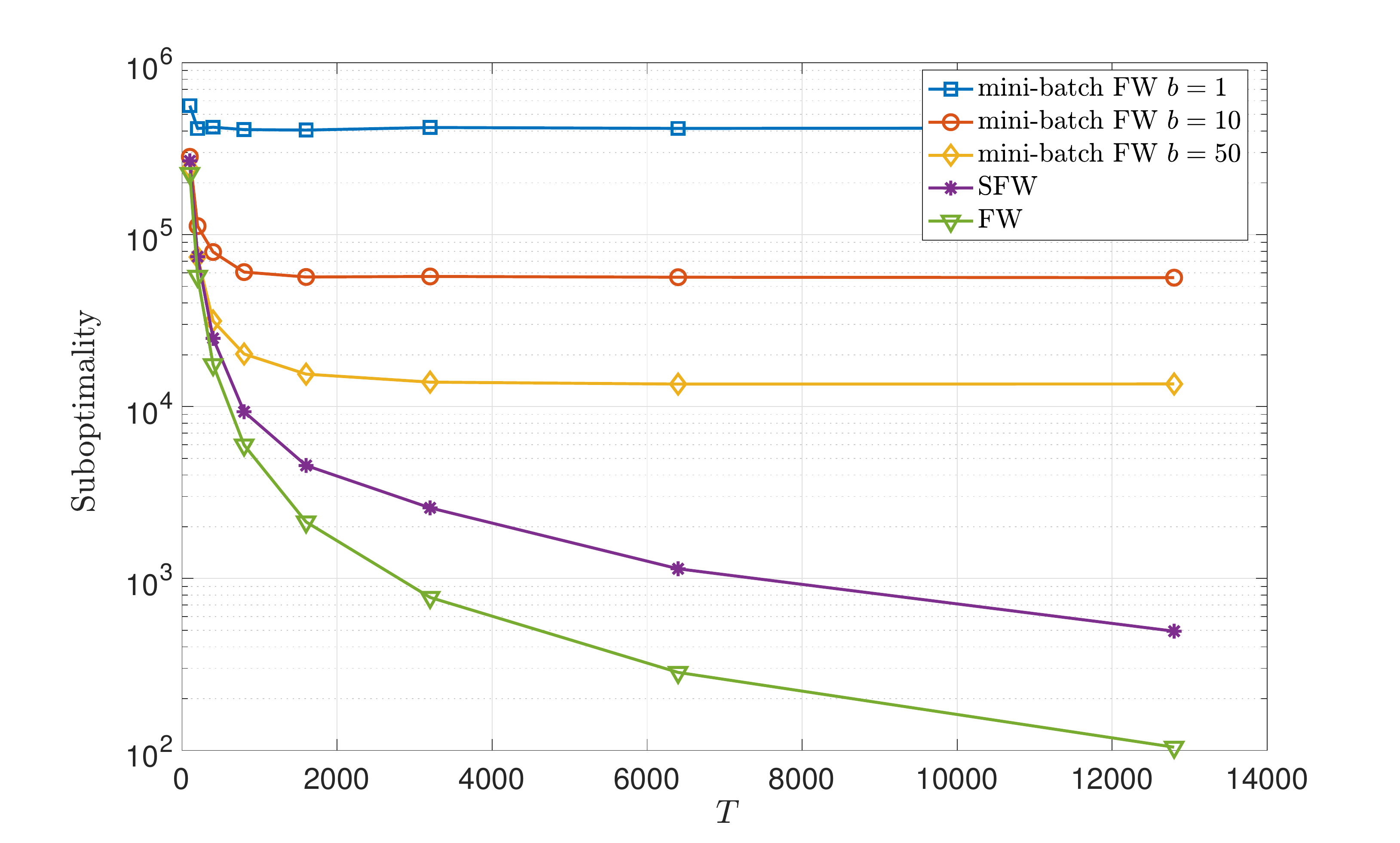}}
      \caption{$\sigma =300$}
      \label{large_var}
    \end{center}
  \end{subfigure}%
\vspace{0mm}
    \caption{Comparison of the performances of (deterministic) FW, mini-batch FW ($b=1$, $b=10$, $b=50$), and the proposed SFW method for the stochastic quadratic convex program defined in~\eqref{test_example_1}. The left and right plots correspond to the cases that the variance is $\sigma=100$ and $\sigma=300$, respectively. In both settings, the proposed SFW method that only uses a single stochastic gradient ($b=1$) performs close to the full-batch (deterministic) FW algorithm, and it outperforms the mini-batch FW method with even larger batch sizes ($b=\{1,10,50\}$). This comparison is in terms of number of iterations. If we compare the algorithms in terms of number of evaluated stochastic gradients the gap between SFW and mini-batch FW algorithms with batch sizes larger than $b=1$ will be even more significant as in this experiment SFW uses only a single stochastic gradient per iteration. 
    %
 \label{fig_cvx}}
 \end{figure*}

\vspace{2mm}
 
\textbf{Matrix Completion.} In this experiment, we study the performance of our proposed SFW algorithm in solving a matrix completion problem which is a canonical application of conditional gradient (Frank-Wolfe) type methods. We focus on a special case of matrix completion in which matrices are assumed to be symmetric. In particular, consider a symmetric matrix $\bbC \in \mathbb{S}^n$, where we only have access to a subset of its indices indicated by $\mathcal{O}$. Note that as $\bbC$ is symmetric, if $(i,j)$ is observed, i.e., $(i,j)\in \mathcal{O}$, then the pair $(j,i)$ is also known, i.e., $(j,i)\in \mathcal{O}$. Our goal is to find a symmetric positive semidefinite matrix $\bbX$ such that its elements in the set $\mathcal{O}$ are close to the ones for $\bbC$, while its nuclear rank is smaller than a threshold. In other words, we focus on the optimization problem 
\begin{align}\label{prob_matrix_completion}
&\min  f(\bbX) := \frac{1}{2} \sum_{(i,j)\in \mathcal {O}} \| \bbX_{ij} -\bbC_{ij}\|^2\nonumber\\
& \st \quad \|\bbX\| \succeq \bb0 , \quad \|\bbX\|_{*}\leq \alpha.
\end{align}
In our simulations, we set the dimension to $n=200$. We form the observation matrix as $\bbC=\hat{\bbX} + \bbE$. Here, $\hat{\bbX}$ is defined as $\hat{\bbX} =\bbW\bbW^T$ where $\bbW\in \reals^{n\times r}$ has independent normal distributed entries, and $\bbE$ is defined as $\bbE= \frac{1}{10} (\bbL +\bbL^T)$ where $\bbL\in \reals^{n\times n}$ has independent normal distributed entries. In our experiments, we set the rank to $r=10$ and the bound on the nuclear norm to $\alpha = \|\hat{\bbX}\|_{*}$, where $\|\hat{\bbX}\|_{*}$ is the nuclear norm of the matrix $\hat{\bbX}$. For settings that $\hat{\bbX}$ is not known in advance, one might use different choices of $\alpha$ and pick the one that performs better. We further define the set of observed entries $\mathcal{O}$ by sampling the elements of the upper triangular part of $\bbC$ uniformly at random with probability $0.8$. Therefore, the size of the set $\mathcal{O}$ is around $0.8\times 200^2=32,000$. In the realization that we use the set $\mathcal{O}$ has $31,884$ elements. 

To solve \eqref{prob_matrix_completion}, by using FW method, we need to solve the subproblem \citep[Chapter~7]{hazan2016introduction}
\begin{align}\label{subproblem_matrix_completion}
&\min \  \tr (\nabla f(\bbX_t)^T\bbV)  \nonumber\\
& \st \ \|\bbV\| \succeq \bb0 , \quad \|\bbV\|_{*}\leq \alpha.
\end{align}
where the gradient $\nabla f(\bbX_k)\in \reals^{n\times n}$ is defined as $\nabla f(\bbX_t)_{i,j}= \bbX_{ij} -\bbC_{ij} $ if $(i,j)\in \mathcal {O}$, and $\nabla f(\bbX_t)_{i,j}= 0$, otherwise. It can be shown that the solution to the subproblem \eqref{subproblem_matrix_completion} is given by 
\begin{align}
\bbV_t = 
\left\{
	\begin{array}{ll}
		\alpha \bbv_n \bbv_n^T  & \mbox{if } \lambda_n \geq 0, \\
		\bb0 & \mbox{if } \lambda_n < 0,
	\end{array}
\right.
\end{align}
where $\lambda_n$ is the smallest eigenvalue of the gradient  $\nabla f(\bbX_t)$ and $\bbv_n$ is its corresponding eigenvector.

\begin{figure*}[t!]
  \centering
        \begin{subfigure}{0.48\textwidth}
    \begin{center}
      \centerline{\includegraphics[width=1.05\columnwidth]{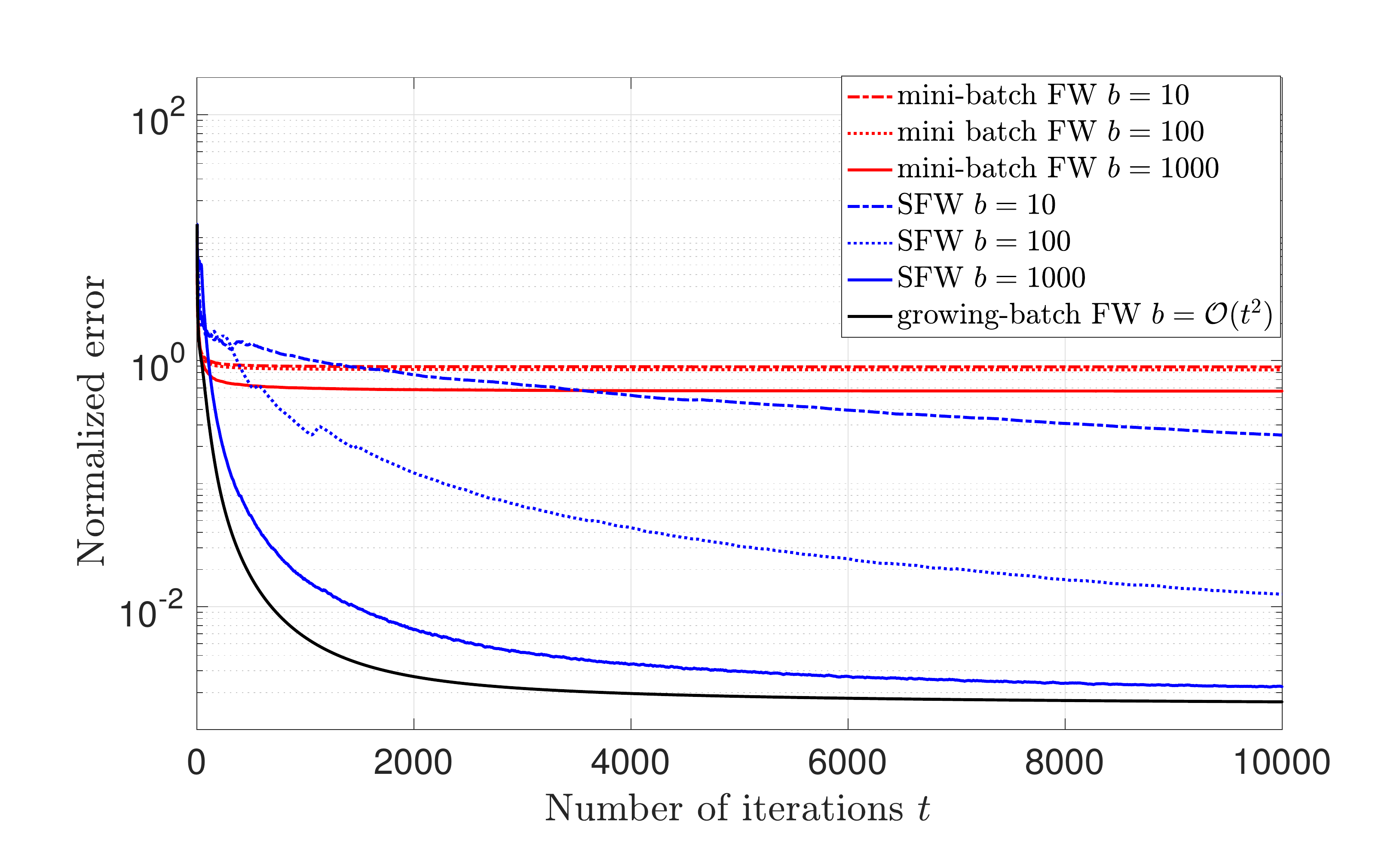}}
            \caption{Comparison in terms of number of iterations}
      \label{iter}
    \end{center}
  \end{subfigure}
    \begin{subfigure}{0.48\textwidth}
    \begin{center}
      \centerline{\includegraphics[width=1.05\columnwidth]{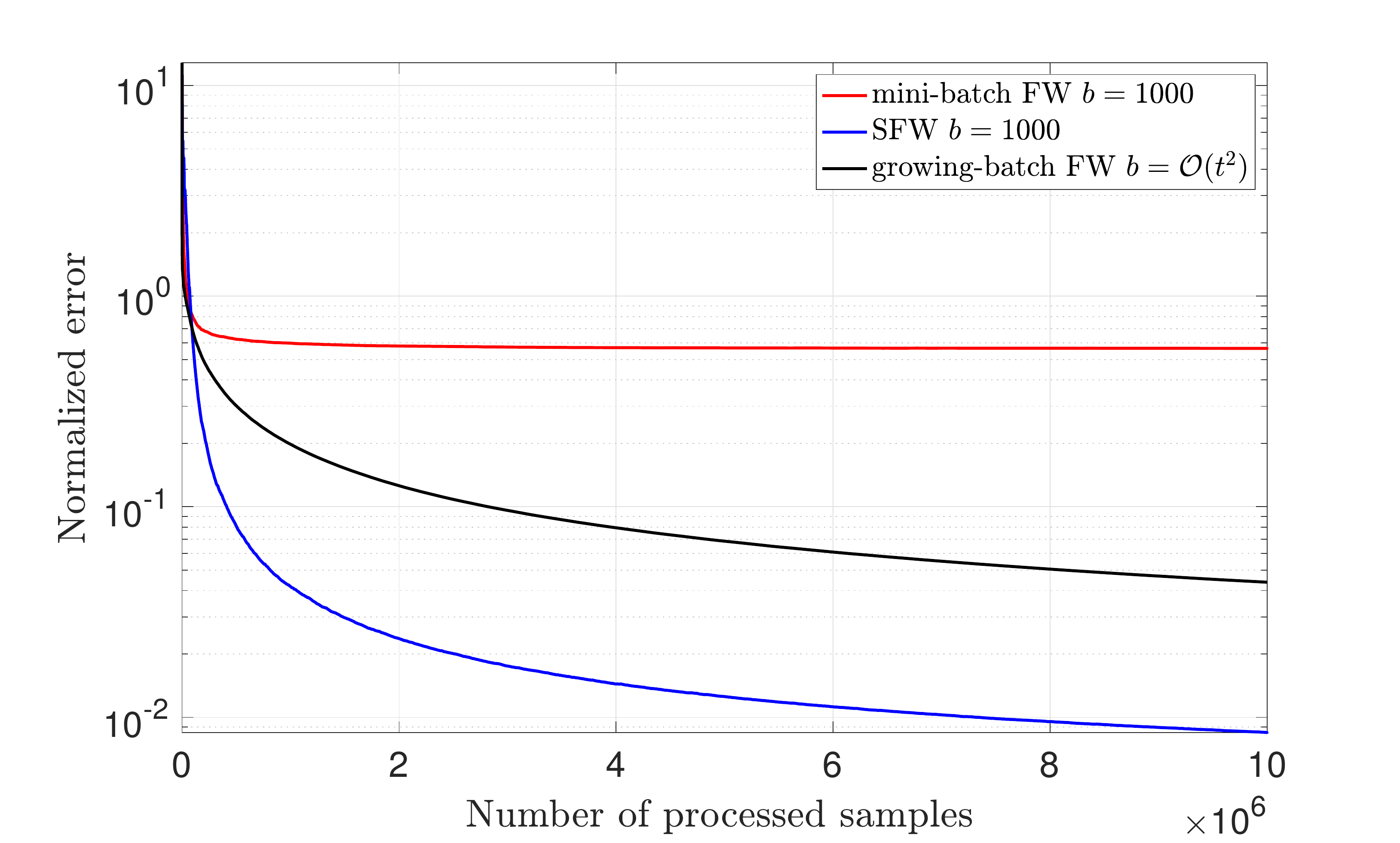}}
      \caption{Comparison in terms of number of samples}
      \label{sample}
    \end{center}
  \end{subfigure}%
%
%
\vspace{0mm}
    \caption{Comparison of the normalized error ${\sum_{(i,j)\in \mathcal {O}} \| \bbX_{ij} -\bbC_{ij}\|^2}/{\sum_{(i,j)\in \mathcal {O}} \|\bbC_{ij}\|^2}$ of the mini-batch FW algorithm, the growing mini-batch FW method proposed in \citep{DBLP:conf/icml/HazanL16}, and our proposed SFW method for solving the matrix completion problem defined in \eqref{prob_matrix_completion}. For all the choices of batch sizes considered here ($b=\{10,100, 1000\}$), the proposed SFW method outperforms the corresponding version of the mini-batch FW algorithm. As we observe, increasing the size of the mini-batch $b$ improves the convergence speed of both algorithms, but the mini-batch FW algorithm with $b=1000$ still underperforms our proposed method with $b=10$. In terms of number of iterations, the growing-batch FW method outperforms our proposed SFW method by using growing batches of size $b=\mathcal{O}(t^2)$, e.g., at iteration $t=10^4$ it uses $10^{8}$ samples; however, in terms of number of samples processed, SFW has the best performance as illustrated in the right plot.}
 \label{fig_matrix_comp}
 \end{figure*}

 Indeed, evaluation of the gradient  $\nabla f(\bbX_t)$ requires access to all observed elements in the set $\mathcal{O}$ which can be computationally costly. In such cases, one may use a subset of the set $\mathcal{O}$ as an unbiased estimate of the gradient. In our experiments, we consider (i) the mini-batch FW method that uses $b$ elements of  $\mathcal{O}$ to compute a stochastic approximation of   $\nabla f(\bbX_t)$, (ii) the growing mini-batch FW method proposed by \cite{DBLP:conf/icml/HazanL16} which uses a batch size of $b=\mathcal{O}(t^2)$ at step $t$, and (iii)~the proposed SFW method that uses the average of stochastic gradients over time as suggested in \eqref{eq:grad_averaging}.

 Figure \ref{fig_matrix_comp} illustrates the convergence paths of mini-batch FW and SFW for batch sizes of $b=\{10,100,1000\}$ as well as the convergence path of growing-batch FW in terms of both number of iterations and number of samples processed. Here, the normalized error is defined as $$\text{Normalized error} \ := \frac{\sum_{(i,j)\in \mathcal {O}} \| \bbX_{ij} -\bbC_{ij}\|^2}{\sum_{(i,j)\in \mathcal {O}} \|\bbC_{ij}\|^2}.$$ Note that the stepsize for all the three algorithms is $\gamma_t= \frac{1}{t+1}$ and the averaging parameter for our proposed SFW algorithm is $\rho_t = \frac{1}{(t+1)^{2/3}}$. As we observe in Figure \ref{iter}, even for a large batch size of $b=1000$, the mini-batch FW algorithm cannot obtain a normalized error better than $0.55$ after $10,000$ iterations. On the other hand, SFW with a small batch size of $b=10$ achieves an error of $0.25$ after $10,000$ iterations. Indeed, by increasing the batch size for SFW its performance becomes better. In particular, SFW with $b=1000$ achieves a normalized of error of $2.3\times 10^{-3}$ after $10,000$ iterations. We would like to highlight that even for the case that we set $b=1000$, we use less than $3.2\%$ of the observed elements at each iteration -- The number of observed elements is $31,884$. However, the best performance in terms of number of iterations belongs to the growing-batch FW method that uses a batch size of $b=\mathcal{O}(t^2)$. This is not surprising as the convergence rate of the growing-batch FW algorithm is $\mathcal{O}(1/t)$, while the convergence rate of  our proposed SFW is $\mathcal{O}(1/t^{1/3})$. However, the number of processed samples at each iteration by our method is much smaller than the one for growing-batch FW when $t$ becomes large. To be more precise, after $T$ iterations our method uses $T b$ samples, where $b$ is a constant much smaller than $T$, while the growing-batch FW uses $\mathcal{O}(T^3)$ samples. Therefore, to have a better comparison between these algorithms we also compare their normalized errors versus the number of processed samples which is a more accurate measure for comparing the sample complexity of these algorithms. As we observe in Figure \ref{sample}, the proposed SFW method outperforms both mini-batch FW and growing-batch FW algorithms when we compare their normalized errors versus number of samples used. Note that in theory, both growing-batch FW and SFW may require processing $\mathcal{O}(1/\eps^3)$ samples to reach a suboptimality gap of $f(\bbx_t)-f^*\leq \eps$, but in practice we observe that the proposed SFW method outperforms growing-batch FW.

\begin{figure*}[t!]
  \centering
\includegraphics[width=0.6\columnwidth]{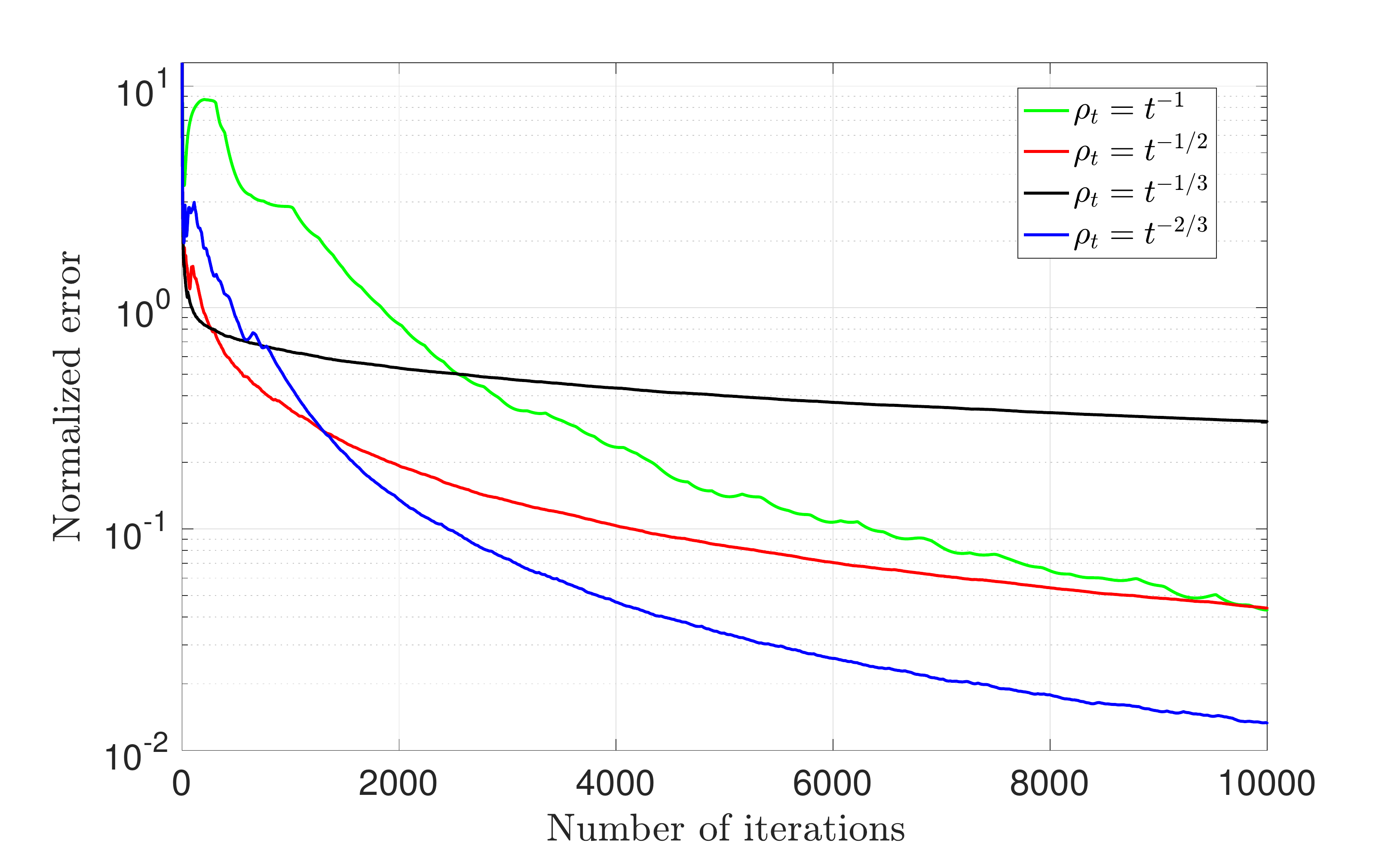}
\vspace{0mm}
    \caption{Comparison of the normalized error ${\sum_{(i,j)\in \mathcal {O}} \| \bbX_{ij} -\bbC_{ij}\|^2}/{\sum_{(i,j)\in \mathcal {O}} \|\bbC_{ij}\|^2}$ of the proposed SFW method with different choices of the averaging parameter $\rho_t$ for solving the matrix completion problem defined in \eqref{prob_matrix_completion}. In this case we set the other parameters to $b=100$ and $\gamma_t=1/(t+1)$. As suggested by our theory, the best performance belongs to the case that $\rho_t=\mathcal{O}(t^{-2/3})$.}
 \label{fig_effect_of_avg_par}
 \end{figure*}

We proceed to study the effect of the averaging parameter $\rho_t$ on the convergence of SFW. Our theoretical bound suggests that the best convergence guarantee is achieved when $\rho_t=\mathcal{O}(t^{-2/3})$. In this experiment we aim to check if this choice is reasonable relative to other possible sublinear rates. To do so, we compare the convergence paths of SFW with four different choices $\rho_t=\mathcal{O}(t^{-1/3})$, $\rho_t=\mathcal{O}(t^{-1/2})$, $\rho_t=\mathcal{O}(t^{-2/3})$, and $\rho_t=\mathcal{O}(t^{-1})$. As it can be observed in Figure \ref{fig_effect_of_avg_par}, the best performance among these four choices belongs to $\rho_t=\mathcal{O}(t^{-2/3})$ used in our theoretical results. We would like to highlight that this experiment does not prove that $\rho_t=\mathcal{O}(t^{-2/3})$ is the optimal choice.

\subsection{Submodular Setting}    
For the submodular setting, we consider a movie recommendation application \citep{serban17} consisting of $N$ users and $n$ movies. Each user $i$ has a user-specific utility function $f(\cdot,i)$ for evaluating sets of movies. The goal is to find a set of $k$ movies such that in expectation over users' preferences it provides the highest utility, i.e., $\max_{|S|\leq k}f(S)$, where $f(S)  \doteq\mathbb{E}_{i\sim P}[f(S,i)]$.  This is an instance of the (discrete) stochastic submodular maximization problem in \eqref{eq:stochsub}. For simplicity, we assume $f$ has the form of an empirical objective function, i.e. $f(S) = \frac{1}{N}\sum_{i=1}^N f(S,i)$. In other words, the distribution $P$ is assumed to be uniform over the set of users. The continuous counterpart of this problem is to consider the  the multilinear extension $F(\cdot,i)$ of any function $f(\cdot,i)$ and solve the problem in the continuous domain as follows. Let $F(\bbx) = \mathbb{E}_{i \sim \mathcal{D}} [F(\bbx,i)]$ for $x \in [0,1]^n$ and define the constraint set $\mathcal{C} =   \{ \bbx \in [0,1]^N: \sum_{i=1}^n x_i \leq k\}$. The discrete and continuous optimization formulations lead to the same optimal value \citep{calinescu2011maximizing}: $$\max_{S: |S| \leq k} f(S) = \max_{\bbx \in \mathcal{C}} F(\bbx).$$
 Therefore, by running SCG we can find a solution in the continuous domain that is at least $1-1/e$ approximation to the optimal value. By rounding that fractional solution (for instance via randomized Pipage rounding~\citep{calinescu2011maximizing}) we obtain a set whose utility is at least $1-1/e$ of the optimum solution set of size $k$. We note that randomized Pipage rounding does not need access to the value of $f$. We also remark that each iteration of SCG can be done very efficiently in $O(n)$ time (the linear program step reduces to finding the largest $k$ elements of a vector of length $n$). Therefore, this approach easily scales to big data scenarios where the size of the data set $N$ (e.g.~number of users) or the number of items $n$ (e.g.~number of movies) are very large.

\noindent In our experiments, we consider the following baselines: 
\begin{itemize}
\item[(i)] Stochastic Continuous Greedy (SCG): with $\rho_t = \frac 12 t^{-2/3}$ and mini-batch size $b$. The details for computing an unbiased estimator for the gradient of $F$ are given in Remark~\ref{sample_remark}.
  \item[(ii)] Stochastic Gradient Ascent (SGA) of \citep{hassani2017gradient}: with  stepsize $\mu_t = c/\sqrt{t}$ and mini-batch size $b$. 
 \item[(iii)] Frank-Wolfe (FW) variant of \citep{bian16guaranteed,calinescu2011maximizing}: with parameter $T$ for the total number of iterations and batch size $b$  (we further let $\alpha =1, \delta=0$, see Algorithm 1 in \citep{bian16guaranteed} or the continuous greedy method of \citep{calinescu2011maximizing} for more details).  		
 \item [(iv)] Batch-mode Greedy (Greedy): by running the vanilla greedy algorithm (in the discrete domain) in the following way. At each round of the algorithm (for selecting a new element), $b$ random users are picked and the function $f$ is estimated by the average over the $b$ selected users.  
\end{itemize}

\begin{figure*}[t!]
  \centering
    \begin{subfigure}{0.32\textwidth}
    \begin{center}
      \centerline{\includegraphics[width=1.12\columnwidth]{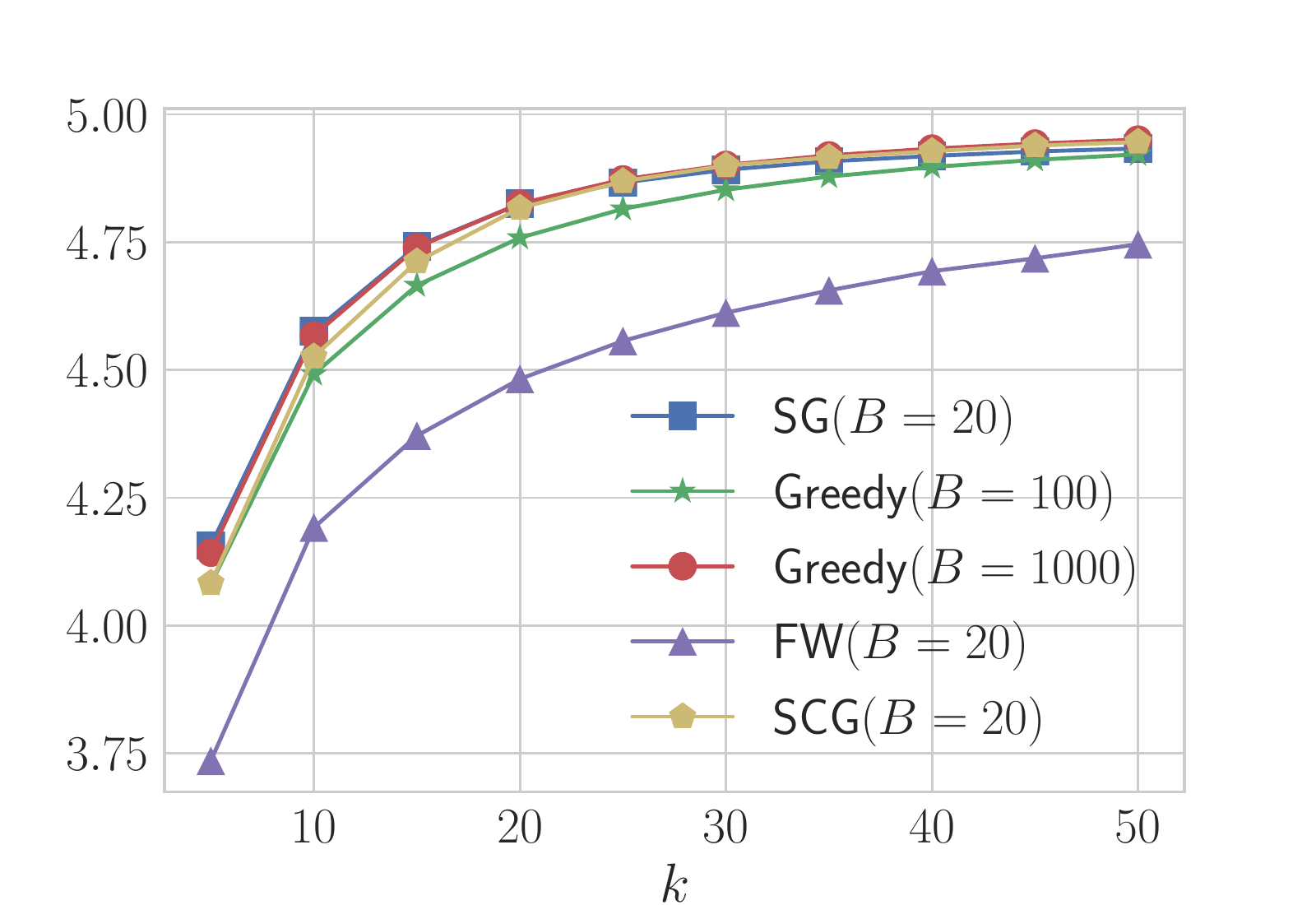}}
            \caption{facility location }
      \label{fig3}
    \end{center}
      \end{subfigure}
        \begin{subfigure}{0.32\textwidth}
    \begin{center}
      \centerline{\includegraphics[width=1.12\columnwidth]{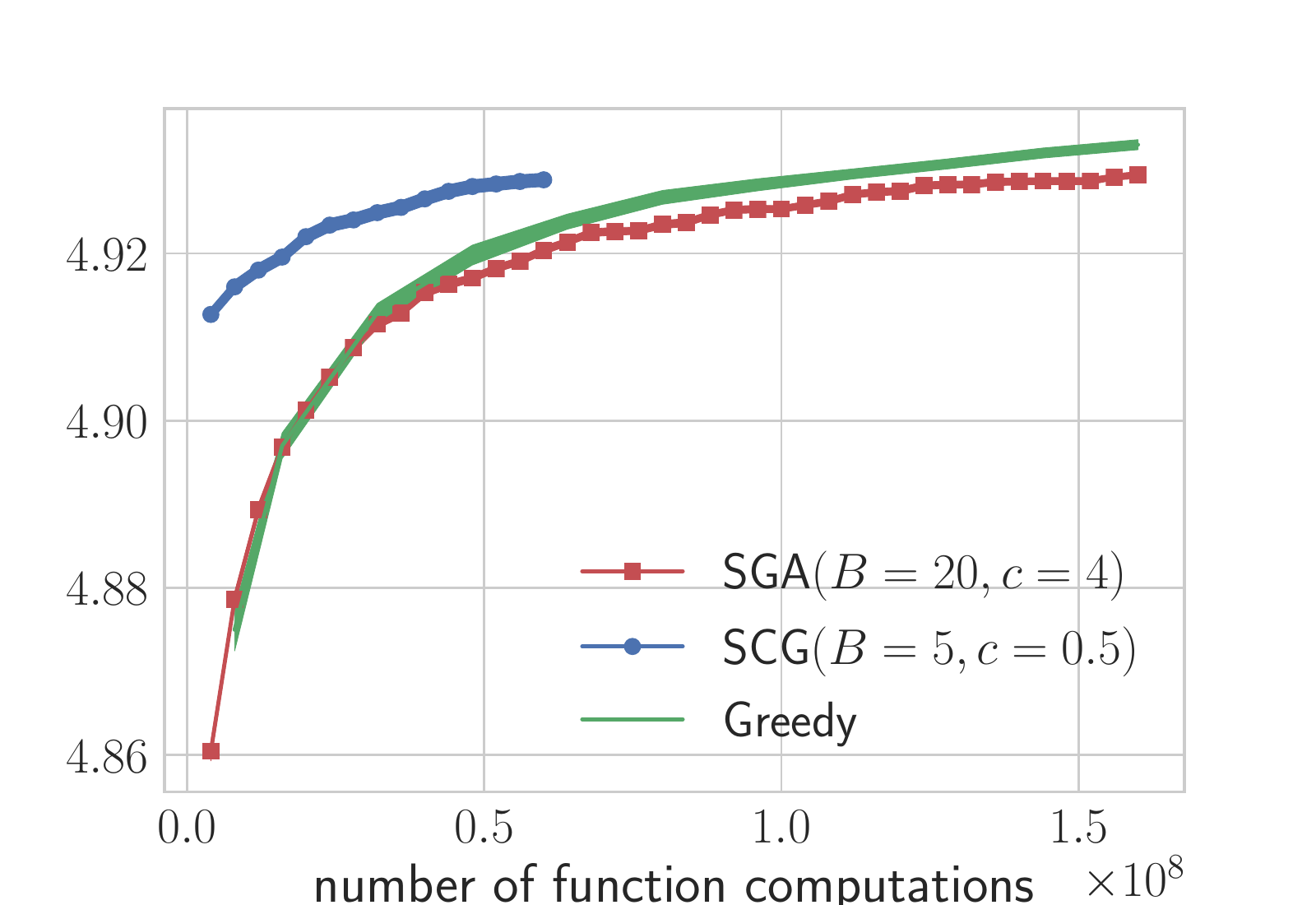}}
            \caption{facility location }
      \label{fig4}
    \end{center}
  \end{subfigure}
    \begin{subfigure}{0.32\textwidth}
    \begin{center}
      \centerline{\includegraphics[width=1.12\columnwidth]{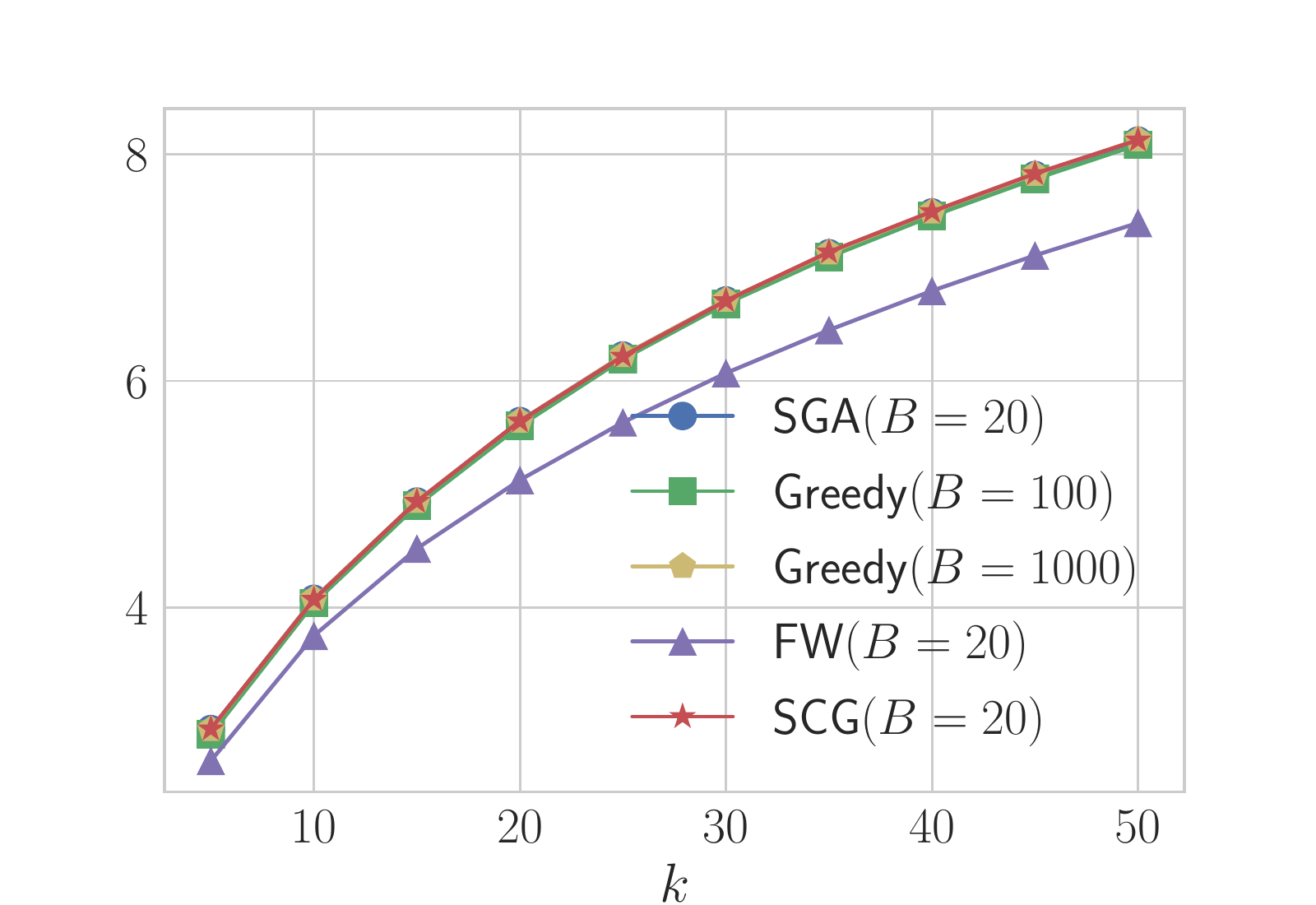}}
      \caption{concave-over-modular}
      \label{fig1}
    \end{center}
  \end{subfigure}%
    \caption{Comparison of the performances of SG, Greedy, FW, and SCG in a movie recommendation application. 
    Fig. \ref{fig3} illustrates the performance of the algorithms in terms of the facility-location objective value w.r.t. the cardinality constraint size~$k$ after $T = 2000$ iterations. Fig. \ref{fig4} compares the considered methods in terms of runtime (for a fixed $k=40$) by illustrating the facility location  objective function value vs. the number of (simple) function evaluations. 
     Fig. \ref{fig1} demonstrates the concave-over-modular objective function value vs. the size of the cardinality constraint $k$ after running the algorithms for $T = 2000$ iterations. 
 \label{fig-kolli}}
 \end{figure*}

To run the experiments we use the MovieLens data set. It consists of 1 million ratings (from 1 to 5) by $N=6041$ users for $n=4000$ movies.  Let $r_{i,j}$ denote the rating of user $i$ for movie $j$ (if such a rating does not exist we assign $r_{i,j}$ to 0). In our experiments, we consider two well motivated objective functions. The first one is called ``facility location "  where the valuation function by user $i$ is defined as  $f(S,i) = \max_{j\in S} r_{i,j}$. In words, the way user $i$ evaluates a set $S$ is by picking the highest rated movie in $S$.  Thus, the objective function is $$f_{\rm fac}(S) = \frac{1}{N}\sum_{i=1}^N  \max_{j\in S} r_{i,j}.$$ 

In our second experiment, we consider a  different user-specific valuation function which is a concave function composed with a modular function, i.e.,  $f(S,i) =(\sum_{j\in S} r_{i,j})^{1/2}.$ Again, by considering the uniform distribution over the set of users, we obtain $$f_{\rm con}(S) = \frac{1}{N} \sum_{i=1}^N  \Big(\sum_{j\in S} r_{i,j}\Big)^{1/2}.$$

Figure~\ref{fig-kolli} depicts the performance of different algorithms for the two proposed objective functions.  As Figures~\ref{fig3} and  \ref{fig1} show, the FW algorithm needs a higher mini-batch size to be comparable to SCG.  Note that a smaller batch size leads to less computational effort (under the same value for $b,T$, the computational complexity of FW, SGA, SCG is almost the same). {Figure~\ref{fig4} shows the performance of the algorithms with respect to the number of times the simple functions $f(\cdot,i)$ are evaluated. Note that the total number of (simple) function evaluations for SGA and SCG is $nbT$, where $T$ is the number of iterations.  Also, for Greedy the total number of evaluations is $nkb$.  This further shows that SCG has a better computational complexity requirement w.r.t. SGA as well as the Greedy algorithm.}

\section{Conclusion}\label{sec_conclusion}\label{sec:conclusion}
In this paper, we developed stochastic conditional gradient methods for solving convex minimization and submodular maximization problems. The main idea of the proposed methods in both domains was using a momentum term in the stochastic gradient approximation step to reduce the noise of the stochastic approximation. In particular, in the convex setting, we proposed a stochastic variant of the Frank-Wolfe algorithm called \algFW that achieves an $\eps$-suboptimal objective function value after at most $\mathcal{O}(1/\eps^3)$ iterations, if the expected objective function is smooth. The main advantage of the \algFW method (SFW) comparing to other stochastic conditional gradient methods for stochastic convex minimization is the required condition on the batch size. In particular, the batch size for SFW can be as small as $1$, while the state-of-the-art stochastic conditional gradient methods require a growing batch size.

In the submodular setting, we provided the first tight approximation guarantee for maximizing a stochastic monotone DR-submodular function subject to a general convex body constraint. We developed \alg that achieves a $[(1-1/e)\text{OPT} -\eps]$ guarantee (in expectation) with $\mathcal{O}{(1/\eps^3)}$ stochastic gradient computations. We further extended our result to the non-monotone case and introduced  \algNM that obtains a $(1/e)$ approximation guarantee. We also demonstrated that our continuous algorithm can be used to provide the first $(1-1/e)$ tight approximation guarantee for maximizing  a \textit{monotone but stochastic} submodular \textit{set} function subject to a general matroid constraint. Likewise, we provided the first $1/e$ approximation guarantee for maximizing a \textit{non-monotone} stochastic submodular \textit{set} function subject to a general matroid constraint. We believe that our results provide an important step towards unifying discrete and continuous submodular optimization in the stochastic setting. 

%
%
%
%
%
%

\section*{Acknowledgements}

This work was done while A. Mokhtari was visiting the Simons Institute for the Theory of Computing, and his work was partially supported by the DIMACS/Simons Collaboration on Bridging Continuous and Discrete Optimization through NSF grant \#CCF-1740425.  The work of A. Karbasi was supported by AFOSR YIP (FA9550-18-1-0160).

\appendix

\section{Proof of Lemma \ref{lemma:bound_on_gradient_error}}\label{app:proof_of_lemma:bound_on_gradient_error}
Use the definition $\bbd_t := (1-\rho_t) \bbd_{t-1} + \rho_t \nabla \tilde{F}(\bbx_t,\bbz_t)$  to write the difference $\|\nabla F(\bbx_t) - \bbd_t\|^2$ as
\begin{align}\label{proof:bound_on_grad_100}
\|\nabla F(\bbx_t) - \bbd_t\|^2  =\|\nabla F(\bbx_t) -(1-\rho_t) \bbd_{t-1} - \rho_t \nabla \tilde{F}(\bbx_t,\bbz_t)\|^2.
\end{align}
Add and subtract the term $(1-\rho_t)\nabla F(\bbx_{t-1})$ to the right hand side of \eqref{proof:bound_on_grad_100}, regroup the terms and expand the squared term to obtain
\begin{align}\label{proof:bound_on_grad_200}
&\|\nabla F(\bbx_t) - \bbd_t\|^2 \nonumber\\
& =\|\nabla F(\bbx_t)-(1-\rho_t)\nabla F(\bbx_{t-1})+(1-\rho_t)\nabla F(\bbx_{t-1}) -(1-\rho_t) \bbd_{t-1} - \rho_t \nabla \tilde{F}(\bbx_t,\bbz_t)\|^2 \nonumber\\
&=\|\rho_t(\nabla F(\bbx_t)-\nabla \tilde{F}(\bbx_t,\bbz_t))  +(1-\rho_t)(\nabla F(\bbx_t)-\nabla F(\bbx_{t-1}))+(1-\rho_t)(\nabla F(\bbx_{t-1}) - \bbd_{t-1} )\|^2 \nonumber\\
&=\rho_t^2\|\nabla F(\bbx_t)-\nabla \tilde{F}(\bbx_t,\bbz_t)\|^2  
   +(1-\rho_t)^2\|\nabla F(\bbx_t)-\nabla F(\bbx_{t-1})\|^2
   +(1-\rho_t)^2\|\nabla F(\bbx_{t-1}) - \bbd_{t-1} \|^2 \nonumber\\
 &  \qquad + 2\rho_t (1-\rho_t)(\nabla F(\bbx_t)-\nabla \tilde{F}(\bbx_t,\bbz_t))^T(\nabla F(\bbx_t)-\nabla F(\bbx_{t-1})) \nonumber\\
  &  \qquad + 2\rho_t (1-\rho_t)(\nabla F(\bbx_t)-\nabla \tilde{F}(\bbx_t,\bbz_t))^T(\nabla F(\bbx_{t-1}) - \bbd_{t-1} )\nonumber\\
  &  \qquad + 2 (1-\rho_t)^2(\nabla F(\bbx_t)-\nabla F(\bbx_{t-1}))^T(\nabla F(\bbx_{t-1}) - \bbd_{t-1} ).
\end{align}
Compute the conditional expectation $\E{ (.) \mid\ccalF_t}$ for both sides of \eqref{proof:bound_on_grad_200}, and use the fact that $\nabla \tilde{F}(\bbx_t,\bbz_t)$ is an unbiased estimator of the gradient $\nabla F(\bbx_t)$, i.e., $\E{ \nabla \tilde{F}(\bbx_t,\bbz_t) \mid\ccalF_t}=\nabla F(\bbx_t)$, to obtain
\begin{align}\label{proof:bound_on_grad_300}
&\E{\|\nabla F(\bbx_t) - \bbd_t\|^2\mid\ccalF_t} \nonumber\\
&=\rho_t^2\E{\|\nabla F(\bbx_t)-\nabla \tilde{F}(\bbx_t,\bbz_t)\|^2\mid\ccalF_t}  
   +(1-\rho_t)^2\|\nabla F(\bbx_t)-\nabla F(\bbx_{t-1})\|^2\nonumber\\
  &  \quad 
   +(1-\rho_t)^2\|\nabla F(\bbx_{t-1}) - \bbd_{t-1} \|^2 + 2 (1-\rho_t)^2(\nabla F(\bbx_t)-\nabla F(\bbx_{t-1}))^T(\nabla F(\bbx_{t-1}) - \bbd_{t-1} ).
   \end{align}
Use the condition in Assumption \ref{ass:bounded_variance} to replace $\E{\|\nabla F(\bbx_t)-\nabla \tilde{F}(\bbx_t,\bbz_t)\|^2\mid\ccalF_t} $ by its upper bound $\sigma^2$. Further, using Young's inequality to substitute the inner product $2\langle \nabla F(\bbx_t)-\nabla F(\bbx_{t-1}) , \nabla F(\bbx_{t-1}) - \bbd_{t-1} \rangle$ by the upper bounded $\beta_t \|\nabla F(\bbx_{t-1}) - \bbd_{t-1}\|^2+(1/\beta_t)\|\nabla F(\bbx_t)-\nabla F(\bbx_{t-1}) \|^2$ where $\beta_t>0$ is a free parameter. Applying these substitutions into \eqref{proof:bound_on_grad_300} yields 
\begin{align}\label{proof:bound_on_grad_400}
&\E{\|\nabla F(\bbx_t) - \bbd_t\|^2\mid\ccalF_t} \nonumber\\
  &\leq \rho_t^2\sigma^2
   +(1-\rho_t)^2 (1+{\beta_t}^{-1})\|\nabla F(\bbx_t)-\nabla F(\bbx_{t-1}) \|^2
   +(1-\rho_t)^2(1+\beta_t)\|\nabla F(\bbx_{t-1}) - \bbd_{t-1} \|^2.
   \end{align}
According to Assumption \ref{ass:smoothness}, the norm $\|\nabla F(\bbx_t)-\nabla F(\bbx_{t-1}) \|$ is bounded above by $L\|\bbx_t-\bbx_{t-1}\|$. In addition, the condition in Assumption \ref{ass:bounded_set} implies that $L\|\bbx_t-\bbx_{t-1}\|= L\frac{1}{T}\|\bbv_t-\bbx_t\| \leq \frac{1}{T} LD$. Therefore, we can replace $\|\nabla F(\bbx_t)-\nabla F(\bbx_{t-1}) \|$ in \eqref{proof:bound_on_grad_400} by its upper bound $\frac{1}{T} LD$ and write
\begin{align}\label{proof:bound_on_grad_500}
&\E{\|\nabla F(\bbx_t) - \bbd_t\|^2\mid\ccalF_t} \nonumber\\
  &\leq \rho_t^2\sigma^2
   +\gamma_{t}^2(1-\rho_t)^2 (1+{\beta_t}^{-1})L^2D^2
   +(1-\rho_t)^2(1+\beta_t)\|\nabla F(\bbx_{t-1}) - \bbd_{t-1} \|^2 .
   \end{align}
Since we assume that $\rho_t\leq1$ we can replace all the terms $(1-\rho_t)^2$ in \eqref{proof:bound_on_grad_500} by $(1-\rho_t)$. Applying this substitution into \eqref{proof:bound_on_grad_500} and setting $\beta:=\rho_t/2$ lead to the inequality 
\begin{align}\label{proof:bound_on_grad_600}
&\E{\|\nabla F(\bbx_t) - \bbd_t\|^2\mid\ccalF_t} \leq \rho_t^2\sigma^2
   +\gamma_{t}^2(1-\rho_t) (1+\frac{2}{\rho_t})L^2D^2
   +(1-\rho_t)(1+\frac{\rho_t}{2})\|\nabla F(\bbx_{t-1}) - \bbd_{t-1} \|^2 .
   \end{align}
Now using the inequalities $(1-\rho_t) (1+(2/\rho_t))\leq (2/\rho_t)$ and $(1-\rho_t)(1+({\rho_t}/{2}))\leq (1-\rho/2)$ we obtain
\begin{align}\label{proof:bound_on_grad_700}
\E{\|\nabla F(\bbx_t) - \bbd_t\|^2\mid\ccalF_t}\leq \rho_t^2\sigma^2
+\frac{2L^2D^2\gamma_{t}^2 }{\rho_t}
   +\left(1-\frac{\rho_t}{2}\right)\|\nabla F(\bbx_{t-1}) - \bbd_{t-1} \|^2,
\end{align}
and the claim in \eqref{eq:bound_on_gradient_error} follows.

\section{Proof of Lemma \ref{lemma:bound_on_suboptimality}}\label{app:proof_of_lemma:bound_on_suboptimality}

Based on the $L$-smoothness of the expected function $F$ we show that $F(\bbx_{t+1})$ is bounded above by
\begin{align}\label{eq:proof_lemma_convex_100}
F(\bbx_{t+1})
 \leq F(\bbx_{t}) +\nabla F(\bbx_{t})^T (\bbx_{t+1}-\bbx_t) +\frac{L}{2} \|\bbx_{t+1}-\bbx_t\|^2.
\end{align}
Replace the terms $\bbx_{t+1}-\bbx_t$ in  \eqref{eq:proof_lemma_convex_100} by $\gamma_{t+1}(\bbv_t-\bbx_t)$ and add and subtract the term $\gamma_{t+1}\bbd_t^T (\bbv_t-\bbx_t)$ to the right hand side of the resulted inequality to obtain
\begin{align}\label{eq:proof_lemma_convex_200}
F(\bbx_{t+1})
\leq 
 F(\bbx_{t}) +\gamma_{t+1} (\nabla F(\bbx_{t})-\bbd_t)^T (\bbv_t-\bbx_t) +\gamma_{t+1}\bbd_t^T (\bbv_t-\bbx_t) +\frac{L\gamma_{t+1}^2}{2} \|\bbv_{t}-\bbx_t\|^2.
\end{align}
Since $\langle \bbx^*, \bbd_t \rangle \geq \min_{v\in \ccalC} \{ \langle \bbv, \bbd_t \rangle\}  = \langle \bbv_t, \bbd_t \rangle$, we can replace the inner product $\langle \bbv, \bbd_t \rangle$ in \eqref{eq:proof_lemma_convex_200} by its upper bound $\langle \bbx^*, \bbd_t \rangle$. Applying this substitution leads to
\begin{align}\label{eq:proof_lemma_convex_300}
F(\bbx_{t+1})  \leq F(\bbx_{t}) +\gamma_{t+1} (\nabla F(\bbx_{t})-\bbd_t)^T (\bbv_t-\bbx_t) +\gamma_{t+1}\bbd_t^T (\bbx^*-\bbx_t) +\frac{L\gamma_{t+1}^2}{2} \|\bbv_{t}-\bbx_t\|^2.
\end{align}
Add and subtract $\gamma_{t+1}\nabla F(\bbx_{t})^T (\bbx^*-\bbx_t)$ to the right hand side of  \eqref{eq:proof_lemma_convex_300} and regroup the terms to obtain
\begin{align}\label{eq:proof_lemma_convex_400}
F(\bbx_{t+1})  \leq F(\bbx_{t}) +\gamma_{t+1} (\nabla F(\bbx_{t})-\bbd_t)^T (\bbv_t-\bbx^*)  +\gamma_{t+1}\nabla F(\bbx_{t})^T (\bbx^*-\bbx_t)+\frac{LD\gamma_{t+1}^2}{2} \|\bbv_{t}-\bbx_t\|^2.
\end{align}
Using the Cauchy-Schwarz inequality, we can show that the inner product $(\nabla F(\bbx_{t})-\bbd_t)^T (\bbv_t-\bbx^*)$ is bounded above by $\|\nabla F(\bbx_{t})-\bbd_t\| \|\bbv_t-\bbx^*\|$. Moreover, the inner product $\nabla F(\bbx_{t})^T (\bbx^*-\bbx_t)$ is upper bounded by $F(\bbx^*)-F(\bbx_{t})$ due to the convexity of the function $F$. Applying these substitutions into \eqref{eq:proof_lemma_convex_400} implies that 
\begin{align}\label{eq:proof_lemma_convex_500}
F(\bbx_{t+1})  
& \leq F(\bbx_{t}) +\gamma_{t+1} \|\nabla F(\bbx_{t})-\bbd_t\| \|\bbv_t-\bbx^*\|  -\gamma_{t+1}( F(\bbx_{t})-F(\bbx^*))+\frac{L\gamma_{t+1}^2}{2}\|\bbv_t-\bbx^*\| ^2
\nonumber\\
& \leq F(\bbx_{t}) +\gamma_{t+1} D\|\nabla F(\bbx_{t})-\bbd_t\| -\gamma_{t+1}( F(\bbx_{t})-F(\bbx^*))+\frac{LD^2\gamma_{t+1}^2}{2},
\end{align}
where the second inequality holds since $\|\bbv_t-\bbx^*\| \leq D$ according to Assumption \ref{ass:bounded_set}. Finally, subtract the optimal objective function value $F(\bbx^*)$ from both sides of \eqref{eq:proof_lemma_convex_500} and regroup the terms to obtain
\begin{align}\label{eq:proof_lemma_convex_600}
F(\bbx_{t+1}) -F(\bbx^*) \leq (1-\gamma_{t+1})(F(\bbx_{t}) -F(\bbx^*))+\gamma_{t+1} D\|\nabla F(\bbx_{t})-\bbd_t\| +\frac{LD^2\gamma_{t+1}^2}{2},
\end{align}
and the claim in \eqref{eq:bound_on_suboptimality} follows.

\section{Proof of Theorem \ref{thm:rate_convex}}\label{app:proof_of_thm:rate_convex}

Computing the expectation of both sides of  \eqref{eq:bound_on_suboptimality}  with respect to $\ccalF_0$ yields 
\begin{align}
\E{F(\bbx_{t+1}) -F(\bbx^*) } 
\leq (1-\gamma_{t+1})\E{(F(\bbx_{t}) -F(\bbx^*))}+\gamma_{t+1} D\E{\|\nabla F(\bbx_{t})-\bbd_t\|} +\frac{LD^2\gamma_{t+1}^2}{2}
\end{align}
Using Jensen's inequality, we can replace $\E{\|\nabla F(\bbx_{t})-\bbd_t\|}$ by its upper bound $\sqrt{\E{\|\nabla F(\bbx_{t})-\bbd_t\|^2}}$ to obtain
\begin{align}\label{eq:proof_of_rate_thm_cvx_100}
\E{F(\bbx_{t+1}) -F(\bbx^*) } 
\leq (1-\gamma_{t+1})\E{(F(\bbx_{t}) -F(\bbx^*))}+\gamma_{t+1} D\sqrt{\E{\|\nabla F(\bbx_{t})-\bbd_t\|^2}} +\frac{LD^2\gamma_{t+1}^2}{2}.
\end{align}
We proceed to analyze the rate that the sequence of expected gradient errors $\E{\|\nabla F(\bbx_{t})-\bbd_t\|^2}$ converges to zero. To do so, we first prove the following lemma which is an extension of Lemma~8 in \citep{mokhtari2015global}.

\begin{lemma}\label{lemma:induction_result}
Consider the scalars $b\geq 0$ and $c>1$. Let $\phi_t$ be a sequence of real numbers satisfying 
\begin{equation}\label{claim:rate_convg_assumption}
\phi_{t} \leq \left(1-\frac{c}{(t+t_0)^\alpha}\right) \phi_{t-1} + \frac{b}{(t+t_0)^{2\alpha}}, 
\end{equation} 
for some $\alpha \leq 1$ and $t_0\geq0$. Then, the sequence $a_t$ converges to zero at the following rate 
\begin{equation}\label{claim:rate_convg}
\phi_{t} \leq \frac{Q}{(t+t_0+1)^{\alpha}},
\end{equation}
where $Q:= \max\{ \phi_0t_0^\alpha, b/(c-1)\}$.
\end{lemma}

\begin{proof}
We prove the claim in \eqref{claim:rate_convg} by induction. First, note that $Q\geq \phi_0 t_0^{\alpha}$ and therefore $\phi_0\leq Q/(t_0)^{\alpha}$ and the base step of the induction holds true. Now assume that the condition in \eqref{claim:rate_convg} holds for $t=k$, i.e., 
\begin{equation}\label{proof:bound_on_grad_sublinear_50000}
\phi_k\leq \frac{Q}{(k+1+t_0)^{\alpha}}.
\end{equation}
The goal is to show that \eqref{claim:rate_convg} also holds for $t=k+1$. To do so, first set $t=k+1$ in the expression in \eqref{claim:rate_convg_assumption} to obtain 
\begin{align}\label{proof:bound_on_grad_sublinear_60000}
\phi_{k+1}
\leq \left(1-\frac{c}{(k+1+t_0)^{\alpha}}\right)\phi_{k}+ \frac{b}{(k+1+t_0)^{2\alpha}}.
\end{align}
According to the definition of $Q$, we know that $b\leq Q(c-1)$. Moreover, based on the induction hypothesis it holds that $\phi_{k}\leq \frac{Q}{(k+1+t_0)^{\alpha}}$. Using these inequalities and the expression in \eqref{proof:bound_on_grad_sublinear_60000} we can write 
\begin{align}\label{proof:bound_on_grad_sublinear_70000}
\phi_{k+1}
\leq \left(1-\frac{c}{(k+1+t_0)^{\alpha}}\right)\frac{Q}{(k+1+t_0)^{\alpha}}+ \frac{Q(c-1)}{(k+1+t_0)^{2\alpha}}.
\end{align}
Pulling out $\frac{Q}{(k+1+t_0)^{\alpha}}$ as a common factor and simplifying and
reordering terms it follows that \eqref{proof:bound_on_grad_sublinear_70000} is equivalent to
\begin{align}\label{proof:bound_on_grad_sublinear_80000}
\phi_{k+1}
&\leq Q\left(\frac{(k+1+t_0)^{\alpha}-1}{(k+1+t_0)^{2\alpha}}\right).
\end{align}
Based on the inequality 
\begin{align}\label{proof:bound_on_grad_sublinear_90000}
((k+1+t_0)^{\alpha}-1)((k+1+t_0)^{\alpha}+1) < (k+1+t_0)^{2\alpha},
\end{align}
the result in \eqref{proof:bound_on_grad_sublinear_80000} implies that 
\begin{align}\label{proof:bound_on_grad_sublinear_100000}
\phi_{k+1}
\leq \left(\frac{Q}{(k+1+t_0)^{\alpha}+1}\right).
\end{align}
Since $(k+1+t_0)^{2/3}+1 \geq (k+1+t_0+1)^{2/3}=(k+t_0+2)^{2/3}$, the result in \eqref{proof:bound_on_grad_sublinear_100000} implies that 
\begin{align}\label{proof:bound_on_grad_sublinear_110000}
\phi_{k+1}
\leq \left(\frac{Q}{(k+2+t_0)^{2/3}}\right),
\end{align}
and the induction step is complete. Therefore, the result in \eqref{claim:rate_convg} holds for all $t\geq0$. 
\end{proof}

Now using the result in Lemma \ref{lemma:induction_result} we can characterize the convergence of the sequence of expected errors $\E{\|\nabla F(\bbx_{t})-\bbd_t\|^2}$ to zero. To be more precise, compute the expectation of both sides of the result in \eqref{eq:bound_on_gradient_error} with respect to $\ccalF_0$ and set $\gamma_t=2/(t+8)$ and $\rho_t=4/(t+8)^{2/3}$ to obtain
\begin{align}\label{eq:almost_done}
\E{\|\nabla F(\bbx_t) - \bbd_t\|^2} \leq \left(1-\frac{2}{(t+8)^{2/3}}\right)\E{\|\nabla F(\bbx_{t-1}) - \bbd_{t-1}\|^2}+\frac{16\sigma^2+2L^2D^2 }{(t+8)^{4/3}}.
\end{align}
According to the result in Lemma \ref{lemma:induction_result}, the inequality in  \eqref{eq:almost_done} implies that 
\begin{equation}\label{eq:done}
\E{\|\nabla F(\bbx_t) - \bbd_t\|^2} \leq \frac{Q}{(t+9)^{2/3}},
\end{equation}
where $Q=\max\{4\|\nabla F(\bbx_0) - \bbd_0\|^2,16\sigma^2+2L^2D^2\}$. This result is achieved by setting $\phi_t=\E{\|\nabla F(\bbx_t) - \bbd_t\|^2}$, $\alpha=2/3$, $b=16\sigma^2+2L^2D^2$, $c=2$, and $t_0=8$ in Lemma  \ref{lemma:induction_result}.

Now we proceed by replacing the term $\E{\|\nabla F(\bbx_t) - \bbd_t\|^2}$ in \eqref{eq:proof_of_rate_thm_cvx_100} by its upper bound in \eqref{eq:done} and $\gamma_{t+1}$ by $2/(t+9)$ to write
\begin{align}\label{eq:proof_of_rate_thm_cvx_200}
\E{F(\bbx_{t+1}) -F(\bbx^*) } 
&\leq \left(1-\frac{2}{t+9}\right)\E{(F(\bbx_{t}) -F(\bbx^*))}+ \frac{2D\sqrt{Q}}{(t+9)^{4/3}} +\frac{2LD^2}{(t+9)^{2}}.
\end{align}
Note that we can write $(t+9)^2=(t+9)^{4/3} (t+9)^{2/3}\geq (t+9)^{4/3} 9^{2/3}\geq 4(t+s)^{4/3}$.Therefore, 
\begin{align}\label{eq:proof_of_rate_thm_cvx_300}
\E{F(\bbx_{t+1}) -F(\bbx^*) } 
&\leq \left(1-\frac{2}{t+9}\right)\E{(F(\bbx_{t}) -F(\bbx^*))}+ \frac{2D\sqrt{Q}+LD^2/2}{(t+9)^{4/3}} .
\end{align}

Now we proceed to prove by induction for $t\geq 0$ that 
\begin{align}\label{eq:claim:for_induction}
\E{F(\bbx_{t}) -F(\bbx^*) } \leq \frac{Q'}{(t+9)^{1/3}}
\end{align}
where $Q'=\max\{9^{1/3}(F(\bbx_{0}) -F(\bbx^*)),2D\sqrt{Q}+LD^2/2\}$. To do so, first note that $Q'\geq 9^{1/3}(F(\bbx_{0}) -F(\bbx^*))$, and, therefore, $F(\bbx_{0}) -F(\bbx^*)\leq Q'/9^{1/3}$. This leads to to the base of induction for $t=0$. Now assume that the inequality \eqref{eq:claim:for_induction} holds for $t=k$, i.e., $\E{F(\bbx_{k}) -F(\bbx^*) } \leq \frac{Q'}{(k+9)^{1/3}}$ and we aim to show that it also holds for $t=k+1$. 

To do so first set $t=k$ in \eqref{eq:proof_of_rate_thm_cvx_300} and replace $\E{F(\bbx_{k}) -F(\bbx^*) }$ by its upper bounds $\frac{Q'}{(k+9)^{1/3}}$ (as guaranteed by the hypothesis of the induction) to obtain 
\begin{align}\label{eq:proof_of_rate_thm_cvx_400}
\E{F(\bbx_{k+1}) -F(\bbx^*) } 
&\leq \left(1-\frac{2}{k+9}\right)\frac{Q'}{(k+9)^{1/3}}+ \frac{2D\sqrt{Q}+LD^2/2}{(k+9)^{4/3}} .
\end{align}
Now as in the proof of Lemma \ref{lemma:induction_result}, replace $2D\sqrt{Q}+LD^2/2$ by $Q'$ and simplify the terms to reach the inequality 
\begin{align}\label{eq:proof_of_rate_thm_cvx_400}
\E{F(\bbx_{k+1}) -F(\bbx^*) } 
&\leq  Q'\left(\frac{k+8}{(k+9)^{4/3}}\right)\leq \frac{Q'}{(k+10)^{1/3}},
\end{align}
and the induction is complete. Therefore, the inequality in \eqref{eq:claim:for_induction} holds for all $t\geq 0$.

\section{Proof of Theorem \ref{thm:almost_sure_cnvrg}}\label{app:proof_of_thm:almost_sure_cnvrg}

To prove the claim in \eqref{eq:almost_sure_cnvrg} we first show that the sum $\sum_{t=1}^\infty \rho_t\|\nabla F(\bbx_{t-1}) - \bbd_{t-1}\|^2$ is finite almost surely. To do so, we construct a supermartingale using the result in Lemma \ref{lemma:bound_on_gradient_error}. Let's define the stochastic process $\zeta_t$ as
\begin{equation}\label{alp}
\zeta_{t}:= \|\nabla F(\bbx_{t})-\bbd_{t}\|^2
+2L^2{D^2}\sum_{u=t+1}^{\infty}\frac{\gamma_{u}^2}{\rho_{u}} 
+\sigma^2\sum_{u=t+1}^{\infty}  \rho_{u}^2.
\end{equation}
Note that $\zeta_t$ is well defined because the sums on the the right hand side of \eqref{alp} are finite according to the hypotheses of Theorem~\ref{thm:almost_sure_cnvrg}. Further, define the stochastic process $\xi_t$ as
\begin{equation}\label{eq:beta}
   \xi_t :=\  \frac{\rho_{t+1}}{2} \|\nabla F(\bbx_{t})-\bbd_{t}\|^2.
\end{equation}
Considering the definitions of the sequences $\zeta_t$ and $\xi_t$ and expression \eqref{eq:bound_on_gradient_error} in Lemma \ref{lemma:bound_on_gradient_error} we can write 
\begin{equation}\label{eq:beta22}
  \E{\zeta_{t}\mid \bbx_t} \leq \zeta_{t-1} -\xi_{t-1}.
\end{equation}
Since the sequences $\zeta_t$ and $\xi_t$ are nonnegative it follows from \eqref{eq:beta22} that they satisfy the conditions of the supermartingale convergence theorem; see e.g. (Theorem E$7.4$ in \citep{solo1994adaptive}). Therefore, we can conclude that: (i) The sequence $\zeta_t$ converges almost surely to a limit. (ii) The sum $\sum_{t=0}^{\infty}\xi_t < \infty$ is almost surely finite. Hence, the second result implies that 
\begin{equation}\label{eq:beta33}
\sum_{t=1}^\infty \rho_t\|\nabla F(\bbx_{t-1}) - \bbd_{t-1}\|^2<\infty, \qquad \text{a.s.}
\end{equation}
 Based on expression \eqref{eq:bound_on_suboptimality} in Lemma \ref{lemma:bound_on_suboptimality}, we know that the suboptimality $F(\bbx_{t+1})-F(\bbx^*)$ is upper bounded by 
\begin{equation}\label{eq:beta44}
{F(\bbx_{t+1}) -F(\bbx^*) } 
\leq (1-\gamma_{t+1}){(F(\bbx_{t}) -F(\bbx^*))}+\gamma_{t+1} D\|\nabla F(\bbx_{t})-\bbd_t\| +\frac{LD^2\gamma_{t+1}^2}{2}.
\end{equation}
 Further use Jensen's inequality to replace $\gamma_{t+1}\|\nabla F(\bbx_{t})-\bbd_t\|$ by the sum $\beta_{t+1}\|\nabla F(\bbx_{t})-\bbd_t\|^2+\gamma_{t+1}^2/\beta_{t+1}$ where $\beta_{t+1}$ is a free positive parameter. Set $\beta_{t+1}=\rho_{t+1}$ to obtain  $\gamma_{t+1}\|\nabla F(\bbx_{t})-\bbd_t\|\leq \rho_{t+1}\|\nabla F(\bbx_{t})-\bbd_t\|^2+\gamma_{t+1}^2/\rho_{t+1}$. Applying this substitution into \eqref{eq:beta44} implies that
\begin{equation}\label{eq:beta441}
{F(\bbx_{t+1}) -F(\bbx^*) } 
\leq (1-\gamma_{t+1}){(F(\bbx_{t}) -F(\bbx^*))}+
\rho_{t+1}D\|\nabla F(\bbx_{t})-\bbd_t\|^2+\frac{D\gamma_{t+1}^2}{\rho_{t+1}}
+\frac{LD^2\gamma_{t+1}^2}{2}.
\end{equation}

To conclude the almost sure convergence of the sequence $F(\bbx_{t}) -F(\bbx^*)$ to zero from the expression in \eqref{eq:beta44} we first state the following Lemma from \citep{DBLP:books/lib/BertsekasT96}.

\begin{lemma}\label{lemma:bertsekas_lemma}
Let $\{X_t\}$, $\{Y_t\}$, and $\{Z_t\}$ be three sequences  of numbers such that $Y_t\geq 0$ for all $t\geq 0$. Suppose that 
\begin{equation}\label{claim:bertsekas_lemma}
X_{t+1}\leq X_t -Y_t+Z_t,\qquad  \for \  \ t=0,1,2,\dots
\end{equation}
and $\sum_{t=0}^\infty Z_t<\infty$. Then, either $X_t\to -\infty $ or else $\{X_t\}$ converges to a finite value and $\sum_{t=0}^\infty Y_t<\infty$.
\end{lemma}

From now on we focus on realizations that support $\sum_{t=1}^\infty \rho_t\|\nabla F(\bbx_{t-1}) - \bbd_{t-1}\|^2<\infty$ which have probability 1, according to the result in \eqref{eq:beta33}. 

Consider the outcome of Lemma \ref{lemma:bertsekas_lemma} with the identifications $X_t=F(\bbx_{t}) -F(\bbx^*)$, $Y_t=\gamma_{t+1}(F(\bbx_{t}) -F(\bbx^*))$, and $Z_t= \rho_{t+1}D\|\nabla F(\bbx_{t})-\bbd_t\|^2+\frac{D\gamma_{t+1}^2}{\rho_{t+1}} +\frac{LD^2\gamma_{t+1}^2}{2}$. Since the sequence $X_t=F(\bbx_{t}) -F(\bbx^*)$ is always non-negative, the first outcome of Lemma \ref{lemma:bertsekas_lemma} is impossible and therefore we obtain that $F(\bbx_{t}) -F(\bbx^*)$ converges to a finite limit and 
\begin{equation}\label{eq:beta500}
\sum_{t=0}^\infty \gamma_{t+1}F(\bbx_{t}) -F(\bbx^*)<\infty.
\end{equation}
Recall that both of these results hold almost surely, since they are valid for the realization that $\sum_{t=1}^\infty \rho_t\|\nabla F(\bbx_{t-1}) - \bbd_{t-1}\|^2<\infty$, which occur with probability 1 as shown in \eqref{eq:beta33}. The result in \eqref{eq:beta500} implies that $\liminf_{t\to\infty}F(\bbx_{t}) -F(\bbx^*)=0$ almost surely. Moreover, we know that the sequence $\{F(\bbx_{t}) -F(\bbx^*)\}$ almost surely converges to a finite limit. Combining these two observation we obtain that the finite limit is zero, and, therefore, $\lim_{t\to\infty}F(\bbx_{t}) -F(\bbx^*)=0$ almost surely. Hence, the claim in \eqref{eq:almost_sure_cnvrg} follows.

\section{Proof of Lemma \ref{lemma:bound_on_grad_approx_sublinear}}\label{proof:lemma:bound_on_grad_approx_sublinear}

By following the steps from \eqref{proof:bound_on_grad_100} to \eqref{proof:bound_on_grad_300} in the proof of Lemma~\ref{lemma:bound_on_gradient_error} we can show that
%
%
\begin{align}\label{proof:bound_on_grad_300000}
&\E{\|\nabla F(\bbx_t) - \bbd_t\|^2\mid\ccalF_t} =\rho_t^2\E{\|\nabla F(\bbx_t)-\nabla \tilde{F}(\bbx_t,\bbz_t)\|^2\mid\ccalF_t}+   (1-\rho_t)^2\|\nabla F(\bbx_{t-1}) - \bbd_{t-1} \|^2 
     \nonumber\\
     &  \quad
   +(1-\rho_t)^2\|\nabla F(\bbx_t)-\nabla F(\bbx_{t-1})\|^2
    + 2 (1-\rho_t)^2\langle \nabla F(\bbx_t)-\nabla F(\bbx_{t-1}) , \nabla F(\bbx_{t-1}) - \bbd_{t-1} \rangle.
 \end{align}
  The term $\E{\|\nabla F(\bbx_t)-\nabla \tilde{F}(\bbx_t,\bbz_t)\|^2\mid\ccalF_t}$ can be bounded above by $\sigma^2$ according to Assumption \ref{ass:bounded_variance}. Based on Assumptions \ref{ass:bounded_set} and \ref{ass:smoothness}, we can also show that the squared norm $\|\nabla F(\bbx_t)-\nabla F(\bbx_{t-1})\|^2$ is upper bounded by $L^2D^2/T^2$. Moreover, the inner product $2\langle \nabla F(\bbx_t)\!-\!\nabla F(\bbx_{t-1}) , \nabla F(\bbx_{t-1}) - \bbd_{t-1} \rangle$ can be upper bounded by $\beta_t \|\nabla F(\bbx_{t-1}) - \bbd_{t-1}\|^2+(1/\beta_t) L^2D^2/T^2 $ using Young's inequality (i.e., $2\langle \bba,\bbb\rangle \leq \beta\|\bba\|^2+\|\bbb\|^2/\beta$  for any $\bba,\bbb\in \reals^n$ and $\beta>0$) and the conditions in Assumptions \ref{ass:bounded_set} and \ref{ass:smoothness}, where $\beta_t>0$ is a free scalar. Applying these substitutions into \eqref{proof:bound_on_grad_300} leads to 
\begin{align}\label{proof:bound_on_grad_400000}
&\E{\|\nabla F(\bbx_t) - \bbd_t\|^2\mid\ccalF_t} \leq \rho_t^2\sigma^2
   +(1-\rho_t)^2 (1+\frac{1}{\beta_t})\frac{L^2D^2}{T^2}
   +(1-\rho_t)^2(1+\beta_t)\|\nabla F(\bbx_{t-1}) - \bbd_{t-1} \|^2. 
   \end{align}
Now by following the steps from \eqref{proof:bound_on_grad_400} to \eqref{proof:bound_on_grad_700} and computing the expected value with respect to $\ccalF_0$ 
%
we obtain
\begin{align}\label{eq:grad_error_bound}
\E{\|\nabla F(\bbx_{t}) - \bbd_{t}\|^2}\leq \left(1-\frac{\rho_t}{2}\right)\E{\|\nabla F(\bbx_{t-1}) - \bbd_{t-1}\|^2}+ \rho_t^2\sigma^2 
+\frac{2L^2D^2 }{\rho_t T^2}.
\end{align}
Define $\phi_t:=\E{ {\|\nabla F(\bbx_{t}) - \bbd_{t}\|^2} }$ and set $\rho_t=\frac{4}{(t+8)^{2/3}}$ to obtain
\begin{align}\label{proof:bound_on_grad_sublinear_100}
\phi_t
\leq \left(1-\frac{2}{(t+8)^{2/3}}\right)\phi_{t-1}+ \frac{16\sigma^2}{(t+8)^{4/3}}
   +\frac{L^2D^2(t+8)^{2/3}}{2T^2}.
\end{align}
Now use the conditions $8\leq T$ and $t\leq T$ to replace $1/T$ in \eqref{proof:bound_on_grad_sublinear_100} by its upper bound $2/(t+8)$. Applying this substitution leads to 
\begin{align}\label{proof:bound_on_grad_sublinear_300}
\phi_t
\leq \left(1-\frac{2}{(t+8)^{2/3}}\right)\phi_{t-1}+ \frac{16\sigma^2+2L^2 D^2}{(t+8)^{4/3}}.
\end{align}
Now using the result in Lemma~\ref{lemma:induction_result}, we obtain that 

\begin{equation}\label{claim:rate_convg1000}
\phi_{t} \leq \frac{Q}{(t+9)^{2/3}},
\end{equation}
where $Q:= \max\{ 4 \phi_0 , 16\sigma^2+2L^2 D^2\}$. Replacing $\phi_t$ by its definition $\E{ {\|\nabla F(\bbx_{t}) - \bbd_{t}\|^2} }$ yields \eqref{eq:grad_error_bound_2}.

\section{Proof of Theorem \ref{thm:optimal_bound_greedy}}\label{proof:thm:optimal_bound_greedy}

Let $\bbx^*$ be the global maximizer within the constraint set $\mathcal{C}$. Based on the smoothness of the function $F$ with constant $L$ we can write 
\begin{align}\label{proof:final_result_100}
F(\bbx_{t+1})
& \geq F(\bbx_{t}) + \langle \nabla F(\bbx_t), \bbx_{t+1}-\bbx_t \rangle - \frac{L}{2}  || \bbx_{t+1}-\bbx_t||^2 \nonumber\\
& = F(\bbx_{t}) + \frac{1}{T} \langle \nabla F(\bbx_t),\bbv_{t} \rangle - \frac{L}{2T^2} || \bbv_{t} ||^2,
\end{align}
where the equality follows from the update in \eqref{eq:von3}. Since  $\bbv_t$ is in the set $\ccalC$, it follows from Assumption~\ref{ass:bounded_set} that the norm $\|\bbv_t\|^2$ is bounded above by $D^2$. Apply this substitution and add and subtract the inner product $\langle \bbd_t,\bbv_t   \rangle$ to the right hand side of \eqref{proof:final_result_100} to obtain 
\begin{align}\label{proof:final_result_200}
 F(\bbx_{t+1}) 
&\geq  F(\bbx_{t}) + \frac{1}{T} \langle \bbv_{t}, \bbd_t \rangle +  \frac{1}{T} \langle \bbv_{t}, \nabla F(\bbx_t) - \bbd_t \rangle - \frac{LD^2}{2T^2} \nonumber\\
&\geq F(\bbx_{t}) +\frac{1}{T}\langle \bbx^*, \bbd_t \rangle + \frac{1}{T} \langle \bbv_{t}, \nabla F(\bbx_t) - \bbd_t \rangle - \frac{LD^2}{2T^2}.
\end{align}
Note that the second inequality in \eqref{proof:final_result_200} holds since based on \eqref{eq:von2} we can write  
\begin{equation}
\langle \bbx^*, \bbd_t \rangle \leq \max_{v\in \ccalC} \{ \langle \bbv, \bbd_t \rangle\}  = \langle \bbv_t, \bbd_t \rangle.
\end{equation}
Now add and subtract the inner product $ \langle \bbx^*, \nabla F(\bbx_t) \rangle /T$ to the RHS of \eqref{proof:final_result_200} to get
\begin{align}\label{proof:final_result_300}
 F(\bbx_{t+1}) & \geq F(\bbx_{t}) + \frac{1}{T} \langle \bbx^*, \nabla F(\bbx_t) \rangle 
 +  \frac{1}{T} \langle \bbv_{t} - \bbx^*, \nabla F(\bbx_t) - \bbd_t \rangle - \frac{LD^2}{2T^2}.
\end{align}
We further have $ \langle \bbx^*, \nabla F(\bbx_t) \rangle \geq F(\bbx^*) - F(\bbx_t)$; this follows from monotonicity of $F$ as well as concavity of $F$ along positive directions; see, e.g., \citep{calinescu2011maximizing}. Moreover, by Young's inequality we can show that the inner product $\langle \bbv_{t} - \bbx^*, \nabla F(\bbx_t) - \bbd_t \rangle$ is lower bounded by 
\begin{equation}
\langle \bbv_{t} - \bbx^*, \nabla F(\bbx_t) - \bbd_t \rangle \geq -  \frac{\beta_t}{2} ||\bbv_{t} - \bbx^*||^2 - \frac{1}{2\beta_t}{|| \nabla F(\bbx_t) - \bbd_t ||^2},
\end{equation}
for any $\beta_t>0$. By applying these substitutions into \eqref{proof:final_result_300} we obtain 
\begin{align}\label{proof:final_result_400}
 &F(\bbx_{t+1}) 
\geq F(\bbx_{t}) + \frac{1}{T} (F(\bbx^*) - F(\bbx_{t})) - \frac{LD^2}{2T^2}
-  \frac{1}{2T}\left(\beta_t||\bbv_{t} - \bbx^*||^2 + \frac{|| \nabla F(\bbx_t) - \bbd_t ||^2}{\beta_t}\right).
\end{align}
Replace $||\bbv_{t} - \bbx^*||^2$ by its upper bound $4D^2$ and compute the expected value of \eqref{proof:final_result_400} to write 
\begin{align}\label{proof:final_result_500}
& \E{F(\bbx_{t+1}) } \geq  \E{F(\bbx_{t})} + \frac{1}{T}\E{F(\bbx^*) - F(\bbx_{t}))} 
-  \frac{1}{2T} \left[ 4\beta_t D^2 + \frac{\E{|| \nabla F(\bbx_t) - \bbd_t ||^2}}{\beta_t}\right] - \frac{LD^2}{2T^2}.
\end{align}
Substitute $\E{|| \nabla F(\bbx_t) - \bbd_t ||^2}$ by its upper bound ${Q}/({(t+9)^{2/3}})$ according to the result in \eqref{eq:grad_error_bound_2}. Further, set $\beta_t= (Q^{1/2})/(2D(t+9)^{1/3})$ and regroup the resulted expression to obtain 
\begin{align}\label{proof:final_result_600}
 \E{F(\bbx^*) - F(\bbx_{t+1}) }& \leq \left(1-\frac{1}{T}\right) \E{F(\bbx^*) -F(\bbx_{t})} 
+  \frac{2DQ^{1/2}}{(t+9)^{1/3}T}+\frac{LD^2}{2T^2}.
\end{align}
By applying the inequality in \eqref{proof:final_result_600} recursively for $t=0,\dots,T-1$ we obtain 
\begin{align}\label{proof:final_result_700}
& \E{F(\bbx^*) - F(\bbx_{T}) } \leq \left(1-\frac{1}{T}\right)^T ({F(\bbx^*) -F(\bbx_{0})}  )
+ \sum_{t=0}^{T-1} \frac{2DQ^{1/2}}{(t+9)^{1/3}T}  + \sum_{t=0}^{T-1} \frac{LD^2}{2T^2}.
\end{align}
Note that we can write
\begin{align}\label{jsjsjsjsjha}
\sum_{t=0}^{T-1} \frac{1}{(t+9)^{1/3}} 
&\leq \int_{t=0}^{T-1}\frac{1}{(t+9)^{1/3}}\ dt  \nonumber\\
&=\frac{3}{2} (t+9)^{2/3}\mid_{t=T-1}\ - \ \frac{3}{2} (t+9)^{2/3}\mid_{t=0}\nonumber\\
&\leq \frac{3}{2} (T+8)^{2/3}\nonumber\\
&\leq  \frac{15}{2} T^{2/3}
\end{align}
where the last inequality holds since $(T+8)^{2/3}\leq 5 T^{2/3}$ for any $T\geq 1$.
By simplifying the terms on the right hand side \eqref{proof:final_result_700} and using the inequality in \eqref{jsjsjsjsjha} we can write 
\begin{align}\label{proof:final_result_800}
& \E{F(\bbx^*) - F(\bbx_{T}) }
  \leq \frac{1}{e} ({F(\bbx^*) -F(\bbx_{0})}  )
+ \ \frac{15DQ^{1/2}}{T^{1/3}}  + \frac{LD^2}{2T}.
\end{align}
Here, we use the fact that $F(\bbx_{0}) \geq0$, and hence the expression in \eqref{proof:final_result_800} can be simplified to  
\begin{equation}\label{proof:final_result_900}
 \E{F(\bbx_{T}) }\geq (1- 1/e) F(\bbx^*)  - \frac{15DQ^{1/2}}{T^{1/3}}-  \frac{LD^2}{2T},
\end{equation}
and the claim in \eqref{eq:claim_for_sto_greedy} follows. 

\section{Proof of Theorem \ref{thm:optimal_bound_greedy_weak}}\label{proof:thm:optimal_bound_greedy_weak}

Following the steps of the proof of Theorem \ref{thm:optimal_bound_greedy} we can derive the inequality
\begin{align}\label{proof:final_result_weak_300}
 F(\bbx_{t+1}) & \geq F(\bbx_{t}) + \frac{1}{T} \langle \bbx^*, \nabla F(\bbx_t) \rangle 
 +  \frac{1}{T} \langle \bbv_{t} - \bbx^*, \nabla F(\bbx_t) - \bbd_t \rangle - \frac{LD^2}{2T^2}.
\end{align}
Using the definition of weak DR-submodularity and monotonicity of $F$ we can wrtie $  \langle \bbx^*, \nabla F(\bbx_t) \rangle \geq \gamma(F(\bbx^*) - F(\bbx_t))$. Further, based on Young's inequality the inner product $\langle \bbv_{t} - \bbx^*, \nabla F(\bbx_t) - \bbd_t \rangle$ is lower bounded by $-  (\beta_t/2)||\bbv_{t} - \bbx^*||^2 - (1/2\beta_t){|| \nabla F(\bbx_t) - \bbd_t ||^2} $ for any $\beta_t>0$. Applying these substitutions into \eqref{proof:final_result_weak_300} leads to 
\begin{align}\label{proof:final_result_weak_400}
 &F(\bbx_{t+1}) 
\geq F(\bbx_{t}) + \frac{\gamma}{ T} (F(\bbx^*) - F(\bbx_{t})) - \frac{LD^2}{2T^2}
-  \frac{1}{2T}\left(\beta_t||\bbv_{t} - \bbx^*||^2 + \frac{|| \nabla F(\bbx_t) - \bbd_t ||^2}{\beta_t}\right).
\end{align}
Substitute $||\bbv_{t} - \bbx^*||^2$ by its upper bound $4D^2$ and compute the expected value of \eqref{proof:final_result_weak_400} to write 
\begin{align}\label{proof:final_result_weak_500}
& \E{F(\bbx_{t+1}) } \geq  \E{F(\bbx_{t})} + \frac{\gamma}{ T}\E{F(\bbx^*) - F(\bbx_{t}))} 
-  \frac{1}{2T} \left[ 4\beta_t D^2 + \frac{\E{|| \nabla F(\bbx_t) - \bbd_t ||^2}}{\beta_t}\right] - \frac{LD^2}{2T^2}.
\end{align}
Substitute $\E{|| \nabla F(\bbx_t) - \bbd_t ||^2}$ by its upper bound ${Q}/({(t+9)^{2/3}})$ according to the result in \eqref{eq:grad_error_bound_2}. Further, set $\beta_t= (Q^{1/2})/(2D(t+9)^{1/3})$ and regroup the resulted expression to obtain 
\begin{align}\label{proof:final_result_weak_600}
 \E{F(\bbx^*) - F(\bbx_{t+1}) }& \leq \left(1-\frac{\gamma}{ T}\right) \E{F(\bbx^*) -F(\bbx_{t})} 
+  \frac{2DQ^{1/2}}{(t+9)^{1/3}T}+\frac{LD^2}{2T^2}.
\end{align}
By applying the inequality in \eqref{proof:final_result_weak_600} recursively for $t=0,\dots,T-1$ we obtain 
\begin{align}\label{proof:final_result_weak_700}
& \E{F(\bbx^*) - F(\bbx_{T}) } \leq \left(1-\frac{\gamma}{T}\right)^T ({F(\bbx^*) -F(\bbx_{0})}  )
+ \sum_{t=0}^{T-1} \frac{2DQ^{1/2}}{(t+9)^{1/3}T}+ \sum_{t=0}^{T-1} \frac{LD^2}{2T^2}.
\end{align}
Simplify the terms on the right hand side \eqref{proof:final_result_weak_700} and use \eqref{jsjsjsjsjha} to obtain 
\begin{align}\label{proof:final_result_weak_800}
& \E{F(\bbx^*) - F(\bbx_{T}) }
  \leq e^{-\gamma} ({F(\bbx^*) -F(\bbx_{0})}  )
+  \frac{15DQ^{1/2}}{T^{1/3}}+ \frac{LD^2}{2T}.
\end{align}
Here, we use the fact that $F(\bbx_{0}) \geq0$, and hence the expression in \eqref{proof:final_result_weak_800} can be simplified to  
\begin{equation}\label{proof:final_result_weak-900}
 \E{F(\bbx_{T}) }\geq (1- e^{-\gamma})F(\bbx^*)  - \frac{15DQ^{1/2}}{T^{1/3}}-  \frac{LD^2}{2T},
\end{equation}
and the claim in \eqref{eq:claim_for_sto_greedy_weak} follows.

\section{Proof of Theorem \ref{thm:non_monotone_bound}}\label{proof:thm:non_monotone_bound}
Using the update of the NMSCG method we can write 
\begin{align}\label{proof_firs_eq}
\bbx_{i,t+1}
&= \bbx_{i,t} +\frac{1}{T} \bbv_{i,t} \nonumber\\
&\leq \bbx_{i,t} +\frac{1}{T} (\bar{\bbu}_i - \bbx_{i,t}) \nonumber\\
&\leq \left(1-\frac{1}{T}\right)\bbx_{i,t} +\frac{1}{T} \bar{\bbu}_i \nonumber\\
&= \left(1-\frac{1}{T}\right)^t\bbx_{i,0} +\frac{1}{T} \bar{\bbu}_i \sum_{j=0}^t \left(1-\frac{1}{T}\right)^j\nonumber\\
&= \bar{\bbu}_i\left(1-\left(1-\frac{1}{T}\right)^{t+1}\right)
\end{align}
where the first inequality follows by the condition $\bbv_{i,t} \leq \bar{\bbu_i}-\bbx_{i,t} $. The result in \eqref{proof_firs_eq} implies that $\bbx_{t} \leq \bar{\bbu} (1-(1-1/T)^{t})$. Therefore, according to Lemma 3 in \citep{DBLP:conf/nips/BianL0B17} which is a generalized version of Lemma 7 in \citep{chekuri2015multiplicative} it follows that
\begin{equation}\label{most_imp}
F(\bbx_t\vee\bbx^*)\geq (1-(1/T))^{t}F(\bbx^*).
\end{equation}
Using this result and the Taylor's expansion of the objective function $F$ we can write
\begin{align}\label{yechi}
F(\bbx_{t+1})-F(\bbx_t)
& \geq \langle \nabla F(\bbx_t), \bbx_{t+1}-\bbx_t\rangle -\frac{L}{2}\|\bbx_{t+1}-\bbx_t\|^2\nonumber\\
& = \frac{1}{T} \langle \nabla F(\bbx_t), \bbv_t\rangle -\frac{L}{2T^2}\|\bbv_t\|^2,
\end{align}
where the equality follows by the update of NMSCG. Add and subtract $\langle \bbd_t, \bbv_t\rangle$ to the right hand side of \eqref{yechi} to obtain
\begin{align}\label{kashkak}
F(\bbx_{t+1})-F(\bbx_t) &\geq \frac{1}{T} \langle \bbd_t, \bbv_t\rangle+\frac{1}{T} \langle \nabla F(\bbx_t)-\bbd_t, \bbv_t\rangle -\frac{L}{2T^2}\|\bbv_t\|^2 \nonumber\\
& \geq \frac{1}{T} \langle \bbd_t, \bbx_t\vee\bbx^*-\bbx_t\rangle+\frac{1}{T} \langle \nabla F(\bbx_t)-\bbd_t, \bbv_t\rangle -\frac{L}{2T^2}\|\bbv_t\|^2,
\end{align}
where the second inequality is valid since $  \{ \bbd_t^T\bbv_t \} \geq \{ \bbd_t^T\bby \}$ for all $\bby\leq \bar{\bbu}-\bbx_t$ and we know that $\bbx_t\vee\bbx^*-\bbx_t\leq \bar{\bbu}-\bbx_t$. Now add and subtract $(1/T) \langle\nabla F(\bbx_t), \bbx_t\vee\bbx^*-\bbx_t\rangle$ to the right hand side of \eqref{kashkak} to obtain
\begin{align}
&F(\bbx_{t+1})-F(\bbx_t) \nonumber\\
& \geq \frac{1}{T} \langle\nabla F(\bbx_t), \bbx_t\vee\bbx^*-\bbx_t\rangle+\frac{1}{T} \langle \bbd_t-\nabla F(\bbx_t), \bbx_t\vee\bbx^*-\bbx_t\rangle+\frac{1}{T} \langle \nabla F(\bbx_t)-\bbd_t, \bbv_t\rangle -\frac{L}{2T^2}\|\bbv_t\|^2 \nonumber\\
& = \frac{1}{T} \langle\nabla F(\bbx_t), \bbx_t\vee\bbx^*-\bbx_t\rangle+\frac{1}{T} \langle \nabla F(\bbx_t)-\bbd_t, \bbv_t+\bbx_t-\bbx_t\vee\bbx^*\rangle -\frac{L}{2T^2}\|\bbv_t\|^2 \nonumber\\
& \geq \frac{1}{T} \langle\nabla F(\bbx_t), \bbx_t\vee\bbx^*-\bbx_t\rangle-\frac{2D}{T} \| \nabla F(\bbx_t)-\bbd_t\|  -\frac{L}{2T^2}\|\bbv_t\|^2 .
\end{align}
The last inequality holds since the inner product $\langle \nabla F(\bbx_t)-\bbd_t, \bbv_t+\bbx_t-\bbx_t\vee\bbx^*\rangle$ can be upper bounded by $\|\nabla F(\bbx_t)-\bbd_t\| \|\bbv_t+\bbx_t-\bbx_t\vee\bbx^*\|$ using the Cauchy-Schwartz inequality and further we can upper bound the norm $\|\bbv_t+\bbx_t-\bbx_t\vee\bbx^*\|$ by $2D$ since $\|\bbv_t+\bbx_t-\bbx_t\vee\bbx^*\|\leq \|\bbv_t\| +\|\bbx_t-\bbx_t\vee\bbx^*\|$ and both $\bbx_t-\bbx_t\vee\bbx^*$ and $\bbv_t$ belong to the set $\ccalC$. This holds due to the assumption that $\ccalC$ is down-closed.

Now replace $\|\bbv_t\|^2$ by its upper bound $D^2$ and use the concavity of the function $F$ in positive directions to write
\begin{align}
F(\bbx_{t+1})-F(\bbx_t)
& \geq \frac{1}{T} ( F(\bbx_t\vee\bbx^*)-F(\bbx_t))
-\frac{2D}{T} \| \nabla F(\bbx_t)-\bbd_t\|  -\frac{LD^2}{2T^2} .
\end{align}
Now use the expression in \eqref{most_imp} to obtain
\begin{align}\label{yechi222}
F(\bbx_{t+1})-F(\bbx_t)
& \geq \frac{1}{T} \left[ \left(1-\frac{1}{T}\right)^{t}F(\bbx^*)-F(\bbx_t)\right]
-\frac{2D}{T} \| \nabla F(\bbx_t)-\bbd_t\|  -\frac{LD^2}{2T^2}.
\end{align}
Computing the expectation of both sides and replacing $\E{\| \nabla F(\bbx_t)-\bbd_t\|}$ by its upper bound ${\sqrt{Q}}/({(t+9)^{1/3}})$ according to the result in \eqref{eq:grad_error_bound_2} lead to
\begin{align}\label{opokopkm}
\E{F(\bbx_{t+1})}
& \geq \left(1-\frac{1}{T}\right)\E{F(\bbx_t)}+\frac{1}{T} \left(1-\frac{1}{T}\right)^{t}F(\bbx^*)
-\frac{2D\sqrt{Q}}{T(t+9)^{1/3}}   -\frac{LD^2}{2T^2}.
\end{align}
By applying the inequality in \eqref{opokopkm} recursively for $t=0,\dots,T-1$ we obtain 
\begin{align}\label{ksksksksk}
\E{F(\bbx_{T})} 
&\geq \left(1-\frac{1}{T}\right)^TF(\bbx_0)
+\sum_{i=0}^{T-1} \left(\frac{1}{T} \left(1-\frac{1}{T}\right)^{i}F(\bbx^*)
-\frac{2D\sqrt{Q}}{T(i+9)^{1/3}}   -\frac{LD^2}{2T^2}\right) \left(1-\frac{1}{T}\right)^{T-1-i}\nonumber\\
&= \left(1-\frac{1}{T}\right)^TF(\bbx_0)
+\left(1-\frac{1}{T}\right)^{T-1}F(\bbx^*)
-\sum_{i=0}^{T-1} \left(\frac{2D\sqrt{Q}}{T(i+9)^{1/3}}   +\frac{LD^2}{2T^2}\right) \left(1-\frac{1}{T}\right)^{T-1-i}\nonumber\\
&\geq \left(1-\frac{1}{T}\right)^TF(\bbx_0)
+\left(1-\frac{1}{T}\right)^{T-1}F(\bbx^*)
-\sum_{i=0}^{T-1} \left(\frac{2D\sqrt{Q}}{T(i+9)^{1/3}}   +\frac{LD^2}{2T^2}\right),
\end{align}
where the second inequality holds since $(1-1/T)$ is strictly smaller than $1$. Replacing the sum $\sum_{i=0}^{T-1} \frac{1}{(i+9)^{1/3}} $ in \eqref{ksksksksk} by its upper bound in \eqref{jsjsjsjsjha} leads to
\begin{align}
\E{F(\bbx_{T}) }\geq \left(1-\frac{1}{T}\right)^TF(\bbx_0)+\left(1-\frac{1}{T}\right)^{T-1}F(\bbx^*)
-\frac{15D\sqrt{Q}}{T^{1/3}}   -\frac{LD^2}{2T}
\end{align}
Use the fact that $F(\bbx_{0}) \geq0$ and the inequality $\left(1-\frac{1}{T}\right)^{T-1}\geq e^{-1}$ to obtain
\begin{align}\label{risk_factor}
\E{F(\bbx_{T})} \geq e^{-1}F(\bbx^*)
-\frac{15D\sqrt{Q}}{T^{1/3}}   -\frac{LD^2}{2T},
\end{align}
and the claim in \eqref{claim:non_monotone_bound} follows.

\section{Proof of Lemma \ref{lemma:lip_constant}}\label{proof:lemma:lip_constant}
Based on the mean value theorem, we can write
\begin{align}
\nabla F(\bbx_t+\frac{1}{T}\bbv_t) -\nabla F(\bbx_T)  = \frac{1}{T} \bbH(\tbx_t)\bbv_t,
\end{align}
where $\tbx_t$ is a convex combination of $\bbx_t$ and $\bbx_t+\frac{1}{T}\bbv_t$ and $\bbH(\tbx_t):=\nabla^2 F(\tbx_t)$. This expression shows that the difference between the coordinates of the vectors $\nabla F(\bbx_t+\frac{1}{T}\bbv_t) $ and $\nabla F(\bbx_t)$ can be written as
\begin{align}
\nabla_{j} F(\bbx_t+\frac{1}{T}\bbv_t) -\nabla_{j} F(\bbx_t)  = \frac{1}{T} \sum_{i=1}^n H_{j,i}(\tbx_t)v_{i,t},
\end{align}
where $v_{i,t}$ is the $i$-th element of the vector $\bbv_t$ and $H_{j,i}$ denotes the component in the $j$-th row and $i$-th column of the matrix $\bbH$. Hence, the norm of the difference $|\nabla_{j} F(\bbx_t+\frac{1}{T}\bbv_t) -\nabla_{j} F(\bbx_t) |$ is bounded above by
\begin{align}
|\nabla_{j} F(\bbx_t+\frac{1}{T}\bbv_t) -\nabla_{j} F(\bbx_t) | \leq \frac{1}{T} \left|\sum_{i=1}^n H_{j,i}(\tbx_t)v_{i,t}\right|.
\end{align}
Note here that the elements of the matrix $\bbH(\tbx_t)$ are less than the maximum marginal value (i.e. $\max_{i,j} |H_{i,j}(\tbx_t)| \leq \max_{i \in \{1, \cdots, n\}} f(i) \triangleq m_f$). We thus get
\begin{align}
|\nabla_{j} F(\bbx_t+\frac{1}{T}\bbv_t) -\nabla_{j} F(\bbx_t) | \leq \frac{m_f}{T} \sum_{i=1}^n | v_{i,t}|.
\end{align}
Note that at each round $t$ of the algorithm, we have to pick a vector $\bbv_t \in \mathcal{C}$ s.t. the inner product $\langle \bbv_t, \bbd_t\rangle$ is maximized.  Hence, without loss of generality we can assume that the vector $\bbv_t$ is one of the extreme points of $\mathcal{C}$, i.e. it is of the form $1_{I}$ for some $I \in \mathcal{I}$ (note that we can easily force integer vectors).  Therefore by noticing that $\bbv_t$ is an integer vector with at most $r$ ones, we have
\begin{align}
|\nabla_{j} F(\bbx_t+\frac{1}{T}\bbv_t) -\nabla_{j} F(\bbx_t) | \leq \frac{m_f\sqrt{r}}{T} \sqrt{\sum_{i=1}^n | v_{i,t}|^2},
\end{align}
which yields the claim in \eqref{claim:lip_constant}.

\section{Proof of Theorem \ref{thm:multi_linear_extenstion_thm}}\label{proof:thm:multi_linear_extenstion_thm}

According to the Taylor's expansion of the function $F$ near the point $\bbx_t$ we can write 
\begin{align}\label{proof:final_result__multi_lin_100}
F(\bbx_{t+1}) &= F(\bbx_{t}) + \langle \nabla F(\bbx_t), \bbx_{t+1}-\bbx_t \rangle 
 + \frac{1}{2}\langle \bbx_{t+1}-\bbx_t, \bbH(\tbx_t) (\bbx_{t+1}-\bbx_t)\rangle \nonumber\\
&= F(\bbx_{t}) + \frac{1}{T}\langle \nabla F(\bbx_t), \bbv_t \rangle 
+ \frac{1}{2T^2}\langle \bbv_t, \bbH(\tbx_t) \bbv_t\rangle,
 \end{align}
where $\tbx_t$ is a convex combination of $\bbx_t$ and $\bbx_t+\frac{1}{T}\bbv_t$ and $\bbH(\tbx_t):=\nabla^2 F(\tbx_t)$. Note that based on the inequality $\max_{i,j} |H_{i,j}(\tbx_t)| \leq \max_{i \in \{1, \cdots, n\}} f(i) \triangleq m_f$, we can lower bound $H_{ij}$ by $-m_f$. Therefore, 
\begin{align}\label{proof:final_result__multi_lin_200}
\langle \bbv_t, \bbH(\tbx_t) \bbv_t\rangle 
= \sum_{j=1}^n\sum_{i=1}^n v_{i,t} v_{j,t} H_{ij}(\tbx_t)\geq -m_f\sum_{j=1}^n\sum_{i=1}^n v_{i,t} v_{j,t} =-m_f\left(\sum_{i=1}^n v_{i,t}\right)^2=-m_f r\|\bbv_t\|^2,
\end{align}
where the last inequality is because $\bbv_t$ is a vector with $r$ ones and $n-r$ zeros (see the explanation in the proof of Lemma~\ref{lemma:lip_constant}).
Replace the expression $\langle \bbv_t, \bbH(\tbx_t) \bbv_t\rangle $ in \eqref{proof:final_result__multi_lin_100} by its lower bound in \eqref{proof:final_result__multi_lin_200} to obtain
\begin{align}\label{proof:final_result__multi_lin_300}
F(\bbx_{t+1}) \geq F(\bbx_{t}) + \frac{1}{T}\langle \nabla F(\bbx_t), \bbv_t \rangle 
- \frac{m_f r}{2T^2}\|\bbv_t\|^2.
\end{align}
In the following lemma we derive a variant of the result in Lemma \ref{lemma:bound_on_grad_approx_sublinear} for the multilinear extension setting.

\begin{lemma}\label{lemma:sub_grad_error}
Consider \alg (SCG)  outlined in Algorithm~\ref{algo_SCGGA}, and recall the definitions of the function $F$ in  \eqref{eq:multilinear_program}, the rank $r$, and $m_f \triangleq \max_{i \in \{1, \cdots, n\}} f(i)$.  If we set $\rho_t=\frac{4}{(t+8)^{2/3}}$, then for $t=0,\dots,T$ and for $j=1,\dots,n$ it holds
\begin{align}\label{ali_mansoor}
\E{ \|\nabla F(\bbx_{t}) - \bbd_{t}\|^2}&\leq \frac{Q}{(t+9)^{2/3}},
\end{align}
 where  $Q:=\max \{ 4\|\nabla F(\bbx_{0}) - \bbd_{0}\|^2  , 16\sigma^2+3m_f^2 r D^2 \}$. 
\end{lemma}

\begin{proof}
The proof is similar to the proof of Lemma~\ref{lemma:bound_on_grad_approx_sublinear}. The main difference is to write the analysis for the $j$-th coordinate and replace and $L$ by  ${m_f\sqrt{r}}$ as shown in Lemma \ref{lemma:lip_constant}. Then using the proof techniques in Lemma \ref{lemma:bound_on_grad_approx_sublinear} the claim in Lemma \ref{lemma:sub_grad_error} follows.
\end{proof}

The rest of the proof is identical to the proof of Theorem~\ref{thm:optimal_bound_greedy}, by following the steps from \eqref{proof:final_result_100} to \eqref{proof:final_result_900} and considering the bound in \eqref{ali_mansoor} we obtain
\begin{equation}\label{proof:final_result_900_v2}
 \E{F(\bbx_{T}) }\geq (1- 1/e) F(\bbx^*)  - \frac{2DQ^{1/2}}{T^{1/3}}-  \frac{m_f rD^2}{2T},
\end{equation}
where $Q:=\max \{ 4\|\nabla F(\bbx_{0}) - \bbd_{0}\|^2  , 16\sigma^2+3rm_f^2 D^2 \}$. Therefore, the claim in Theorem \ref{thm:multi_linear_extenstion_thm} follows by replacing $\sigma^2$ by $n\max_{j  \in [n]} \mathbb{E}[ \tilde{f}(\{j\}, \mathbf{z})^2 ]$ as shown in \eqref{var_bound_multi}.

\section{Proof of Theorem \ref{thm:multi_linear_extenstion_thm_non_monotone_case}}\label{proof:thm:multi_linear_extenstion_thm_non_monotone_case}

Following the steps of the proof of Theorem \ref{thm:multi_linear_extenstion_thm} from \eqref{proof:final_result__multi_lin_100} to \eqref{proof:final_result__multi_lin_300} we obtain
\begin{align}\label{proof:final_result__multi_lin_nm_100}
F(\bbx_{t+1}) \geq F(\bbx_{t}) + \frac{1}{T}\langle \nabla F(\bbx_t), \bbv_t \rangle 
- \frac{m_f r}{2T^2}\|\bbv_t\|^2.
\end{align}
Then, by following the steps from \eqref{most_imp} to \eqref{yechi222} we obtain
\begin{align}\label{yechi22000}
F(\bbx_{t+1})-F(\bbx_t)
& \geq \frac{1}{T} \left[ \left(1-\frac{1}{T}\right)^{t}F(\bbx^*)-F(\bbx_t)\right]
-\frac{2D}{T} \| \nabla F(\bbx_t)-\bbd_t\|  -\frac{m_f rD^2}{2T^2}.
\end{align}
Note that the result in Lemma \ref{lemma:sub_grad_error} holds for both monotone and non-monotone functions $F$. Therefore, by computing the expected value of both sides of \eqref{yechi222}  we obtain
Then, by following the steps from \eqref{most_imp} to \eqref{yechi222} we obtain
\begin{align}\label{yechi2200iei}
\E{F(\bbx_{t+1})}-\E{F(\bbx_t)}
& \geq \frac{1}{T} \left[ \left(1-\frac{1}{T}\right)^{t}F(\bbx^*)-\E{F(\bbx_t)}\right]
-\frac{2D}{T} \frac{\sqrt{Q}}{(t+9)^{1/3}}  -\frac{m_f rD^2}{2T^2}.
\end{align}
Now by following the steps from \eqref{opokopkm} to \eqref{risk_factor} the claim in \eqref{eq:claim_for_sto_greedy_multi_linear_non_monotone_case} follows. 
\section{Proof of Theorem \ref{thm:optimal_bound_curvature_greedy}}\label{proof:thm:optimal_bound_curvature_greedy}

According to the result in Lemma 3 of \citep{hassani2017gradient}, it can be shown that  when the function $F$ has a curvature $c$ then
\begin{equation}
\max_{\bbv}\bbv^T \nabla F(\bbx_t) \geq  F(\bbx^*) - c F(\bbx_t),
\end{equation}
Using this inequality we can write 
\begin{align}
\bbv_t^T\bbd_t
&= \max_{\bbv}\{\bbv^T \bbd_t\}\nonumber\\
& = \max_{\bbv} \{\bbv^T \nabla F(\bbx_t) +\bbv^T(\nabla F(\bbx_t)-\bbd_t)\}\nonumber\\
&\geq \max_{\bbv} \{\bbv^T \nabla F(\bbx_t) \}-\max_{\bbv} \{\bbv^T(\nabla F(\bbx_t)-\bbd_t)\}\nonumber\\
&\geq F(\bbx^*) - c F(\bbx_t)  - D\|\nabla F(\bbx_t)-\bbd_t\|.
\end{align}
Therefore, using the result in \eqref{proof:final_result__multi_lin_300} we can write
\begin{align}
F(\bbx_{t+1}) & \geq F(\bbx_{t}) + \frac{1}{T} \langle \bbv_{t}, \nabla F(\bbx_t) \rangle -  \frac{m_fr}{2T^2} || \bbv_{t} ||^2\nonumber \\
&=  F(\bbx_{t}) + \frac{1}{T} \langle \bbv_{t}, \bbd_t \rangle +  \frac{1}{T} \langle \bbv_{t}, \nabla F(\bbx_t) - \bbd_t \rangle -  \frac{L}{2T^2} || \bbv_{t} ||^2  \nonumber\\
&\geq F(\bbx_{t}) + \frac{1}{T}(F(\bbx^*) - c F(\bbx_t)  ) -  \frac{2D}{T} \|\nabla F(\bbx_t) - \bbd_t \| - \frac{m_frD^2}{2T^2}  .
\end{align}
Compute the expected value of both sides and use the result of Lemma \ref{lemma:sub_grad_error} to obtain 
\begin{align}
\E{F(\bbx_{t+1})}  \geq \E{F(\bbx_{t})} + \frac{1}{T}(F(\bbx^*) - c \E{F(\bbx_t)}  ) -  \frac{2D\sqrt{Q}}{T(t+9)^{1/3}}  - \frac{m_frD^2}{2T^2} .
\end{align}
Subtract $F^*:=F(\bbx^*)$ from both sides and regroup the terms to obtain
\begin{align}\label{pppp}
\E{F(\bbx_{t+1})} -F^* \geq (1-\frac{c}{T})(\E{F(\bbx_{t})} -F^*)+ \frac{1-c}{T}F^*   -  \frac{2D\sqrt{Q}}{T(t+9)^{1/3}}  - \frac{m_frD^2}{2T^2}  .
\end{align}
Now applying the expression in \eqref{pppp} for $t=0,\dots,T-1$ recursively yields
\begin{align}
\E{F(\bbx_{T})} -F^* 
&\geq (1-\frac{c}{T})^T(F(\bbx_{0}) -F^*)+  \sum_{i=0}^{T-1} \left[\frac{1-c}{T}F^*   -  \frac{2D\sqrt{Q}}{T(i+9)^{1/3}}  - \frac{m_frD^2}{2T^2}  \right]  (1-\frac{c}{T})^{T-1-i}\nonumber\\
&\geq (1-\frac{c}{T})^T(F(\bbx_{0}) -F^*)+ \frac{1-c}{T}F^* \sum_{i=0}^{T-1}(1-\frac{c}{T})^{T-1-i}      - \sum_{i=0}^{T-1} \left[     \frac{2D\sqrt{Q}}{T(i+9)^{1/3}}  +\frac{m_frD^2}{2T^2}  \right] \nonumber\\
&\geq (1-\frac{c}{T})^T(F(\bbx_{0}) -F^*)+ \frac{1-c}{c} \left({1-(1-\frac{c}{T})^T}\right) F^*      -     \frac{15D\sqrt{Q}}{T^{1/3}}  - \frac{m_frD^2 }{2T} \nonumber\\
&\geq -(1-\frac{c}{T})^TF^*+ \frac{1-c}{c} \left({1-(1-\frac{c}{T})^T}\right) F^*      -     \frac{15D\sqrt{Q}}{T^{1/3}}  - \frac{m_frD^2 }{2T} \nonumber\\
&= \left(\frac{1-c}{c} -\frac{1}{c} (1-\frac{c}{T})^T\right) F^*      -     \frac{15D\sqrt{Q}}{T^{1/3}}  - \frac{m_frD^2 }{2T},
\end{align}
where in the second inequality we use the fact that $(1-c/T)\leq 1$, in the third inequality we use the result in \eqref{jsjsjsjsjha}, and in the fourth inequality we use the assumption that $F(\bbx_0)\geq 0$.
Therefore, by regrouping the terms we obtain that
\begin{align}
\E{F(\bbx_{T})}  
&\geq \frac{1}{c} \left( 1- (1-\frac{c}{T})^T\right) F^*-  \frac{15D\sqrt{Q}}{T^{1/3}}  - \frac{m_frD^2 }{2T} \nonumber\\
&\geq \frac{1}{c} \left( 1- e^{-c}\right) F^*-  \frac{15D\sqrt{Q}}{T^{1/3}}  - \frac{m_frD^2 }{2T} ,
\end{align}
and the claim in \eqref{eq:claim_for_sto_greedy_curvature} follows.

%
%

\bibliography{bibliography,references}

\begin{thebibliography}{57}
\providecommand{\natexlab}[1]{#1}
\providecommand{\url}[1]{\texttt{#1}}
\expandafter\ifx\csname urlstyle\endcsname\relax
  \providecommand{\doi}[1]{doi: #1}\else
  \providecommand{\doi}{doi: \begingroup \urlstyle{rm}\Url}\fi

\bibitem[Bach(2015)]{bach2015submodular}
Francis Bach.
\newblock Submodular functions: from discrete to continous domains.
\newblock \emph{arXiv preprint arXiv:1511.00394}, 2015.

\bibitem[Bertsekas and Tsitsiklis(1996)]{DBLP:books/lib/BertsekasT96}
Dimitri~P. Bertsekas and John~N. Tsitsiklis.
\newblock \emph{Neuro-dynamic programming}, volume~3 of \emph{Optimization and
  neural computation series}.
\newblock Athena Scientific, 1996.

\bibitem[Bian et~al.(2017{\natexlab{a}})Bian, Levy, Krause, and
  Buhmann]{DBLP:conf/nips/BianL0B17}
Andrew~An Bian, Kfir~Yehuda Levy, Andreas Krause, and Joachim~M. Buhmann.
\newblock Non-monotone continuous dr-submodular maximization: Structure and
  algorithms.
\newblock In \emph{Advances in Neural Information Processing Systems 30: Annual
  Conference on Neural Information Processing Systems 2017, 4-9 December 2017,
  Long Beach, CA, {USA}}, pages 486--496, 2017{\natexlab{a}}.

\bibitem[Bian et~al.(2017{\natexlab{b}})Bian, Mirzasoleiman, Buhmann, and
  Krause]{bian16guaranteed}
Andrew~An Bian, Baharan Mirzasoleiman, Joachim~M. Buhmann, and Andreas Krause.
\newblock Guaranteed non-convex optimization: Submodular maximization over
  continuous domains.
\newblock In \emph{Proceedings of the 20th International Conference on
  Artificial Intelligence and Statistics, {AISTATS} 2017, 20-22 April 2017,
  Fort Lauderdale, FL, {USA}}, pages 111--120, 2017{\natexlab{b}}.

\bibitem[Bottou(2010)]{bottou2010large}
L{\'e}on Bottou.
\newblock Large-scale machine learning with stochastic gradient descent.
\newblock In \emph{Proceedings of COMPSTAT'2010}, pages 177--186. Springer,
  2010.

\bibitem[Buchbinder and Feldman(2016)]{buchbinder2016constrained}
Niv Buchbinder and Moran Feldman.
\newblock Constrained submodular maximization via a non-symmetric technique.
\newblock \emph{arXiv preprint arXiv:1611.03253}, 2016.

\bibitem[Buchbinder et~al.(2014)Buchbinder, Feldman, Naor, and
  Schwartz]{buchbinder2014submodular}
Niv Buchbinder, Moran Feldman, Joseph Naor, and Roy Schwartz.
\newblock Submodular maximization with cardinality constraints.
\newblock In \emph{Proceedings of the Twenty-Fifth Annual {ACM-SIAM} Symposium
  on Discrete Algorithms, {SODA} 2014, Portland, Oregon, USA, January 5-7,
  2014}, pages 1433--1452, 2014.

\bibitem[Buchbinder et~al.(2015)Buchbinder, Feldman, Naor, and
  Schwartz]{buchbinder2015tight}
Niv Buchbinder, Moran Feldman, Joseph Naor, and Roy Schwartz.
\newblock A tight linear time (1/2)-approximation for unconstrained submodular
  maximization.
\newblock \emph{{SIAM} Journal on Computing}, 44\penalty0 (5):\penalty0
  1384--1402, 2015.

\bibitem[Calinescu et~al.(2011)Calinescu, Chekuri, P{\'a}l, and
  Vondr{\'a}k]{calinescu2011maximizing}
Gruia Calinescu, Chandra Chekuri, Martin P{\'a}l, and Jan Vondr{\'a}k.
\newblock Maximizing a monotone submodular function subject to a matroid
  constraint.
\newblock \emph{SIAM Journal on Computing}, 40\penalty0 (6):\penalty0
  1740--1766, 2011.

\bibitem[Chekuri et~al.(2014)Chekuri, Vondr{\'a}k, and
  Zenklusen]{chekuri2014submodular}
Chandra Chekuri, Jan Vondr{\'a}k, and Rico Zenklusen.
\newblock Submodular function maximization via the multilinear relaxation and
  contention resolution schemes.
\newblock \emph{SIAM Journal on Computing}, 43\penalty0 (6):\penalty0
  1831--1879, 2014.

\bibitem[Chekuri et~al.(2015)Chekuri, Jayram, and
  Vondr{\'a}k]{chekuri2015multiplicative}
Chandra Chekuri, TS~Jayram, and Jan Vondr{\'a}k.
\newblock On multiplicative weight updates for concave and submodular function
  maximization.
\newblock In \emph{Proceedings of the 2015 Conference on Innovations in
  Theoretical Computer Science}, pages 201--210. ACM, 2015.

\bibitem[Collins et~al.(2008)Collins, Globerson, Koo, Carreras, and
  Bartlett]{collins2008exponentiated}
Michael Collins, Amir Globerson, Terry Koo, Xavier Carreras, and Peter~L
  Bartlett.
\newblock Exponentiated gradient algorithms for conditional random fields and
  max-margin markov networks.
\newblock \emph{Journal of Machine Learning Research}, 9\penalty0
  (Aug):\penalty0 1775--1822, 2008.

\bibitem[Devanur and Jain(2012)]{devanur2012online}
Nikhil~R. Devanur and Kamal Jain.
\newblock Online matching with concave returns.
\newblock In \emph{Proceedings of the 44th Symposium on Theory of Computing
  Conference, {STOC} 2012, New York, NY, USA, May 19 - 22, 2012}, pages
  137--144, 2012.

\bibitem[Eghbali and Fazel(2016)]{eghbali2016designing}
Reza Eghbali and Maryam Fazel.
\newblock Designing smoothing functions for improved worst-case competitive
  ratio in online optimization.
\newblock In \emph{Advances in Neural Information Processing Systems 29: Annual
  Conference on Neural Information Processing Systems 2016, December 5-10,
  2016, Barcelona, Spain}, pages 3279--3287, 2016.

\bibitem[Ene and Nguyen(2016)]{DBLP:conf/focs/EneN16}
Alina Ene and Huy~L. Nguyen.
\newblock Constrained submodular maximization: Beyond 1/e.
\newblock In \emph{{IEEE} 57th Annual Symposium on Foundations of Computer
  Science, {FOCS} 2016, 9-11 October 2016, Hyatt Regency, New Brunswick, New
  Jersey, {USA}}, pages 248--257, 2016.

\bibitem[Feige(1998)]{feige1998threshold}
Uriel Feige.
\newblock A threshold of ln n for approximating set cover.
\newblock \emph{Journal of the ACM (JACM)}, 45\penalty0 (4):\penalty0 634--652,
  1998.

\bibitem[Feige et~al.(2011)Feige, Mirrokni, and Vondrak]{feige2011maximizing}
Uriel Feige, Vahab~S Mirrokni, and Jan Vondrak.
\newblock Maximizing non-monotone submodular functions.
\newblock \emph{SIAM Journal on Computing}, 40\penalty0 (4):\penalty0
  1133--1153, 2011.

\bibitem[Feldman et~al.(2011)Feldman, Naor, and Schwartz]{feldman2011unified}
Moran Feldman, Joseph Naor, and Roy Schwartz.
\newblock A unified continuous greedy algorithm for submodular maximization.
\newblock In \emph{{IEEE} 52nd Annual Symposium on Foundations of Computer
  Science, {FOCS} 2011, Palm Springs, CA, USA, October 22-25, 2011}, pages
  570--579, 2011.

\bibitem[Feldman et~al.(2017)Feldman, Harshaw, and Karbasi]{feldman2017greed}
Moran Feldman, Christopher Harshaw, and Amin Karbasi.
\newblock Greed is good: Near-optimal submodular maximization via greedy
  optimization.
\newblock In \emph{Proceedings of the 30th Conference on Learning Theory,
  {COLT} 2017, Amsterdam, The Netherlands, 7-10 July 2017}, pages 758--784,
  2017.

\bibitem[Frank and Wolfe(1956)]{frank1956algorithm}
Marguerite Frank and Philip Wolfe.
\newblock An algorithm for quadratic programming.
\newblock \emph{Naval Research Logistics (NRL)}, 3\penalty0 (1-2):\penalty0
  95--110, 1956.

\bibitem[Fujishige(2005)]{fujishige2005submodular}
Satoru Fujishige.
\newblock \emph{Submodular functions and optimization}, volume~58.
\newblock Elsevier, 2005.

\bibitem[Fujishige and Isotani(2011)]{fujishige2011submodular}
Satoru Fujishige and Shigueo Isotani.
\newblock A submodular function minimization algorithm based on the
  minimum-norm base.
\newblock \emph{Pacific Journal of Optimization}, 7\penalty0 (1):\penalty0
  3--17, 2011.

\bibitem[Gharan and Vondr{\'{a}}k(2011)]{gharan2011submodular}
Shayan~Oveis Gharan and Jan Vondr{\'{a}}k.
\newblock Submodular maximization by simulated annealing.
\newblock In \emph{Proceedings of the Twenty-Second Annual {ACM-SIAM} Symposium
  on Discrete Algorithms, {SODA} 2011, San Francisco, California, USA, January
  23-25, 2011}, pages 1098--1116, 2011.

\bibitem[Golovin et~al.(2014)Golovin, Krause, and Streeter]{golovin2014online}
Daniel Golovin, Andreas Krause, and Matthew Streeter.
\newblock Online submodular maximization under a matroid constraint with
  application to learning assignments.
\newblock \emph{arXiv preprint arXiv:1407.1082}, 2014.

\bibitem[Hassani et~al.(2017)Hassani, Soltanolkotabi, and
  Karbasi]{hassani2017gradient}
Hamed Hassani, Mahdi Soltanolkotabi, and Amin Karbasi.
\newblock Gradient methods for submodular maximization.
\newblock \emph{arXiv preprint arXiv:1708.03949}, 2017.

\bibitem[Haykin(2008)]{haykin2008adaptive}
Simon~S Haykin.
\newblock \emph{Adaptive filter theory}.
\newblock Pearson Education India, 2008.

\bibitem[Hazan and Kale(2012)]{DBLP:conf/icml/HazanK12}
Elad Hazan and Satyen Kale.
\newblock Projection-free online learning.
\newblock In \emph{Proceedings of the 29th International Conference on Machine
  Learning, {ICML} 2012, Edinburgh, Scotland, UK, June 26 - July 1, 2012},
  pages 1843--1850, 2012.

\bibitem[Hazan and Luo(2016)]{DBLP:conf/icml/HazanL16}
Elad Hazan and Haipeng Luo.
\newblock Variance-reduced and projection-free stochastic optimization.
\newblock In \emph{Proceedings of the 33nd International Conference on Machine
  Learning, {ICML} 2016, New York City, NY, USA, June 19-24, 2016}, pages
  1263--1271, 2016.

\bibitem[Hazan et~al.(2016)]{hazan2016introduction}
Elad Hazan et~al.
\newblock Introduction to online convex optimization.
\newblock \emph{Foundations and Trends{\textregistered} in Optimization},
  2\penalty0 (3-4):\penalty0 157--325, 2016.

\bibitem[Jaggi(2013)]{DBLP:conf/icml/Jaggi13}
Martin Jaggi.
\newblock Revisiting {Frank}-{Wolfe}: Projection-free sparse convex
  optimization.
\newblock In \emph{Proceedings of the 30th International Conference on Machine
  Learning, {ICML} 2013, Atlanta, GA, USA, 16-21 June 2013}, pages 427--435,
  2013.

\bibitem[Karimi et~al.(2017)Karimi, Lucic, Hassani, and
  Krause]{karimi2017stochastic}
Mohammad Karimi, Mario Lucic, Hamed Hassani, and Andreas Krause.
\newblock Stochastic submodular maximization: The case of coverage functions.
\newblock In \emph{Advances in Neural Information Processing Systems}, 2017.

\bibitem[Lov{\'a}sz(1983)]{lovasz1983submodular}
L{\'a}szl{\'o} Lov{\'a}sz.
\newblock Submodular functions and convexity.
\newblock In \emph{Mathematical Programming The State of the Art}, pages
  235--257. Springer, 1983.

\bibitem[Mehta et~al.(2007)Mehta, Saberi, Vazirani, and
  Vazirani]{mehta2007adwords}
Aranyak Mehta, Amin Saberi, Umesh~V. Vazirani, and Vijay~V. Vazirani.
\newblock Adwords and generalized online matching.
\newblock \emph{Journal of the ACM}, 54\penalty0 (5), 2007.

\bibitem[Mirzasoleiman et~al.(2016)Mirzasoleiman, Badanidiyuru, and
  Karbasi]{DBLP:conf/icml/MirzasoleimanBK16}
Baharan Mirzasoleiman, Ashwinkumar Badanidiyuru, and Amin Karbasi.
\newblock Fast constrained submodular maximization: Personalized data
  summarization.
\newblock In \emph{Proceedings of the 33nd International Conference on Machine
  Learning, {ICML} 2016, New York City, NY, USA, June 19-24, 2016}, pages
  1358--1367, 2016.

\bibitem[Mokhtari and Ribeiro(2015)]{mokhtari2015global}
Aryan Mokhtari and Alejandro Ribeiro.
\newblock Global convergence of online limited memory {BFGS}.
\newblock \emph{Journal of Machine Learning Research}, 16\penalty0
  (1):\penalty0 3151--3181, 2015.

\bibitem[Mokhtari et~al.(2017)Mokhtari, Koppel, Scutari, and
  Ribeiro]{mokhtari2017large}
Aryan Mokhtari, Alec Koppel, Gesualdo Scutari, and Alejandro Ribeiro.
\newblock Large-scale nonconvex stochastic optimization by doubly stochastic
  successive convex approximation.
\newblock In \emph{Acoustics, Speech and Signal Processing (ICASSP), 2017 IEEE
  International Conference on}, pages 4701--4705. IEEE, 2017.

\bibitem[Nemhauser and Wolsey(1981)]{nemhauser81mip}
G.~L. Nemhauser and L.~A. Wolsey.
\newblock Maximizing submodular set functions: formulations and analysis of
  algorithms.
\newblock \emph{Studies on Graphs and Discrete Programming, volume 11 of Annals
  of Discrete Mathematics}, 1981.

\bibitem[Nemhauser et~al.(1978)Nemhauser, Wolsey, and
  Fisher]{nemhauser1978analysis}
George~L Nemhauser, Laurence~A Wolsey, and Marshall~L Fisher.
\newblock An analysis of approximations for maximizing submodular set
  functions--{I}.
\newblock \emph{Mathematical Programming}, 14\penalty0 (1):\penalty0 265--294,
  1978.

\bibitem[Nemirovski and Yudin(1978)]{nemirovski1978cezari}
Arkadi Nemirovski and D~Yudin.
\newblock On {Cezari's} convergence of the steepest descent method for
  approximating saddle point of convex-concave functions.
\newblock In \emph{Soviet Math. Dokl}, volume~19, 1978.

\bibitem[Nemirovskii et~al.(1983)Nemirovskii, Yudin, and
  Dawson]{nemirovskii1983problem}
Arkadii Nemirovskii, David~Borisovich Yudin, and Edgar~Ronald Dawson.
\newblock Problem complexity and method efficiency in optimization.
\newblock 1983.

\bibitem[Reddi et~al.(2016)Reddi, Sra, P{\'o}czos, and
  Smola]{reddi2016stochastic}
Sashank~J Reddi, Suvrit Sra, Barnab{\'a}s P{\'o}czos, and Alex Smola.
\newblock Stochastic {Frank}-{Wolfe} methods for nonconvex optimization.
\newblock In \emph{54th Annual Allerton Conference on Communication, Control,
  and Computing, Allerton 2016, Monticello, IL, USA, September 27-30, 2016},
  pages 1244--1251. IEEE, 2016.

\bibitem[Ribeiro(2010)]{ribeiro2010ergodic}
Alejandro Ribeiro.
\newblock Ergodic stochastic optimization algorithms for wireless communication
  and networking.
\newblock \emph{IEEE Transactions on Signal Processing}, 58\penalty0
  (12):\penalty0 6369--6386, 2010.

\bibitem[Robbins and Monro(1951)]{robbins1951stochastic}
Herbert Robbins and Sutton Monro.
\newblock A stochastic approximation method.
\newblock \emph{The annals of mathematical statistics}, pages 400--407, 1951.

\bibitem[Ruszczy{\'n}ski(1980)]{ruszczynski1980feasible}
Andrzej Ruszczy{\'n}ski.
\newblock Feasible direction methods for stochastic programming problems.
\newblock \emph{Mathematical Programming}, 19\penalty0 (1):\penalty0 220--229,
  1980.

\bibitem[Ruszczy{\'n}ski(2008)]{ruszczynski2008merit}
Andrzej Ruszczy{\'n}ski.
\newblock A merit function approach to the subgradient method with averaging.
\newblock \emph{Optimisation Methods and Software}, 23\penalty0 (1):\penalty0
  161--172, 2008.

\bibitem[Shapiro et~al.(2009)Shapiro, Dentcheva, and
  Ruszczy{\'n}ski]{shapiro2009lectures}
Alexander Shapiro, Darinka Dentcheva, and Andrzej Ruszczy{\'n}ski.
\newblock \emph{Lectures on stochastic programming: modeling and theory}.
\newblock SIAM, 2009.

\bibitem[Solo and Kong(1994)]{solo1994adaptive}
Victor Solo and Xuan Kong.
\newblock \emph{Adaptive signal processing algorithms: stability and
  performance}.
\newblock Prentice-Hall, Inc., 1994.

\bibitem[Soma et~al.(2014)Soma, Kakimura, Inaba, and
  Kawarabayashi]{soma2014optimal}
Tasuku Soma, Naonori Kakimura, Kazuhiro Inaba, and Ken{-}ichi Kawarabayashi.
\newblock Optimal budget allocation: Theoretical guarantee and efficient
  algorithm.
\newblock In \emph{Proceedings of the 31th International Conference on Machine
  Learning, {ICML} 2014, Beijing, China, 21-26 June 2014}, pages 351--359,
  2014.

\bibitem[Staib and Jegelka(2017)]{staib2017robust}
Matthew Staib and Stefanie Jegelka.
\newblock Robust budget allocation via continuous submodular functions.
\newblock In \emph{Proceedings of the 34th International Conference on Machine
  Learning, {ICML} 2017, Sydney, NSW, Australia, 6-11 August 2017}, pages
  3230--3240, 2017.

\bibitem[Stan et~al.(2017)Stan, Zadimoghaddam, Krause, and Karbasi]{serban17}
Serban Stan, Morteza Zadimoghaddam, Andreas Krause, and Amin Karbasi.
\newblock Probabilistic submodular maximization in sub-linear time.
\newblock In \emph{Proceedings of the 34th International Conference on Machine
  Learning, {ICML} 2017, Sydney, NSW, Australia, 6-11 August 2017}, pages
  3241--3250, 2017.

\bibitem[Sviridenko et~al.(2015)Sviridenko, Vondr{\'{a}}k, and
  Ward]{sviridenko2017optimal}
Maxim Sviridenko, Jan Vondr{\'{a}}k, and Justin Ward.
\newblock Optimal approximation for submodular and supermodular optimization
  with bounded curvature.
\newblock In \emph{Proceedings of the Twenty-Sixth Annual {ACM-SIAM} Symposium
  on Discrete Algorithms, {SODA} 2015, San Diego, CA, USA, January 4-6, 2015},
  pages 1134--1148, 2015.

\bibitem[Vapnik(2013)]{vapnik2013nature}
Vladimir Vapnik.
\newblock \emph{The nature of statistical learning theory}.
\newblock Springer science \& business media, 2013.

\bibitem[Vondr{\'a}k(2007)]{vondrak2007submodularity}
Jan Vondr{\'a}k.
\newblock Submodularity in combinatorial optimization.
\newblock 2007.

\bibitem[Vondr{\'{a}}k(2008)]{vondrak2008optimal}
Jan Vondr{\'{a}}k.
\newblock Optimal approximation for the submodular welfare problem in the value
  oracle model.
\newblock In \emph{Proceedings of the 40th Annual {ACM} Symposium on Theory of
  Computing, Victoria, British Columbia, Canada, May 17-20, 2008}, pages
  67--74, 2008.

\bibitem[Vondr{\'a}k(2010)]{vondrak2010submodularity}
Jan Vondr{\'a}k.
\newblock Submodularity and curvature : The optimal algorithm (combinatorial
  optimization and discrete algorithms).
\newblock \emph{RIMS Kokyuroku Bessatsu}, B23:\penalty0 253--266, 2010.

\bibitem[Wolsey(1982)]{wolsey1982analysis}
Laurence~A Wolsey.
\newblock An analysis of the greedy algorithm for the submodular set covering
  problem.
\newblock \emph{Combinatorica}, 2\penalty0 (4):\penalty0 385--393, 1982.

\bibitem[Yang et~al.(2016)Yang, Scutari, Palomar, and
  Pesavento]{yang2016parallel}
Yang Yang, Gesualdo Scutari, Daniel~P. Palomar, and Marius Pesavento.
\newblock A parallel decomposition method for nonconvex stochastic multi-agent
  optimization problems.
\newblock \emph{IEEE Transactions on Signal Processing}, 64\penalty0
  (11):\penalty0 2949--2964, 2016.

\end{thebibliography}
\bibliographystyle{abbrvnat}

\end{document}